\documentclass[11pt, reqno]{amsart}
\usepackage{amsmath, amssymb, amsthm, graphicx,marvosym}
\usepackage{cancel}
\usepackage[nobysame]{amsrefs}[11pts]
\usepackage{lipsum}
\usepackage{mathrsfs}
\usepackage{mathtools}

\usepackage[toc,page]{appendix}
\usepackage{bbm}
\usepackage{esint}

\usepackage{color}

\newtheorem{theorem}{Theorem}[section]
\newtheorem{definition}{Definition}[section]
\newtheorem{lemma}{Lemma}[section]
\newtheorem{remark}{Remark}[section]
\newtheorem{proposition}{Proposition}[section]
\newtheorem{corollary}{Corollary}[section]

\numberwithin{equation}{section}
\numberwithin{figure}{section}
\newtheorem{assumption}{Assumption}[section]

\makeatletter \@addtoreset{equation}{section} \makeatother

\newcommand{\beq}{\begin{equation}}
\newcommand{\eeq}{\end{equation}}

\newcommand{\be}{\begin{equation}}
\newcommand{\ee}{\end{equation}}
\newcommand{\g}{\gamma}
\newcommand{\e}{\text{e}}

\begin{document}
\title[the Relativistic Euler equations]{Development of singularities in the  relativistic Euler equations}

\author{Nikolaos Athanasiou}
\address{Nikolaos Athanasiou:\,Department of Mathematics, Imperial College, London, SW7 2AZ, UK.}
\email{\tt n.athanasiou20@imperial.ac.uk}

\author{Tianrui Bayles-Rea}
\address{Tianrui Bayles-Rea:\, Mathematical Institute,
University of Oxford, Oxford,  OX2 6GG, UK.}
\email{\tt  Tianrui.Bayles-Rea@maths.ox.ac.uk}

\author{Shengguo Zhu}
\address{Shengguo Zhu:\, School of Mathematical Sciences, Shanghai Jiao Tong University, Shanghai, 200240, China; Mathematical Institute, University of Oxford, Oxford,  OX2 6GG, UK.}
\email{\tt zhushengguo@sjtu.edu.cn}

\keywords{Relativistic Euler equations, singularity formation, large data, far field vacuum}

\subjclass[2010]{35Q31, 35Q75, 35A09, 35B44, 35L67.}
\date{\today}


\begin{abstract}
The purpose of this paper is to  study the phenomenon of singularity formation in large data problems for $C^1$ solutions to the Cauchy problem of  the relativistic Euler equations.
The classical theory established
in  [P. D. Lax, J. Math. Phys. 5: 611–614, 1964] shows that, for  $2\times 2$ hyperbolic systems,
 the break-down of $C^1$ solutions
occurs in finite time if initial data contain any compression in some truly nonlinear characteristic
field under some additional conditions, which include   genuine nonlinearity and  the strict positivity of the difference between two corresponding eigenvalues. These harsh structural assumptions mean that it is highly non-trivial to apply this theory  to  archetypal systems of conservation laws, such as the (1+1)-dimensional relativistic Euler equations. Actually, in the (1+1)–dimensional spacetime setting, if the mass–energy density $\rho$ does not vanish initially at any finite point, the essential difficulty in considering the possible break-down  is coming up with a way to obtain sharp enough control on the lower bound of $\rho$. To this end, based on  introducing several key artificial  quantities and  some  elaborate analysis on  the difference of the two Riemann invariants,
we  characterized the decay of mass–energy density lower bound in time, and  ultimately made some concrete progress. On the one hand, for the $C^1$ solutions with large data and possible far field vacuum to the isentropic flow, we verified the theory obtained by P. D. Lax in 1964. On the other hand, for the $C^1$ solutions with large data and strictly positive initial mass–energy density to the  non-isentropic flow,  we exhibit a numerical value $N$, thought of as representing the strength of an initial compression, above which all initial data lead to a finite-time singularity formation. These singularities manifest as a blow-up in the gradient of certain Riemann invariants associated with corresponding  systems.
\end{abstract}
\maketitle

\bigskip

\tableofcontents

\section{Introduction}\label{sec:intro}

The relativistic Euler equations is  a system of great physical meaning, as they constitute a generalization of the classical,  compressible Euler equations that manages to capture and account for the effects of special relativity. In the (1+1)–dimensional spacetime setting, it can be given by 
\begin{equation}
\label{releul1D}
\begin{cases}
\displaystyle{
	\partial_t \left( \frac{c^4\rho+Pu^2}{c^2(c^2-u^2)}
	\right)
	+\partial_x\left(\frac{(c^2\rho +P)u}{c^2-u^2}\right)=0},\\[4mm]
\displaystyle{\partial_t\left(\frac{(c^2\rho +P)u}{c^2-u^2}\right)+\partial_x\left(\frac{(c^2\rho +P)u^2}{c^2-u^2} \right)
	+\partial_x
	P = 0,} \\[4mm]
\displaystyle{\partial_t S + u \partial_x S = 0},
\end{cases}
\end{equation}
where $x$ is the Eulerian spatial variable, and  $t\in\mathbb{R}^+$ is the time. $\rho \geq 0$ denotes the mass--energy density, $u$  the relativistic velocity,  $c>0$ a large  constant corresponding to  the speed of light, $P=P(\rho,S)$ the pressure of the fluid, and    $S$  the entropy function.




 (1.1)  
is one of the most important physical models for hyperbolic conservation laws
\beq\label{CL}
	\partial_t{\bf u}+	\partial_x{\bf f}({\bf u})=0\,,
\eeq
where
${\bf u}={\bf u}(t,x)\in
\mathbb{R}^n$ is the unknown vector and  ${\bf f}:\mathbb{R}^n\rightarrow\mathbb{R}^n$
is the nonlinear flux.
It is well-known that   if  initial data are smooth enough, then the  unique classical  solution to the Cauchy problem of \eqref{CL}  exists in  a  short  time (see \cite{evans}, \cite{kato1975cauchy}, \cite{ta1985boundary} etc.). A typical phenomenon observed in such systems is the development of singularities, i.e. shock wave, no matter how small and smooth the data are, which has been justified  in a series of seminal works by Lax \cite{lax1964development} in 1964 for general $2\times2$ systems, and by
John \cite{john1974formation}, Liu \cite{ liu1979development} and Li-Zhou-Kong \cites{ta1994weak, lizhoukong} for general $n\times n$ systems. These results  confirm that for the Cauchy problem of  general strictly hyperbolic systems, if the initial datum is a generic small smooth perturbation near a constant equilibrium, then the initial compression\footnote{An initial tendency of the characteristic curves propagating the solution to cross each other, defined rigorously through the negativity of  certain Riemann invariants' spatial derivatives.} in any truly nonlinear (not weakly linearly degenerate \cite{ta1994weak}) characteristic field develops singularity in finite time.   

Particularly,  in   \cite{lax1964development} one finds that,  for  the Cauchy problem  of  $2\times2$ genuinely nonlinear, strictly hyperbolic systems with large data,  if we additionally assume that  
\begin{equation}\label{assume00}
\textit{the \ difference \ between\ two \ corresponding\ eigenvalues\ is \ strictly \ positive},
\end{equation}
   then 
  \begin{equation}\label{conclusion11}
\begin{split}
&\textit{ the \ existence \  of \ compression \  in \ the \  initial \  data } \\
\Longleftrightarrow &\textit{ \ formation  \ of \ singularities \   in \ finite \  time.}
\end{split}
\end{equation} 
 Obviously, such harsh structural assumptions make it highly non-trivial to apply this theory  to the large data problems of  archetypal systems of conservation laws such as the 1-dimensional (1-D) compressible   Euler equations and  (1+1)-dimensional ((1+1)-D) relativistic Euler equations.

Actually, for the  1-D  isentropic  compressible   Euler equations with the pressure law $P(\rho)=\rho^\gamma$ ($\gamma\geq 1$ is a fixed constant), due to the specific formulas of eigenvalues, the structural assumption in  \eqref{assume00} can be replaced by 
\begin{equation}\label{assume11}
\textit{a \ sharp \ enough, \ time \ dependent \ lower \ bound \ on  \ the density \ function} \  \rho.
\end{equation}
In fact, the research on this issue has made some solid progress.
For purely rarefactive piecewise Lipschitz continuous solutions, Lin \cite{lin2} proved that the density has a $O(1+t)^{-1}$  order lower bound via a polygonal scheme. For general smooth solutions including
compressions, in  Chen-Pan-Zhu \cite{zhu2017formation}, via considering  the system in Lagrangian coordinates and making  use of  the  simple structure of the $P$-system, crucially the continuity equation
$$
(1/\rho)_t-u_x=0,
$$
they  found a lower bound $O(1+t)^{\frac{4}{\gamma-3}}$  for the density when $1<\gamma<3$, which has been  further extended in \cite{zhu2017formation} to
non-isentropic flow. Recently  in  Chen-Pan-Zhu \cite{zhuploylower}, they showed that  the lower bound $O(1+t)^{-1}$ obtained in \cite{lin2} actually holds for  any Lipschitz continuous solution. Based on these lower bound estimates on the density, in  \cite{zhu2017formation} and Chen-Chen-Zhu \cite{chen2019formation},  the theory obtained by Lax in 1964 has been verified in this important physical model: for any $C^1$ solution with large data to the Cauchy problem of the 1-D  isentropic  compressible   Euler equations, \eqref{conclusion11} holds.
In \cite{zhu2017formation}, $\rho(0,x)$ is assumed to be strictly positive, and in \cite{chen2019formation}, $\rho(0,x)$ is still required to be  positive but is allowed to vanish in the far field. For the  1-D  non-isentropic  compressible   Euler equations, the equations written in Lagrangian coordinates possess a linearly degenerate characteristic field. Then in order to  obtain similar singularity formation results, the authors in  \cites{chen2019formation, zhu2017formation} insisted that the compression along a genuinely nonlinear characteristic field should be strong enough.

It should be pointed out that, in 1985, Sideris published a remarkable paper on the development of singularities in three-dimensional compressible Euler equations \cite{Sideris}, proving that the classical solution  has to blow up in finite time (see also Rammaha \cite{Rammaha} for two-dimensional case) under some initial assumptions. His proof was based on studying certain averaged quantities formed out of the solution, showing that they satisfy some differential inequalities whose solutions have finite life span.    The idea is that by using the averaged quantities, one is able to avoid local analysis of solutions, which is essentially different from the characteristic arguments used in \cites{chen2019formation, zhu2017formation,Chen2013shock, lax1964development,john1974formation,liu1979development}.

For the (1+1)-D  relativistic Euler equations, some progress on the  singularity formation and global well-posedness of  solutions with large data have also been obtained.  Under the assumption that $P=\sigma^2 \rho$ for some positive constant $\sigma<c$, Smoller-Temple established in their seminal work \cite{TempleSmoller} the existence of a global BV weak solution to the  Cauchy problem, by crucially noticing that the shock curves for the (1+1)-D system   satisfy very strong geometric properties. Furthermore, Frid-Perepelitsa \cites{frid2} proved the global existence of a spatially periodic entropy solution, for any given periodic initial data of locally bounded total variation and strictly positive initial density, and later, the global $L^1$ stability of  such solutions have been established by Calvo-Colombo-Frid in \cite{frid1}. For a general pressure law $P(\rho)$ satisfying 
$
P'(\rho)>0$ and $ P''(\rho)>0$,
 Chen \cite{Chen} solved  the corresponding Riemann problem. Later, Hsu-Lin-Markino  \cite{Hsu} established the existence of global $L^\infty$ weak solutions with initial data containing the vacuum state.
 Ruan-Zhu \cite{RuanZhu} firstly  proved the global well-posedness of $C^1$ solutions with large data to the  Cauchy problem of the isentropic flow, if the initial data do not contain  compression (see Definition \ref{def1}).   For the multi-dimensional  case, some results on singularity formation along the spirit of Sideris in 1985 have been obtained   by  Guo-Tahvildar-Zadeh \cite{GTZ99},  Pan--Smoller \cite{PanSmoller}, Rendall \cite{Rendall} and so on. Some analysis on the blow-up mechanism can also be found in Christodoulou \cite{Christodoulou} and  Christodoulou-Miao \cite{ChristodoulouMiao}.  In a very recent  paper Athanasiou-Zhu \cite{nikoszhu}, via introducing  an elaborate argument on a particular ODE inequality and   some crucial artificial quantity,	they  can obtain a time-dependent lower bound $O(1+t)^{\frac{4}{\gamma-3}}$ of the mass-energy density for   the  (1+1)-D isentropic  relativistic Euler equations.  
  Ultimately, for $C^1$ solutions with large data and uniformly positive initial mass-energy density, in \cite{nikoszhu} one finds  that \eqref{conclusion11} still holds. Some other interesting  results can also be found in  Oliynyk \cites{Todd0,Todd1}.

 In this paper,  we will address the following two issues for   relativistic Euler equations:
\smallskip
\paragraph{{\rm (i)}. \em
Verification of \eqref{conclusion11} for  the (1+1)-D  isentropic  relativistic Euler equations with far-filed vacuum.}
 The theory obtained by Lax in 1964 will be verified  in this important physical model: for any $C^1$ solution  to the Cauchy problem of the (1+1)-D  isentropic     relativistic Euler equations with large data and possible far field vacuum, \eqref{conclusion11} holds. 
Compared with the initial assumption in  \cite{nikoszhu}, now one   can allow the  initial density profiles 
 $\rho_0 \in C^1(\mathbb{R})\cap L^1(\mathbb{R})$ such that  $\lim_{x\to \infty} \rho_0(x)=0$ and  $\lim_{x \to -\infty} \rho_0(x)=0$. Generally vacuum will appear in the far field under some physical requirements such as finite total mass and finite  total energy  in $\mathbb{R}$.

\smallskip
\paragraph{{\rm (ii)}. \em The development of singularities for   the (1+1)-D  non-isentropic  relativistic Euler equations with   strong compression.}
For any $C^1$ solutions to the Cauchy problem of the (1+1)-D  non-isentropic     relativistic Euler equations  with large data and strictly positive initial mass–energy density,  we exhibit a numerical value $N$, thought of as representing the strength of an initial compression, above which all initial data lead to a finite-time singularity formation.   Compared to the isentropic case, a further difficulty is encountered:  one needs to  establish the upper bounds of   the Riemann variables' $L^{\infty}$ norms in the slab of existence. It is no longer as straightforward as 
that of the isentropic case, since the constructed Riemann variables satisfy some inhomogeneous ODEs  along characteristics. 
To rectify this, an elaborate ODE argument was introduced in  our paper, which, together with the Gronwall inequality, yields the   the desired boundedness.

\smallskip
This paper is divided into four sections and two appendices.
In Section 2, we introduce some basic notations and state the main results.
	In Section 3, we verify \eqref{conclusion11} for  the (1+1)-D  isentropic  relativistic Euler equations with large data and possible  far-filed vacuum.     	In Section 4, we prove the   development of singularities for   the (1+1)-D  non-isentropic  relativistic Euler equations with large data, strictly positive initial mass–energy density and   strong compression. Finally, some properties on  the equation of state for the polytropic gas and continuous functions, and the derivation for two important quantities $h$ and $g$ defined in \eqref{tidleAlphaBeta}, are given in Appendices A, B and C respectively.

\section{Main results}
In this section, we first  introduce some basic notations and then state the main results.
\subsection{Isentropic flow with far-field vacuum}
We  consider  the Cauchy problem of the  (1+1)-D isentropic relativistic Euler equations, which, in  Eulerian coordinates, is  given by 
	\begin{equation}\label{RE}
	\begin{cases}
	\displaystyle
	\partial_t	\left( \frac{c^4 \rho + P u^2}{c^2(c^2-u^2)} \right)  + 	\partial_x\left( \frac{(c^2 \rho+ P)u}{c^2 - u^2} \right) = 0, \\[4mm]
	\displaystyle
		\partial_t\left( \frac{(c^2 \rho+ P)u}{c^2 - u^2} \right) +  	\partial_x\left( \frac{(c^2 \rho+ P) u^2}{c^2 - u^2} \right) +	\partial_x P=0,\\[4mm]
	(\rho,u)|_{t=0}=(\rho_0, u_0)(x) \quad \text{for} \quad x\in \mathbb{R}.
	\end{cases}
	\end{equation}	
Here,  we consider the pressure law  
	\be \label{P} P(\rho) = k^2 \rho^\gamma, \ee   for some  fixed  constants  $k>0$ and  $\gamma > 1$.
	
\vspace{3mm}

 For simplicity, we first define the $C^1$ solutions as follows:
	\begin{definition}\label{ 2.1-1} Let $T>0$ be some  time. The pair  $( \rho(t,x),u(t,x))$ is called a  $C^1$  solution to the  equations $\eqref{RE}_1$-$\eqref{RE}_2$  on $(0, T ) \times \mathbb{R}$  if
		\begin{equation*}
		\begin{split}
		\rho>0, \quad \rho \in C^1([0, T )\times \mathbb{R}),\quad   u  \in C^1([0, T )\times \mathbb{R}),  
		\end{split}
		\end{equation*}
		and the equations $\eqref{RE}_1$-$\eqref{RE}_2$ are satisfied in the pointwise sense on $(0,T) \times \mathbb{R}$. It is called a $C^1$ solution to the Cauchy problem \eqref{RE}  if it is a $C^1$ solution to the equations $\eqref{RE}_1$-$\eqref{RE}_2$ on $(0, T ) \times \mathbb{R} $ and admits the initial data  $\eqref{RE}_3$ continuously. \end{definition}

\subsubsection{Basic setup}	\label{2.1.1}

  We proceed with a rudimentary analysis of the Cauchy problem \eqref{RE}. First, the two  eigenvalues $\lambda_1$ and $\lambda_2$ of equations   $\eqref{RE}_1$-$\eqref{RE}_2$  can be given by
\begin{equation}
    \label{lambda12}
    \lambda_1 = \frac{u- \sqrt{P'}}{1- \frac{u\sqrt{P'}}{c^2}}, \quad  \lambda_2 = \frac{u+ \sqrt{P'}}{1+ \frac{u\sqrt{P'}}{c^2}}. 	
\end{equation}
	We denote the directional derivatives as
		\begin{equation*} \partial_-= \partial_t + \lambda_1 \partial_x, \quad \partial_+ = \partial_t +\lambda_2 \partial_x, 	\end{equation*}
	along two characteristic directions
	\begin{equation*}
	\frac{dx^1}{dt}=\lambda_1, \quad  \frac{dx^2}{dt}=\lambda_2,
		\end{equation*}
	respectively. Introduce the corresponding Riemann invariants
	\begin{equation}\label{wz}
	z = \frac{c}{2}\hspace{.5mm} \text{ln}\Big(\frac{c+u}{c-u}\Big) + \int_0^\rho \frac{\sqrt{P'(\sigma)}}{\sigma + \frac{P(\sigma)}{c^2}} \text{d}\sigma, \quad 
	w=  \frac{c}{2}\hspace{.5mm} \text{ln}\Big(\frac{c+u}{c-u}\Big) - \int_0^\rho \frac{\sqrt{P'(\sigma)}}{\sigma + \frac{P(\sigma)}{c^2}} \text{d}\sigma.
		\end{equation}
	Via  direct calculations, one can verify that $w$ and $z$ satisfy
	\be  \label{REODE}    \partial_+z = 0 \quad \text{and} \quad \partial_-w = 0. \ee
	
	\par   Let $h_1$ and $h_2$ be functions satisfying 
	
	\be    \label{h1h2def}  \partial_z h_{1} = \frac{\partial_z\lambda_{1}}{\lambda_1 - \lambda_2}, \quad  \partial_w h_{2} = \frac{\partial_w\lambda_{2}}{\lambda_2 - \lambda_1}. \ee  Define $\alpha = \partial_x z$, $\beta = \partial_x w$ and introduce
	
	\be \label{zwode} \xi= \e^{h_1}\alpha, \quad  \zeta= \e^{h_2} \beta. \ee

\par   	Finally, define the compression (C) and rarefaction (R) characters.
	\begin{definition}
		\label{def1}
		The local {$R/C$} character for a classical solution of \eqref{RE} is
		\[\begin{array}{lllll} 
		\text{Forward}& \ \  $R$\ \  \text{iff} \ \ \partial_x z>0;\quad \text{Forward} \ \ \ $C$ \ \ \text{iff} \ \ \partial_x z<0;\\
		\text{Backward}& \ \ $R$ \ \ \text{iff}\ \ \partial_x w>0;\quad \text{Backward} \ \  $C$ \ \ \text{iff} \ \ \partial_x w<0.
		\end{array}\]
	\end{definition}

\subsubsection{Main result}
For simplicity, we  give the following  assumption on the initial data.

\begin{assumption}\label{cond} 
Denote $	(w_0,  z_0)(x)=(w(0,x), z(0,x))$.  Assume  $(w_0, \ z_0) \in C^1(\mathbb{R})$, and 
		\begin{equation*}\begin{split}
		&z_0 - w_0>0,\quad 
			\sup_{x \in \mathbb{R}} z_0 - \inf_{x \in \mathbb{R}} w_0 < \frac{4 c \sqrt{\gamma}}{\gamma-1} \text{\text{Arctan}}\big( \frac{1}{\sqrt{\g}} \big),\quad  \|(z_0, w_0)\|_{C^1(\mathbb{R})}\le M_0,
		\end{split}
		\end{equation*}	
		for some  constant $M_0>0$.

	\end{assumption}
\par  Then the main result of this subsection can be stated as follows:
\begin{theorem} \label{isentropicGradientBlowUp}
If the  initial data $(\rho_0, u_0)$ satisfy Assumption  \ref{cond},  then the Cauchy problem \eqref{RE} has a unique global-in-time $C^1$ solution if and only if 
\[ \text{min} \begin{Bmatrix}
\inf_{x \in \mathbb{R}} \xi(0,x) , \hspace{1mm} \inf_{x \in \mathbb{R}} \zeta(0,x) 
\end{Bmatrix} \geq 0. \]
\end{theorem}
\begin{remark}First,   the initial assumption $\sup_{x \in \mathbb{R}} w_0 - \inf_{x \in \mathbb{R}} z_0 < \frac{4 c \sqrt{\gamma}}{\gamma-1} \text{\text{Arctan}}\big( \frac{1}{\sqrt{\g}} \big)$ is used to make sure that $|\sqrt{P'(\rho)}|<c$   holds for $C^1$ solutions of the  problem \eqref{RE} (see \eqref{Pprimeformularel}).

Second, according to the above theorem, we  see that  the  definition of compression and rarefaction characters
gives a clear description on the singularity formation.

\end{remark}

\subsection{Non-isentropic flow with strictly positive mass–energy density}
Consider the Cauchy problem of the  non-isentropic system  \eqref{releul1D} with the following initial data
\begin{equation}\label{non-initial}
(\rho,u,S)|_{t=0}=(\rho_0, u_0,S_0)(x) \quad \text{for} \quad x\in \mathbb{R}.
	\end{equation}

We first consider the constitutive relation.
Denote $e$  the specific internal energy, n the rest-mass density, and  $\theta$ is the absolute  temperature. Usually, one has 
\begin{equation}
    \label{nrho}
    \rho = n\left(1+\frac{e}{c^2}\right),
\end{equation}
which can be seen in \cites{geng2011special, geng2009, lidaqianBook,relativisticShockTube, jingchen1997}.
For ideal gas, it holds that
\begin{equation}
    \label{P1}
    P = R\theta n,
\end{equation}
where $R>0$ is the gas constant.
For the  polytropic gas, one still has 
	\begin{equation*}
    e = C_v \theta,
\end{equation*}
for the specific heat at constant volume $C_v = \frac{R}{\gamma -1}$ with $\gamma > 1$, which, along with  \eqref{P1}, yields that 
\begin{equation}
    \label{P2}
    P =R/C_v en= (\gamma-1)en.
\end{equation}

Based on the above relations and the first law of thermodynamics:
	\begin{equation*}
    \theta \text{d}S = \text{d}e + P\text{d}n^{-1},
\end{equation*}
one can obtain that 
\begin{equation}\label{formula-3}
    S-S_* =  R \log n^{-1} + \frac{R}{\gamma-1} \log(e) = \log\left(\frac{e^{C_v}}{n^{C_v(\gamma-1)}}\right),
\end{equation}
where $S^*$ is some constant. In this manuscript, without loss of generality, we assume that $S^*=0$, which, along with \eqref{formula-3}  and \eqref{P2}, yields that 
\begin{equation}
    \label{PnS}
    P = (\gamma-1) n^\gamma \exp\left(\frac{S}{C_v}\right).
\end{equation}
Moreover,  it follows  from \eqref{nrho} and \eqref{P2} that 
\begin{equation}\label{formula-5}
P=c^2(\gamma-1)(\rho-n),
\end{equation}
which, along with  \eqref{PnS}, implies that 
\begin{equation}
    \label{Prho}
    P - (\gamma-1)\left(\rho - \frac{P}{(\gamma-1)c^2}\right)^\gamma \exp\left(\frac{S}{C_v}\right) = 0.
\end{equation}

 Second,  we define the $C^1$ solutions considered in this subsection  as follows:
	\begin{definition}\label{ 2.1-non} Let $T>0$ be some  time. The triple  $( \rho(t,x),u(t,x),S(t,x))$ is called a  $C^1$  solution to the non-isentropic relativistic Euler equations on $(0, T ) \times \mathbb{R}$  if
		\begin{equation*}\begin{split}
		&\rho>0, \quad \rho \in C^1([0, T )\times \mathbb{R}),\quad  u  \in C^1([0, T )\times \mathbb{R}), \\
		  &  S \in C^1([0, T )\times \mathbb{R}), \quad \partial_{xx}S\in C([0, T )\times \mathbb{R}),
		    \end{split}
		\end{equation*}
		and the equations $\eqref{releul1D}$ are satisfied in the pointwise sense on $(0,T) \times \mathbb{R}$. It is called a $C^1$ solution to the Cauchy problem \eqref{releul1D} with \eqref{non-initial},  if it is a $C^1$ solution to the equations $\eqref{releul1D}$ on $(0, T ) \times \mathbb{R} $, and it admits the initial data  $\eqref{non-initial}$ continuously. \end{definition}

\subsubsection{Basic setup}
One can rewrite \eqref{releul1D} in the following form:
\begin{equation*}
 \partial_t v + 
    A(\rho,u,S) \partial_x v = 0,
\end{equation*}
where $v=(\rho, u,S)^\top$, and $A(\rho,u,S)$ is a $3\times 3$ matrix defined by 
	\begin{equation}\label{matrix}
    \scalebox{0.96}{$
    A(\rho,u,S) \coloneqq \begin{bmatrix} \displaystyle \frac{c^2u(c^2-\partial_\rho P)}{c^4-u^2\partial_\rho P} & 
    \displaystyle \frac{c^2(P+\rho c^2)}{c^4-u^2\partial_\rho P} & 
    \displaystyle -\frac{\partial_S P u(c^2-u^2)}{c^4-u^2\partial_\rho P} \\[3mm]
    \displaystyle \frac{c^2\partial_\rho P(c^2-u^2)^2}{(c^4-u^2\partial_\rho P)(P+\rho c^2) } & 
    \displaystyle \frac{c^2u(c^2-\partial_\rho P)}{c^4-u^2\partial_\rho P} & 
    \displaystyle \frac{c^2\partial_S P(c^2-u^2)^2}{(c^4-u^2\partial_\rho P)(P+\rho c^2)} \\[3mm] 0 & 0 & u \end{bmatrix}.
    $}
\end{equation}
It follows from direct calculations that the above system  has eigenvalues
	\begin{equation}\label{3tezheng}
\lambda_1 = u,\  \
\lambda_2 = \displaystyle \frac{c^2(u+\sqrt{\partial_\rho P})}{c^2+u\sqrt{\partial_\rho P}}, \ \
\lambda_3 = \displaystyle \frac{c^2(u-\sqrt{\partial_\rho P})}{c^2-u\sqrt{\partial_\rho P}}.
\end{equation}

Comparing with the isentropic system, one of significant differences for the non-isentropic  system is the disappearance of Riemann invariances, in fact, the Riemann variables are:
\begin{equation*}
w = \displaystyle{\frac{c}{2}\ln\left(\frac{c+u}{c-u}\right)- \int_{0}^{\rho} \frac{\sqrt{\partial_\sigma P(\sigma, S)}}{ \sigma +\frac{P(\sigma,S)}{c^2}} \text{d} \sigma},\ \  z={\frac{c}{2}\ln\left(\frac{c+u}{c-u}\right)+ \int_{0}^{\rho} \frac{\sqrt{\partial_\sigma P(\sigma, S)}}{ \sigma +\frac{P(\sigma,S)}{c^2}} \text{d} \sigma},
\end{equation*}
which vary along characteristics as shown in \eqref{riPDEs}. The functions below are all expressed in the coordinate $(w(\rho,u, S), z(\rho,u, S), \hat{S} (\rho,u, S))$, with $\hat{S} (\rho,u, S) = S$. Denote 
	\begin{equation}\label{Lambda}
    \begin{split}
    \Lambda\left(\frac{z-w}{2},\hat{S}\right) \coloneqq \partial_\rho P \left( \mathcal{J}^{-1}\left(\frac{z-w}{2},\hat{S}\right) \right),
    \end{split}
\end{equation}
where $\mathcal{J}^{-1}$ is the inverse function of $
    \mathcal{J}(\rho,S) = \left( \frac{z-w}{2}(\rho,S),  \hat{S} \right)^\top$, and 
	\begin{equation}\label{HandG}
    \begin{cases}
     \displaystyle  
 H(w,z,\hat{S}) \coloneqq \frac{w+z}{c}+ \ln\left(\frac{c-\sqrt{\Lambda}}{c+\sqrt{\Lambda}}\right)=\ln\left(\frac{c+\lambda_3}{c-\lambda_3}\right), \\[4mm]
    \displaystyle
 G(w,z,\hat{S}) \coloneqq \frac{w+z}{c}+ \ln\left(\frac{c+\sqrt{\Lambda}}{c-\sqrt{\Lambda}}\right)=\ln\left(\frac{c-\lambda_2}{c+\lambda_2}\right).
    \end{cases}
\end{equation}
Then one has that
\begin{equation}
    \label{lambdas}
    \displaystyle \lambda_2 = c\left(1-\frac{2}{e^{G}+1}\right), \quad  \lambda_3 = c\left(1-\frac{2}{e^{H}+1}\right).
\end{equation}

Next, denote
	\begin{equation}\label{tidleAlphaBeta}
    \begin{split}
    &\alpha \coloneqq \partial_x w, \quad \beta \coloneqq \partial_x z,\quad 
    \widetilde{\alpha} = \alpha - a \eta,\quad 
    \widetilde{\beta} = \beta + a \eta, \quad 
    \eta \coloneqq \partial_x S,\\
    & \widetilde{n} \coloneqq \frac{cn}{\sqrt{c^2-u^2}},\quad    a = \partial_S \left(\frac{w-z}{2} \right) + \frac{c^2}{(P+\rho c^2)}\frac{\partial_S P}{\sqrt{\partial_\rho P}},\\
  \displaystyle & h \coloneqq -\ln \left(e^H+1\right) + \frac{1}{2}\ln \left(e^{G+H}-e^{2H}\right) - \int_{\epsilon}^{\frac{z-w}{2}} \frac{(c+ \sqrt{\Lambda})^2}{2c^2 \sqrt{\Lambda}} (\Xi, \hat{S}) \ \text{d}\Xi, \\[1mm]
  \displaystyle & g \coloneqq -\ln \left(e^G+1\right) + \frac{1}{2}\ln \left(e^{2G}-e^{G+H}\right) - \int_{\epsilon}^{\frac{z-w}{2}} \frac{(c- \sqrt{\Lambda})^2}{2c^2 \sqrt{\Lambda}}  (\Xi, \hat{S}) \ \text{d}\Xi,\\
   & L(z,w,\hat{S}) \coloneqq \displaystyle -\frac{1}{\widetilde{n}^2} \int_{w}^{z} e^h \widetilde{n} \partial_\vartheta a \left(\frac{\lambda_1-\lambda_2}{\lambda_3-\lambda_2}\right) (w,\vartheta,\hat{S}) \ \text{d}\vartheta, \\[3mm]
    \displaystyle & M(z,w,\hat{S}) \coloneqq \displaystyle -\frac{1}{\widetilde{n}^2} \int_{z}^{w} e^g \widetilde{n} \partial_\varsigma a \left(\frac{\lambda_1-\lambda_3}{\lambda_3-\lambda_2}\right) (\varsigma,z,\hat{S}) \ \text{d}\varsigma,
    \end{split}
\end{equation}
where $\epsilon = \frac{1}{2} \inf_x (z_0-w_0)$ and $z_0 = z(0,x)$, $w_0= w(0,x)$. 
Now we can introduce two key quantities, which play important roles in our analysis on singularity formation:
\begin{equation}
    \label{yandq}
    \displaystyle
    r \coloneqq e^h \widetilde{\alpha} - L\eta \widetilde{n}, \quad
    \displaystyle q \coloneqq e^g \widetilde{\beta} - M\eta \widetilde{n}.
\end{equation}

Lastly, define
	\begin{equation}\label{coeffientsA}
    \scalebox{0.9}{$    
    \displaystyle
    \begin{cases}
    a_0 \coloneqq \displaystyle e^h \left[\partial_3 S (a\partial_w h + \partial_{\hat{S}} h) -\partial_z \lambda_3 \frac{a \partial_2 S}{\lambda_3-\lambda_2} - \eta (a\partial_w \lambda_3 - \partial_{\hat{S}} \lambda_3) - \eta \partial_w a (\lambda_1 - \lambda_3)\right], \\[9pt]
    a_1 \coloneqq \displaystyle e^h \eta \partial_z a \left(\frac{\lambda_1-\lambda_2}{\lambda_3 - \lambda_2}\right) \left[\frac{\partial_w\widetilde{n}}{\partial_z\widetilde{n}}\partial_3 w + \frac{\partial_{\hat{S}}\widetilde{n}}{\partial_z\widetilde{n}} \partial_3 S - a \partial_2 S - \eta a(\lambda_3-\lambda_2)\right], \\[9pt]
    a_2 \coloneqq \displaystyle \widetilde{n}\eta \left[ \frac{\partial_w\widetilde{n}}{\partial_z\widetilde{n}}\partial_3 w \partial_z L + \frac{\partial_{\hat{S}}\widetilde{n}}{\partial_z\widetilde{n}}\partial_3 S \partial_z L - \partial_3 w \partial_w L - \partial_{\hat{S}} L \partial_3 S \right] - L \widetilde{n} \left(\partial_3 \eta - \frac{\eta}{\widetilde{n}}\partial_3 \widetilde{n}\right), \\[9pt] 
    a_3 \coloneqq e^{-h} (a_0 - 2 \partial_w \lambda_3 L \widetilde{n} \eta), \\[9pt]
    a_4 \coloneqq e^{-h} L \eta \widetilde{n} a_0 + a_1 + a_2 - \partial_w \lambda_3 e^{-h} L^2 \eta^2 \widetilde{n}^2,
    \end{cases}
    $}
\end{equation}
and
	\begin{equation}\label{coeffientsB}
    \displaystyle
    \scalebox{0.9}{$
    \begin{cases}
    b_0 \coloneqq \displaystyle e^g \left[\partial_2 S (-a\partial_z g + \partial_{\hat{S}} g) -\partial_w \lambda_2 \frac{a \partial_3 S}{\lambda_3-\lambda_2} + \eta (a\partial_z \lambda_2 - \partial_{\hat{S}} \lambda_2) + \eta \partial_z a (\lambda_1 - \lambda_2)\right], \\[9pt]
    b_1 \coloneqq \displaystyle e^g \eta \partial_w a \left(\frac{\lambda_1-\lambda_3}{\lambda_3 - \lambda_2}\right) \left[\frac{\partial_z\widetilde{n}}{\partial_w\widetilde{n}}\partial_2 z + \frac{\partial_{\hat{S}}\widetilde{n}}{\partial_w\widetilde{n}}\partial_2 S + a \partial_3 S - \eta a(\lambda_3-\lambda_2)\right], \\[9pt]
    b_2 \coloneqq \displaystyle \widetilde{n}\eta \left[ \frac{\partial_z\widetilde{n}}{\partial_w\widetilde{n}}\partial_2 z \partial_w M + \frac{\partial_{\hat{S}}\widetilde{n}}{\partial_w\widetilde{n}}\partial_2 S \partial_w M - \partial_2 z \partial_z M - \partial_{\hat{S}} M \partial_2 S \right] - M \widetilde{n} \left(\partial_2 \eta - \frac{\eta}{\widetilde{n}}\partial_2 \widetilde{n}\right), \\[9pt] 
    b_3 \coloneqq e^{-g} (b_0 - 2 \partial_z \lambda_2 M \widetilde{n} \eta), \\[9pt]
    b_4 \coloneqq e^{-g} M \eta \widetilde{n} b_0 + b_1 + b_2 - \partial_z \lambda_2 e^{-g} M^2 \eta^2 \widetilde{n}^2,
    \end{cases}
    $}
\end{equation}
where
\begin{equation}
    \label{partialShortHandNotations}
    \partial_1 = \partial_t - \lambda_1 \partial_x,\quad 
    \partial_2 = \partial_t - \lambda_2 \partial_x,\quad 
    \partial_3 = \partial_t - \lambda_3 \partial_x.
\end{equation}

\subsubsection{Main result}
For simplicity, we first give the following initial   assumption.
\begin{assumption}\label{REAssumpInitialData}
Assume that $(w_0,  z_0) \in C^1(\mathbb{R})$, $S_0 \in C^2(\mathbb{R}) \cap \text{BV}$, and 
\begin{equation*}
\inf_{x\in \mathbb{R}} (z_0-w_0)>0,\quad \sup_{x\in \mathbb{R}} |u_0(x)| < c, \quad  \|(z_0, w_0)\|_{C^1(\mathbb{R})}\leq M_1,\quad  \|S_0\|_{C^2(\mathbb{R})} \leq B, 
\end{equation*}
where $M_1$ and $B$ are both   positive constants,  and $BV$ denotes the space of functions with bounded total variation.
\end{assumption}

Second, for stating  our main result precisely, we give the following several propositions.

\begin{proposition}\label{PropG} 
The following function
\begin{equation*}
    \psi (\rho) \coloneqq \int_{0}^{\rho} \inf_{S \in [-B,B]} \frac{c^2 \sqrt{\partial_\sigma P(\sigma,S)}}{c^2 \sigma + P(\sigma,S)} \text{d} \sigma
\end{equation*}
 is strictly monotonically increasing, continuous, 
$$
0\leq  \psi(\rho) \leq  \frac{z-w}{2},\quad  \text{and}\quad  \psi(\rho)= 0 \iff \rho = 0.
$$
\end{proposition}

\begin{proof}
It follows from  Proposition \ref{AProp_OrderOfPetc} in Appendix A that 
\begin{equation*}
     \frac{c^2\sqrt{\partial_\rho P(\rho, S)}}{c^2 \rho +P(\rho,S)} \sim \mathcal{O}(\rho^{\frac{\gamma-3}{2}}), \quad  \text{as} \quad  \rho \to 0,
\end{equation*}
which quickly implies the desired conclusions.
\end{proof}

\begin{proposition}\label{PropL}
Let $\kappa>0$ be  a constant. Then there exists an increasing function: 
\begin{equation}
    \Psi(\kappa) \coloneqq \max_{(\rho,S) \in [0,\kappa]\times[-B,B]} \displaystyle \frac{2\sqrt{\partial_\rho P(\rho,S)}}{(z-w)(\rho,S)}.
\end{equation} 

\end{proposition}

\begin{proof}
It follows from  Proposition \ref{AProp_OrderOfPetc} in Appendix A that 
\begin{equation*}
    \begin{split}
        \sqrt{\partial_\rho P} \sim \mathcal{O}(\rho^{\frac{\gamma-1}{2}}) \quad \text{and } \quad 
        \frac{(z-w)}{2}(\rho,S) \sim \mathcal{O}(\rho^{\frac{\gamma-1}{2}}) \quad \text{as} \quad \rho \to 0,
    \end{split}
\end{equation*}
which quickly yields the desired conclusion.
\end{proof}

Denote $E(c,B) \coloneqq \sup \left\{ \rho \ | \ \sqrt{\partial_\rho P} < c, \ |S| \leq  B \right\}$.
It follows from   Proposition \ref{AProp_drhoPetc} in Appendix A and the relation between $n$ and $\rho$ that, such  $E(c,B)$ exists and is finite.

\begin{proposition} \label{ACoro_aleqK}
There exists a constant $K(c,B)  $ depending only on  $(c,B)$ such that 
\begin{equation}
    \label{a<ri1}
      K(c,B)=\sup_{\mathcal{M}} \frac{|a(\rho,S)|}{ (z-w)(\rho,S)} \quad \text{with} \quad  \mathcal{M}=\{ (\rho,S) \in [0,E(c,B))\times [-B,B]\}.
\end{equation}
\end{proposition}

\begin{proof}
By explicit computation, one can check that
\begin{equation*}
    a = \displaystyle - \int_{0}^{\rho}
    \frac{c^2 \partial_{\sigma S}P}{2\sqrt{\partial_\sigma P}(c^2 \sigma +P)} \text{d} \sigma 
    + \int_{0}^{\rho} \frac{c^2 \sqrt{\partial_\sigma P} \partial_S P}{(c^2 \sigma +P)^2} \text{d} \sigma + \frac{c^2}{(P+\rho c^2)}\frac{\partial_S P}{\sqrt{\partial_\rho P}}.
\end{equation*}
It follows from  Proposition \ref{AProp_OrderOfPetc} in Appendix A that 
\begin{equation*}
    a(\rho,S) \sim \mathcal{O}\left(\rho^{\frac{\gamma-1}{2}}\right)\quad \text{and } \quad  \frac{z-w}{2}(\rho,S) \sim \mathcal{O}\left(\rho^{\frac{\gamma-1}{2}}\right) \quad \text{as} \quad \rho \to 0,
\end{equation*}
which quickly yields the desired conclusion.

\end{proof}

Then one can  denote $ V \coloneqq  \exp\left(K(c,B)\|S_0\|_{BV}\right)$, and the following quantities:
\begin{equation}\label{kcb}
\begin{split}
    U_1 \coloneqq &\sup_x|w_0| V + K(c,B)V^2 \sup_x|z_0| \|S_0\|_{BV},\\
    U_2 \coloneqq & \sup_x|z_0| V + K(c,B)V^2 \sup_x|w_0| \|S_0\|_{BV},\\
    \max |w| \coloneqq & U_1 + K(c,B)^2V^2\|S_0\|_{BV}^2 \exp\left(K(c,B)^2V^2 \|S_0\|_{BV}^2\right) U_1,\\
    \max |z| \coloneqq & U_2 + K(c,B)^2V^2\|S_0\|_{BV}^2 \exp\left(K(c,B)^2V^2 \|S_0\|_{BV}^2\right) U_2,\\
    M_2 \coloneqq & \psi^{-1}\left(\frac{\max |z| + \max|w|}{2}\right).
    \end{split}
\end{equation}

\begin{assumption} \label{AssumptionSoundSpeed<c}

\begin{equation*}
    \Psi(M_2)\left(\frac{\max |z| + \max|w|}{2}\right) < c.
\end{equation*}
\end{assumption}

Finally, denote two  constants $N_1$ and $N_2$ depending only on $(z_0,w_0,S_0)$, $R$, $c$ and $\gamma$, by 
\begin{equation*}
    \begin{split}
        N_1 \coloneqq  \sup_{\mathcal{N}} \sqrt{\frac{2}{\partial_w \lambda_3 e^{-h}} \left( \frac{|a_3|^2}{2\partial_w \lambda_3 e^{-h}} + |a_4| \right)},\  \  
        N_2 \coloneqq  \sup_{\mathcal{N}}  \sqrt{\frac{2}{\partial_z \lambda_2 e^{-g}} \left( \frac{|b_3|^2}{2\partial_z \lambda_2 e^{-g}} + |b_4| \right)},
    \end{split}
\end{equation*}
where 
\begin{equation*}
    \mathcal{N} \coloneqq \left\{(z,w,S)\ | \ |z| \leq \max|z|, \ |w| \leq \max|w|, \ |S| \leq B \right\}.
\end{equation*}

Based on the above preparations, now we can state the  second  result as follows:
\begin{theorem}\label{GradientBlowUp}
Let Assumptions \ref{REAssumpInitialData}-\ref{AssumptionSoundSpeed<c} hold.
If there exists a point $x^* \in \mathbb{R}$ such that 
\begin{equation}
    \label{yandq<N}
    r(0,x^*) < -N_1 \quad \text{or} \quad q(0,x^*) < -N_2,
\end{equation}
then for the $C^1$ solution $(\rho,u,S)$ to the Cauchy problem \eqref{releul1D} with \eqref{non-initial}, $\partial_x \rho$ and/or $\partial_x u$ must  blow up in a finite time.
\end{theorem}

\begin{remark}
First, Assumption \ref{AssumptionSoundSpeed<c} is used to make sure that  $|\sqrt{\partial_\rho P}|<c$ holds for $C^1$ solutions of the Cauchy problem  \eqref{releul1D} with \eqref{non-initial}
(see Theorem \ref{REbound}).

Second, owing to the complex structure of \eqref{releul1D},
it is not easy to give a clear definition of R/C characters as
the isentropic case. Here we formally say that  $(\rho_0,u_0,S_0)$ contain compression, if  there exists a point $x^* \in \mathbb{R}$ such that $r(0,x^*) < 0$ or $ q(0,x^*) < 0$. That's why we state  that $N_i$ (i=1,2) represents the strength of the  initial compression.
\end{remark}

\section{Singularity formation  in  $2\times 2$ system}


This section will be devoted to proving Theorem 2.1.  It is well-known that  under Assumption \ref{cond}, 
the local-in-time existence of the unique   $C^1$ solution $(\rho, u)$   to the Cauchy problem \eqref{RE}    can be established  by the  classical characteristic  method; c.f. Theorem 4.1 on p. 66 of \cite{ta1985boundary}, or see also \cites{Daformos, evans, Lidaqian}. For simplicity, in this section, we denote  $C_g>0$  a generic  constant depending only on $(z_0,w_0)$, $\kappa$, $\gamma$ and $c$, and   independent of the time.

	

\subsection{Preliminaries} In the rest of this section,  let $(\rho,u)$ be  the $C^1$ solution satisfying Assumption \ref{cond}  of  $\eqref{RE}$ in $[0,T] \times \mathbb{R}$   for some $T>0$. Before giving the detailed proof, we  give two fundamental lemmas that will be used later.
First, we show that
			\begin{lemma}\label{lemma5.1}
				\begin{equation}\label{guga}
				|u(t,x)|<c,\quad \sqrt{P'(\rho)}(t,x)<c,\quad \rho(t,x)>0 \quad \text{for} \quad (t,x)\in [0,T] \times \mathbb{R},
				\end{equation}
and the following Riccati ODEs  hold:
				\be  \label{vbn1}\partial_- \xi= -\big(\e^{-{h}_1}{\partial_w \lambda}_{1 }\big)\xi^2, \hspace{2mm} \partial_+\zeta= -\big(\e^{-{h}_2}{ \partial_ z\lambda}_{2 }\big)\zeta^2. \ee 
				Moreover, let $x^i(t,x^i_0)$ ($i=1$, $2$) be two characteristics  starting from $(0,x^i_0)$. Then
				\be \label{vbn}
		\begin{split}
		\frac{1}{\xi(t,x^1(t,x^1_0))} =& \frac{1}{\xi(0,x^1_0) } + \int_0^t \big(\mathrm{e}^{-h_1}\partial_w \lambda_{1}\big) (\sigma,x^1(\sigma,x^1_0))\hspace{.5mm} \text{d}\sigma,\\
		\frac{1}{\zeta(t,x^2(t,x^2_0))} =& \frac{1}{\zeta(0,x^2_0) } + \int_0^t \big(\mathrm{e}^{-h_2}\partial_z \lambda_{2}\big) (\sigma,x^2(\sigma,x^2_0))\hspace{.5mm} \text{d}\sigma.   
			 	\end{split}
		    \ee 	
				
			\end{lemma}		
	\begin{proof}	
	First, the facts $	|u(t,x)|<c$ and $\sqrt{P'(\rho)}(t,x)<c$ directly follow from 
		\be \label{Pprimeformularel}
		\begin{split}
		\big\lvert \ln\Big(\frac{c+u}{c-u}\Big)\big\rvert = \big\lvert \frac{w+z}{c} \big\rvert \leq & \frac{2M_0}{c},\\
		 \sqrt{P'(\rho)} = k \sqrt{\gamma} \rho^{(\gamma-1)/2 } =& c\sqrt{\gamma}\hspace{1mm} \text{\text{Tan}}\Big( \frac{(z-w)(\g-1)}{4c \sqrt{\g}} \Big),
		    	\end{split}
		    \ee 
	where one has used \eqref{wz}-\eqref{REODE}  and  Assumption \ref{cond}.
	
	Second, the proof for the fact $\rho(t,x)>0$ can be found  on p. 997-998 of \cite{RuanZhu}.

At last, the proof for \eqref{vbn1}-\eqref{vbn} follows from direct calculations, and here we omit it.

\end{proof}

	Second, one  rewrites $\lambda_i$ and $h_i$ ($i=1,2$) as functions of $w$ and $z$.
\begin{lemma}\label{lemma5.5}
				Define the following functions
				\begin{equation}
				\label{fg}
				\begin{split}
				H_1(w,z):=& \frac{w+z}{c} + \ln\Big( \frac{1- \sqrt{\gamma}\hspace{1mm} \text{Tan}\big(\frac{(z-w)(\gamma-1)}{4 c \sqrt{\gamma}}\big)}{1+\sqrt{\gamma}\hspace{1mm}\text{Tan}\big( \frac{(z-w)(\gamma-1)}{4 c \sqrt{\gamma}}\big) } \Big), \\[1pt]
					H_2(w,z):=&  \frac{w+z}{c} + \ln\Big( \frac{1+ \sqrt{\gamma}\hspace{1mm} \text{Tan}\big( \frac{(z-w)(\gamma-1)}{4 c \sqrt{\gamma}}\big)}{1-\sqrt{\gamma}\hspace{1mm}\text{Tan}\big( \frac{(z-w)(\gamma-1)}{4 c \sqrt{\gamma}}\big) } \Big),\\
					    Y(w,z) \coloneqq & \text{Arctan}\big(\frac{k \rho^{(\gamma-1)/2}}{c} \big)= \frac{(z-w)(\gamma-1)}{4 c \sqrt{\gamma}}.
				\end{split} 
				\end{equation}
Then one has 
\begin{equation}
\label{fgmm}
				\begin{split}
\lambda_1 =& c\hspace{.5mm} \left( 1 - \frac{2}{\e^{H_1} +1} \right) , \quad  \lambda_2=  c\hspace{.5mm} \left( 1 - \frac{2}{\e^{H_2} +1} \right),\\
	h_1=& \frac{3\gamma-1}{2\gamma-2}\ln(\text{Cos}(Y)) + \frac{\gamma-3}{2\gamma-2} \ln(\text{Sin}(Y)) + \frac{z-w}{2c} \\
			& - \ln \Big(\big(1+   \e^{\frac{w+z}{c}}\big)\text{Cos}(Y) - \big(-1+ \e^{\frac{w+z}{c}}\big)\hspace{.5mm} \sqrt{\gamma} \hspace{.5mm}\text{Sin}(Y)   \Big), \\
			h_2= &\frac{3\gamma-1}{2\gamma-2}\ln(\text{Cos}(Y)) + \frac{\gamma-3}{2\gamma-2} \ln(\text{Sin}(Y)) + \frac{z-w}{2c}  \\
			& -\ln \Big(\big(\e^{\frac{2w}{c}}+   \e^{\frac{z-w}{c}}\big)\text{Cos}(Y) + \big(\e^{\frac{2w}{c}}-  \e^{\frac{z-w}{c}}\big)\hspace{.5mm}\sqrt{\gamma} \hspace{.5mm}\text{Sin}(Y)   \Big).
			\end{split}
			\end{equation}
\end{lemma}
$\lambda_1$ and $\lambda_2$ follow from direct calculations. The verification of formulas for $h_1$ and $h_2$ follows from very complex  calculations, which can be found in Appendix \ref{apphg}

			\subsection{Riccati-type ODE inequality}

			In the following proof,     we will focus only on $\xi$, and  $\zeta$ can be
    dealt with similarly.   Obviously,  by \eqref{REODE} and  \eqref{vbn}, the key point in the proof of Theorem 2.1 is to  give one good enough lower bound estimate on  $\e^{-h_1} 
\partial_z \lambda_{1}$, which will be achieved by some elaborate ODE argument.  For this purpose,  we first  observed  that 
 \begin{lemma}
 Denote $y \coloneqq \frac{k\rho^{(\g-1)/2}}{c}$.   One has 
\begin{equation}
    \label{integralBound}
    C_g^{-1}\mathcal{Y} \leq  e^{-h_1} \partial_w \lambda_{1} \leq  C_g \mathcal{Y},
\end{equation}
where 
\be
\mathcal{Y} \coloneqq \label{Yquantity} \left( \frac{k \rho^{(\gamma-1)/2}/c}{\sqrt{1 + \frac{k^2  \rho^{\gamma-1}}{c^2}}}\right)^{\frac{3-\gamma}{2\gamma-2}}\big(1+ \frac{k^2 \rho^{\gamma-1}}{c^2}\big)^{\frac{\gamma+1}{4\gamma-4}}=\big(\frac{y}{\sqrt{y^2+1}}\big)^\frac{3-\g}{2\g-2} (1+y^2)^{\frac{\g+1}{4\g-4}}.
\ee

 \end{lemma}
 \begin{proof}
First, it follows from Lemma \ref{lemma5.5}  and direct calculations that 
\be \label{explicitformula1}
\e^{-h_1} 
\partial_w \lambda_{1}= \frac{c \hspace{.5mm} \e^{-\frac{2\sqrt{\gamma} \text{Arctan}(y)}{\g-1} }(c+u)(\g+1) \big(\frac{y}{\sqrt{y^2+1}}\big)^{\frac{3-\g}{2\g-2}} (1+y^2)^{\frac{\g+1}{4\g-4}}(1-y^2)        }{2c^2 - 2u\sqrt{P'(\rho)}}.        \ee 

Second,  according to Lemma \ref{lemma5.1} and its proof,  one can obtain that 
\be 
    \notag \begin{split} 
    c \hspace{.5mm} \mathrm{e}^{-\frac{\pi \sqrt{\gamma}}{\gamma-1}} \leq c \hspace{.5mm} \e^{-\frac{2\sqrt{\gamma} \text{Arctan}(y)}{\g-1} }\leq c, \quad
    1 -\g^{-1}\leq 1-y^2\leq  & 1,\quad C_g^{-1} \leq \frac{(c+u)(\gamma+1)}{2c^2 -2 u\sqrt{P'(\rho)}}\leq C_g,\end{split}
\ee
 which, along with \eqref{explicitformula1}, yields the desired estimates.


\end{proof}

Until now, from  \eqref{integralBound}, it is  clear that for  proving  Theorem \ref{isentropicGradientBlowUp}, one needs to  give a proper time-dependent lower bound on $\mathcal{Y}$. For $\gamma \geq  3$, observe that $\mathcal{Y} > 1$, and next we need to consider the case   for $1 <\gamma < 3$.
We first consider the 	upper bounds of $\xi$ and $\zeta$.	

\begin{lemma}
\label{Q1Q2lemma} 
				\be \xi(x,t) \leq Q_1 , \hspace{2mm} \zeta(x,t) \leq Q_2, \ee
where 
\be Q_1 := max\begin{Bmatrix} 0, \sup_x \xi(x,0) \end{Bmatrix}, \hspace{2mm} Q_2 := max \begin{Bmatrix} 0, \sup_x \zeta(x,0)\end{Bmatrix}.   \ee 			
\end{lemma}

\begin{proof} 
We only consider $\xi$. First,  it follows from Lemma \ref{lemma5.5} and direct calculations that  
	\be \label{lambda1z}\partial_w \lambda_{1}= \frac{ \e^{\frac{w+z}{c}}\hspace{.5mm}(1+\g)\hspace{.5mm}\text{Cos}(2Y)\hspace{.5mm}\text{Sec}(Y)^2}{\big(1+\e^{\frac{w+z}{c}} -(\e^{\frac{w+z}{c}}-1)\hspace{.5mm}\sqrt{\g}\hspace{.5mm}\text{\text{Tan}}(Y)\big)^2}.
				\ee

Second, we show that  $\text{Cos}(2Y) > 0$, which, can be obtained  from Lemma \ref{lemma5.1},
\[Y= \frac{(z-w)(\g-1)}{4 \hspace{.5mm} c \hspace{.5mm} \sqrt{\g}}= \text{Arctan}\Big(\frac{k \rho^{(\g-1)/2}}{c}\Big)\quad \text{and} \quad  \text{Cos}(2\text{Arctan}(x)) = \frac{1-x^2}{1+x^2}.    \]
Thus,  the desired  conclusion follows  from the  Riccati ODEs established in Lemma \ref{lemma5.1}.	
\end{proof}

Now we are ready to  develop an    ODE inequality for $y$.			
\begin{lemma}\label{lemma5.7}
	Assume $1< \gamma < 3$.	One has 		
 \be \label{yode} \partial_-y \geq -C_{g}\hspace{.5mm}  y^{\frac{\g+1}{2\g-2}}\hspace{.5mm} (1+y^2)^{3/2}.\ee

\end{lemma}  

\begin{proof} 
First,  we rewrite \eqref{REODE} in the form of  the difference $(w-z)$ as:
			\be  \label{wwzz} \partial_-(z-w)  = \partial_t (z-w) +  \lambda_1 \partial_x (z-w) = (\lambda_1 - \lambda_2)\partial_x z= \frac{\lambda_1- \lambda_2}{\mathrm{e}^{h_2}}\zeta. \ee

Second,  it follows from \eqref{P}-\eqref{lambda12}  and direct  computations that 
\be \label{5.31} \frac{ \lambda_2 - \lambda_1}{\e^{h_2}}= \hspace{.5mm} \frac{4 c \hspace{.5mm} \e^{ \frac{2\hspace{.5mm} \sqrt{\gamma}\hspace{.5mm} \text{Arctan}(y) }{\gamma-1}    } \hspace{.5mm}  \hspace{.5mm} (c+u)\hspace{.5mm}  c\hspace{.5mm} \sqrt{\gamma}\hspace{.5mm}  y\hspace{.5mm}  \big(\frac{y}{\sqrt{y^2+1}}\big)^\frac{3-\gamma}{2\gamma-2}  (1+y^2)^\frac{\gamma+1}{4\gamma-4}                  }{           c^2- u \sqrt{P'(\rho)}}.                 \ee
			
Notice that, according to Lemma \ref{lemma5.1} and its proof,   one has
\begin{equation}\label{cc1}
\begin{split}
C^{-1}_g\leq &  \e^{ \frac{2\hspace{.5mm} \sqrt{\gamma}\hspace{.5mm} \text{Arctan}(y) }{\gamma-1}    }\leq C_g, \quad \text{and} \quad C^{-1}_g\leq  \frac{c+u}{c^2-u \sqrt{P'(\rho)}}\leq C_g,	
\end{split}		
\end{equation}
 which, along with \eqref{wwzz}-\eqref{5.31}, yields that  
\be \label{5.34} \partial_-(z-w) \geq - C_{g}\hspace{.5mm}   y\hspace{.5mm}  \big(\frac{y}{\sqrt{y^2+1}}\big)^\frac{3-\gamma}{2\gamma-2}  (1+y^2)^\frac{\gamma+1}{4\gamma-4} = - C_{g}\hspace{.5mm}  y^{\frac{\gamma+1}{2\gamma-2}}\hspace{.5mm}  (1+y^2)^{\frac{1}{2}}.   \ee

Due to 
$$z-w= \frac{4 \hspace{.5mm}  c \hspace{.5mm}  \sqrt{\gamma}}{\gamma-1}\text{Arctan}(y),\quad \text{and} \quad \partial_-(z-w)= C_{g} \frac{\partial_-y}{y^2+1},   $$
then one can   rewrite \eqref{5.34} as  \eqref{yode}.
\end{proof}

	Finally,  it turns out that, one can obtain the following  Riccati-type inequality for $\mathcal{Y}$:	
\begin{lemma}\label{lemma5.8}
Assume $1< \gamma < 3$. One has 
	\begin{equation}\label{day}\partial_-\mathcal{Y} \geq - C_{g}\hspace{.5mm} \mathcal{Y}^2.
	\end{equation}
	In particular, along  the characteristic $x^1(t,x_*)$ starting from $(0,x_*)$, one has
        \[ \mathcal{Y}(t, x^1(t , x_*)) \geq \frac{\mathcal{Y}(0,x_*)}{1+ C_g \mathcal{Y}(0,x_*) t}. \]
\end{lemma}

\begin{proof} 
First, it follows from \eqref{Yquantity},  \eqref{yode} and  direct calculations that 
\be\partial_- \mathcal{Y}\geq -C_{g}\hspace{.5mm}  y^\frac{6-2\g}{2\g -2} \hspace{.5mm} (1+y^2)\hspace{.5mm} \big((\g+1)y^2 +(3-\g)\big),  \ee 
which  can equivalently be rewritten as
				\be \label{finalode}\partial_-\mathcal{Y} \geq - C_{g}\hspace{.5mm}\mathcal{Y}^2 \big((\g+1)y^2 +(3-\g)\big). \ee 
				
			Second, it follows from the fact   $\sqrt{P'(\rho)}< c$  that 
		$	y< \g^{-\frac{1}{2}}$.
			 Therefore \[(\g+1) y^2 + (3-\g) \leq \frac{\gamma+1}{\gamma} + (3-\g),\]
			 which, along with \eqref{finalode}, implies \eqref{day}.

At last, by integrating along the characteristic $x^1(t , x_*)$, one obtains the desired result.	

\end{proof}

\subsection{Verification of Lax's theory in  $2\times 2$ system: Proof of Theorem \ref{isentropicGradientBlowUp}}		
Without loss of generality,  assume that there exists a point $x_*$ such that $\xi(0,x_*)<0$. We shall prove  that a singularity forms in finite time by contradiction. Indeed, assume otherwise,  $T^* = \infty$. It follows from  Lemma \ref{lemma5.1} that 
\begin{equation} \label{xieq1}
    \xi(t, x^1(t , x_*)) = \frac{\xi(0,x_*)}{1+\xi(0,x_*)\int_0^{t} \left( \e^{-h_1}\partial_w \lambda_{1} \right)(\sigma, x^1(\sigma, x_*)) \hspace{1mm}\text{d}\sigma   }.
\end{equation}
There are two cases to consider. The case $\gamma \geq 3$ and the case $1 < \gamma < 3$.

Case 1: $\gamma \geq 3$.
By \eqref{integralBound} and the fact that $\mathcal{Y} > 1$, one obtains 
\begin{equation*}
    \int_0^{t} \left( \e^{-h_1}\partial_w \lambda_{1} \right)(\sigma, x^1(\sigma, x_*)) \hspace{1mm}\text{d}\sigma  \geq C_g^{-1} t.
\end{equation*}

Case 2: $1 < \gamma < 3$.
By Lemma \ref{lemma5.8} and \eqref{integralBound}, one obtains

\[\left( \e^{-h_1}\partial_w \lambda_{1} \right) (\sigma, x^1(\sigma, x_*)) \geq C_g^{-1} \frac{\mathcal{Y}(0,x_*)}{1+ C_g \mathcal{Y}(0,x_*) \sigma}, \]
which  yields that 
\begin{equation*}
    \int_0^{t} \left( \e^{-h_1}\partial_w \lambda_{1} \right)(\sigma, x^1(\sigma, x_*)) \hspace{1mm}\text{d}\sigma  \geq C \log(1+C_g \mathcal{Y}(0,x_*) t),
\end{equation*}
for a constant $C>0$ depending only on the $C^1$ norm of the initial data and $\mathcal{Y}(0,x_*)$. 

In both cases, the contradiction now follows from the fact that $\xi(0,x_*)<0$ and the fact  \[\int_0^{t} \left( \e^{-h_1}\partial_w \lambda_{1} \right)(\sigma, x^1(\sigma, x_*)) \hspace{1mm}\text{d}\sigma \rightarrow \infty \quad \text{as} \quad t \rightarrow \infty. \] 

Second, for the other direction, assume that there exists no compression in the initial data. In that case, the boundedness of the Riemann invariants follows as usual through their constancy along characteristics, and moreover, equation \eqref{xieq1} implies that the $L^{\infty}$ norm of  $\partial_x z$ remains finite for all time. Hence, at every finite time $T$, one can invoke the local existence result, as the $C^1$ norm remains bounded. The solution, therefore, exists for all time and the result follows.

\section{Singularity formation in $3\times 3$ system}
This section will be devoted to proving Theorem 2.2. It is well-known that  under Assumptions \ref{REAssumpInitialData}-\ref{AssumptionSoundSpeed<c}, 
the local-in-time existence of the unique   $C^1$ solution $(\rho, u,S)$   to the Cauchy problem \eqref{releul1D} with \eqref{non-initial}   can be established  by the classical characteristic  method; c.f. Theorem 4.1 on p. 66 of \cite{ta1985boundary}, or  see also \cites{Daformos, evans, Lidaqian}.

First of all, we will derive the equations satisfied by the Riemann variables  $(z,w)$ from the  original system \eqref{releul1D}. It follows from direct calculations  that the matrix $A(\rho, u,S)$ given by \eqref{matrix} has right eigenvectors
\begin{equation*}
    r_1 = \begin{bmatrix} -\partial_S P \\ 0 \\ \partial_\rho P \end{bmatrix}, \quad 
    r_2 = \begin{bmatrix} P + \rho c^2 \\  \sqrt{\partial_\rho P}(c^2-u^2) \\ 0 \end{bmatrix}, \quad 
    r_3 = \begin{bmatrix} -(P + \rho c^2) \\  \sqrt{\partial_\rho P}(c^2-u^2) \\ 0 \end{bmatrix},
\end{equation*}
and left eigenvectors
\begin{equation*}
    l_1 = \begin{bmatrix} 0 \\ 0 \\ 1 \end{bmatrix}, \quad 
    l_2 = \begin{bmatrix} \sqrt{\partial_\rho P}(c^2-u^2) \\ P + \rho c^2 \\ (c^2-u^2)\frac{\partial_S P}{\sqrt{\partial_\rho P}} \end{bmatrix}, \quad 
    l_3 = \begin{bmatrix} -\sqrt{\partial_\rho P}(c^2-u^2) \\ P + \rho c^2 \\ -(c^2-u^2)\frac{\partial_S P}{\sqrt{\partial_\rho P}} \end{bmatrix}.
\end{equation*}
Moreover, the corresponding Riemann variables  satisfy  the following properties:
\begin{equation*}
\begin{split}
    &\nabla w \cdot r_2 = 0 \implies \nabla w = a l_1 + b l_3;\\
     & \nabla z \cdot r_3 = 0 \implies \nabla z = -a l_1 + b l_2,
    \end{split}
\end{equation*}
where $\nabla w=( \partial_\rho w,\partial_{u} w,\partial_S w)^\top$, $\nabla z=( \partial_\rho z,\partial_{u} z,\partial_S z)^\top$, $a$ is defined in \eqref{tidleAlphaBeta} and 
\begin{equation*}
    b = \frac{c^2}{(c^2 - u^2)(P+\rho c^2)}.
\end{equation*}
Then via direct  calculations, one can obtain that 
\begin{equation*}
\begin{split}
&\partial_t w + \lambda_3 \partial_x w = a (\lambda_3 - \lambda_1) l_1 \cdot \partial_x v,\\[4pt]
&\partial_t w + \lambda_1 \partial_x w = b (\lambda_1 - \lambda_3) l_3 \cdot \partial_x v,\\[4pt]
&\partial_t z + \lambda_2 \partial_x z = a (\lambda_1 - \lambda_2) l_1 \cdot \partial_x v,\\[4pt]
&\partial_t z + \lambda_1 \partial_x z = b (\lambda_1 - \lambda_2) l_2 \cdot \partial_x v,
\end{split}
\end{equation*}
which, along with  $l_1 \cdot \partial_x v = \partial_x S$ and  $\eqref{releul1D}_3$, yields that 
\begin{equation}
    \label{riPDEs}
    \begin{cases}
    \displaystyle
    \partial_t w + \lambda_3 \partial_x w = a (\partial_t S + \lambda_3 \partial_x S), \\[3pt]
    \displaystyle
    \partial_t z + \lambda_2 \partial_x z = -a (\partial_t S + \lambda_2 \partial_x S).
    \end{cases}
\end{equation}

Second, for simplicity, we  define some functions  and notations. Denote 
\begin{equation}
    \label{F}
    \displaystyle F(\rho,S) = \frac{z-w}{2} = \int_{0}^{\rho} \frac{c^2 \sqrt{\partial_\sigma P(\sigma,S)}}{c^2 \sigma + P(\sigma,S)} \text{d} \sigma,
\end{equation}
then  $a(\rho,S)$ can be expressed in terms of $F(\rho,S)$ and $P(\rho,S)$ as
\begin{equation}
    \label{a}
    a(\rho,S) = - \partial_S F + \frac{ \partial_S P}{\partial_\rho P} \partial_\rho F.
\end{equation}
The characteristics $x^i(t,x_0^i)$ can be defined by 
\begin{equation}
    \label{characteristics}
    \frac{dx^1}{dt} = \lambda_1, \quad \frac{dx^2}{dt} = \lambda_2, \quad \frac{dx^3}{dt} = \lambda_3,
\end{equation}
with $x_0^i = x^i(0,x_0^i)$ $(i=1,2,3)$. For simplicity, $x^i(t,x_0^i)$ will be abbreviated as $x^i(t)$.

At last, we  give the   outline of the following   proof:
\begin{enumerate}
    \item in \S 4.1, we prove   that the Riemann variables $(z,w)$ are uniformly bounded;
    \item  in \S 4.2,  we convert  the equations \eqref{riPDEs}  for the Riemann variables $(z,w)$ into some  decoupled  ODEs;
    \item  in \S 4.3,  we  show that  all the coefficients of   the decoupled  ODEs obtained above  have  uniform  bounds;
    \item  in \S 4.4,  for the case  $1 < \gamma < 3$, we  find a  sharp enough time-dependent lower bound   for  the mass-energy density $\rho$.
\end{enumerate}
Based on the above four steps, one can prove the desired conclusion on the  singularity formation.
Since now, $T$ denote  the life span  of the $C^1$ solution to  the Cauchy problem \eqref{releul1D} with \eqref{non-initial}. All the conclusions  are proved under the assumption that $t < T$.

\subsection{Uniform bounds on the Riemann variables}

First, one has 
\begin{proposition}\label{EntropyBound}
\begin{equation*}
    |S(t,x)| \leq \sup_{x \in \mathbb{R}} |S_0(x)| \quad \text{for} \quad (t,x)\in [0,T)\times \mathbb{R}.
\end{equation*}
\end{proposition}
The proof is obvious, and then  we omit it here.  

Second, define the time $T^*$ by
\begin{equation*}
    T^* \coloneqq \sup \{t \in [0,T) \mid \rho >0 \ \ \text{pointwisely}, \ \ |u| < c \  \ \text{and} \ \ \sqrt{\partial_\rho P} < c \ \ \text{uniformly} \}.
\end{equation*}
Here, for the  function $f(t,x)$,  at the time $t_0$,  $|f(t_0,x)| < c$ uniformly   means: there exists one constant  $ C_0 > 0$ independent of $(t,x) \in [0,t_0]\times \mathbb{R}$ such that 
$$c - |f(t,x)| > C_0 \quad \text{for} \quad (t,x) \in [0,t_0]\times \mathbb{R}. $$
It follows from  Assumptions \ref{REAssumpInitialData}-\ref{AssumptionSoundSpeed<c} and the  continuity of the solution that  $T^* >0$. Next we want to  prove that $(w,z)$ are uniformly bounded up to the  time $T$. We will first  show that  such  uniform bounds indeed exist for $t \in [0,T^*)$, and then prove  $T^* = T$.

\begin{proposition}\label{EigvenvalueIneq}
When $t\in [0,T^*)$, one has
\begin{equation}
    \label{eigenvalue_ineq}
    \lambda_3(t,x) < \lambda_1(t,x) < \lambda_2(t,x) \quad \text{for} \quad x\in \mathbb{R},
\end{equation}
and any two characteristics $x^i(t)$ and $x^j(t)$  $(i \neq j)$ intersect each other at most once.
\end{proposition}

\begin{proof}
First, for $t<T^*$,  \eqref{eigenvalue_ineq} follows directly  from \eqref{3tezheng}.

Second, provided $i\neq j$,   define the function $f(t)$ by 
\begin{equation*}
    f(t) \coloneqq x^i(t) - x^j(t).
\end{equation*}
One  observes that if for some $t_* \in [0,T^*)$, $f(t_*) = 0$, then $f'(t_*) \neq 0$. Actually, note that 
$f(t_*) = 0$ implies that 
$$x^i (t_*) = x^j (t_*) \coloneqq x_*.$$
Then 
 if $f'(t_*) = 0$, one can get  
\begin{equation*}
    0 = f'(t_*) = \lambda_i(t_*, x_*) - \lambda_j(t_*, x_*),
\end{equation*}
which contradicts \eqref{eigenvalue_ineq}.

Now  we  assume that  $x^i(t)$ and $x^j(t)$ intersect each other more than once in $[0,T^*)$. Then $f(t) = 0$ holds  at least for  two different  points: $\Bar{a}_1$ and $\Bar{a}_2$. One can choose  a time $\widetilde{t} \in (\Bar{a}_1,\Bar{a}_2)$ such that $f(\widetilde{t}) \neq 0$. Without loss of generality, assume   $f(\widetilde{t}) > 0$. It is clear  that  such kind of  points indeed  exist, or else $f(t) \equiv 0$ for  $t \in (\Bar{a}_1,\Bar{a}_2)$, which is  impossible. 

By the continuity of $f(t)$, there exists two points $\Bar{a}_3\in [\Bar{a}_1,\Bar{a}_2)$ and $\Bar{a}_4\in (\Bar{a}_1,\Bar{a}_2]$ such that $\widetilde{t} \in (\Bar{a}_3,\Bar{a}_4)$,  $f(t) > 0$ for all $t \in (\Bar{a}_3,\Bar{a}_4)$, and   $f(\Bar{a}_3) = f(\Bar{a}_4) = 0$, which implies that 
$$x^i (\Bar{a}_3) = x^j (\Bar{a}_3) \coloneqq \Bar{A}\quad \text{and}\quad x^i (\Bar{a}_4) = x^j (\Bar{a}_4) \coloneqq \Bar{B}.$$
As $f'(\Bar{a}_3) \neq 0$ and $f'(\Bar{a}_4) \neq 0$, one has  $f'(\Bar{a}_3) > 0$ and $f'(\Bar{a}_4) < 0$. Then 
\begin{equation*}
 \lambda_i(\Bar{a}_3,\Bar{A}) > \lambda_j(\Bar{a}_3,\Bar{A}),\quad   \lambda_i(\Bar{a}_4, \Bar{B}) < \lambda_j(\Bar{a}_4, \Bar{B}),
\end{equation*}
which contradicts \eqref{eigenvalue_ineq}. The proof of this lemma is complete.
\end{proof}

Next we show  that $(w,z)$ are uniformly bounded up to the  time $T^*$. 
\begin{lemma}\label{RERiemannInvairantBound}
\begin{equation}
    \sup_{(t,x) \in [0,T^*)\times \mathbb{R}} |w(t,x)| \leq \max|w|,\quad     \sup_{(t,x) \in [0,T^*)\times \mathbb{R}} |z(t,x)| \leq \max|z|,
\end{equation}
where  $\max|w|$ and $\max|z|$ are the  constants defined in  \eqref{kcb}, and are independent of $T^*$.
\end{lemma}
\begin{proof}
We will focus only on $w$, and  $z$ can be dealt with similarly.
For simplicity, let $\landdownint$ denote the integral  along the characteristic $x^3(t)$, and  $\landupint$  the integral  along the characteristic $x^2(t)$, i.e., for any function $f(t,x)$,  
\begin{equation*}
    \landdownint_{0}^{t} f(\sigma) \text{d}\sigma \equiv \int_{0}^{t} f(\sigma,x^3(\sigma)) \text{d}\sigma, \quad 
    \landupint_{0}^{t} f(\sigma) \text{d}\sigma \equiv \int_{0}^{t} f(\sigma,x^2(\sigma)) \text{d}\sigma.
\end{equation*}
Moreover,  $K(c,B)$ will be abbreviated as $K$. The proof is  divided into two steps: the integral  along the characteristics and the   boundedness of $(w,z)$.

Step 1: the integral  along the characteristics. There are two cases that need to be   dealt with separately:  $w(t,x^3(t)) \geq 0$ and  $w(t,x^3(t)) < 0$. First of all, when  $w(t,x^3(t)) \geq 0$,
it follows from  Proposition  \ref{ACoro_aleqK} and $\eqref{riPDEs}_1$ that
\begin{equation}
    \label{dtwineq}
    \partial_t w(t,x^3(t)) \leq K(z-w) |\partial_t S(t,x^3(t))|.
\end{equation}
Define the function $I(t)$ by 
\begin{equation}
    \label{I}
    I(t) \coloneqq \exp \left(K \int_{0}^{t} |\partial_\sigma S(\sigma,x^3(\sigma))| \text{d}\sigma \right).
\end{equation}
Multiplying \eqref{dtwineq} by $I(t)$ and integrating over $[0,t]$ along $x^3(t)$, one has 
\begin{equation}
\label{w_ineq0}
    \begin{split}
        \displaystyle w(t,x^3(t)) &\leq w_0(x_0^3)\frac{1}{I(t)} +K\frac{1}{I(t)} \landdownint_{0}^{t} z I |\partial_\sigma S| \text{d}\sigma.
    \end{split}
\end{equation}

According to the continuity of   $\partial_t S(t,x^3(t))$ and  Proposition \ref{B_SplitContinuousFunctions} in Appendix B, one can split $[0,t]$ into countably many disjoint intervals $\bigcup_{n=0}^{\infty} [t_n,t_{n+1}]$ with $t_0 = 0$ and  $t_\infty = t$, such that $\partial_t S(t,x^3(t))$ has a constant sign on each $[t_n,t_{n+1}]$.

Now we consider the characteristic $x^3(t,x_0^3)$ that starts from the point $(0,x_0^3)$. For any two different time endpoints $t_i, t_j\in (0,T^*)$ obtained above, one can introduce two  characteristics:  $x^1(t,x_i)$ that intersects $x^3(t,x_0^3)$ at $(t_i, x^3(t_i,x_0^3))$, and intersects  
$x$-axis at $(0,x_{i})$; and  $x^1(t,x_j)$ that intersects $x^3(t,x_0^3)$ at $(t_j, x^3(t_j,x_0^3))$, and intersects  
$x$-axis at $(0,x_{j})$. According to  Proposition \ref{EigvenvalueIneq}, each of the characteristics $x^1(t)$ intersects  $x^3(t)$  at most once, so one finds  that  $x_{i}$ is different from $x_j$ provided $i \neq j$. See the figure below.


\begin{center}
    \includegraphics[scale=0.35]{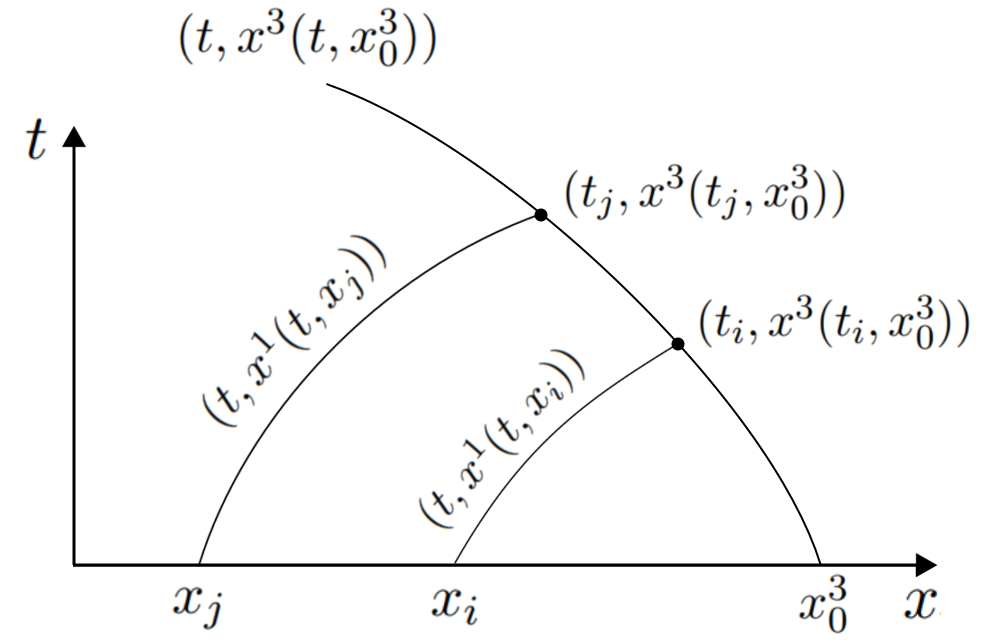}
\end{center}

Then since $\partial_t S(t,x^1(t))=0$, one can obtain that 
\begin{equation}
    \label{intdtS}
    \begin{split}
    &\int_{0}^{t} |\partial_\sigma S(\sigma,x^3(\sigma))| \text{d}\sigma  = \sum_{i = 1}^\infty \int_{t_{i-1}}^{t_i} |\partial_\sigma S(\sigma,x^3(\sigma))| \text{d}\sigma \\
    =& \sum_{i = 1}^\infty |S(t_{i-1},x^3(t_{i-1})) - S(t_{i},x^3(t_{i}))| = \sum_{i = 1}^\infty |S_0(x_{i-1}) - S_0(x_{i})|  \leq \|S_0\|_{BV},
    \end{split}
\end{equation}
which, along with \eqref{I}, yields that 
\begin{equation}
    \label{I_ineq}
    \frac{1}{I(t)}\leq 1  \leq I(t) \leq \exp(K\|S_0\|_{BV}) = V.
\end{equation}

Then it follows from  \eqref{w_ineq0} and  \eqref{I_ineq} that 
\begin{align}
    \label{wbound0}
    |w(t,x^3(t))| &\leq
    |w_0(x_{0}^{3})| V + KV \landdownint_{0}^{t} |z\partial_\sigma S| \text{d}\sigma.
\end{align}

Second,  when $w(t,x^3(t)) < 0$, it follows from  Proposition  \ref{ACoro_aleqK} and $\eqref{riPDEs}_1$ that 
\begin{equation*}
    \partial_t w(t,x^3(t)) \geq -K(z-w) |\partial_t S(t,x_3(t))|.
\end{equation*}
Hence 
\begin{equation}
    \displaystyle 0> w(t,x^3(t)) \geq w_0(x_0^3)I(t) - K I(t) \landdownint_{0}^{t} z \frac{1}{I} |\partial_\sigma S| \text{d}\sigma,
\end{equation}
which, along with  \eqref{I_ineq}, implies that \eqref{wbound0} still holds when $w(t,x^3(t)) < 0$.

 Note the integral in \eqref{wbound0} depends on $x_0^3$, hence 
\begin{align}
    \label{wbound1}
    |w(t,x)| &\leq
    \sup_{x \in \mathbb{R}}|w_0(x)| V + KV \sup_{x_0^3 \in \mathbb{R}} \landdownint_{0}^{t} |z\partial_\sigma S| \text{d}\sigma.
\end{align}

Similarly, for $z$ satisfying  $\eqref{riPDEs}_2$, one can also  obtain 
\begin{equation}
    \label{zbound1}
    |z(t,x)| \leq
    \sup_{x \in \mathbb{R}} |z_0(x)| V + KV \sup_{x_0^2 \in \mathbb{R}} \landupint_{0}^{t} |w\partial_\sigma S| \text{d}\sigma.
\end{equation}

Substituting  \eqref{zbound1} into \eqref{wbound1}, one can obtain that 
\begin{equation}
    \label{wbound4}
    \begin{split}
    |w(t,x)| \leq & \sup_{x \in \mathbb{R}}|w_0(x)| V + KV^2 \sup_{x \in \mathbb{R}} |z_0(x)| \|S_0\|_{BV} \\ &+ K^2V^2 \sup_{x_0^3 \in \mathbb{R}}  \landdownint_{0}^{t}|\partial_\sigma S| \Big(\sup_{x_0^2 \in \mathbb{R}}  \landupint_{0}^{\sigma} |\partial_\phi S||w(\phi)| \text{d} \phi \Big)\text{d} \sigma,
    \end{split}
\end{equation}
where one has used   \eqref{intdtS}.

Step 2:  the uniform bound for $|w|$. Replacing  $w(t,x)$ by $w(\tau, x^2(\tau))$ in \eqref{wbound4}, multiplying  by $|\partial_\tau S|$ and integrating  over $[0,t]$ for  $\tau$ along  $x^2(t)$,  one gets
\begin{equation}
    \label{intwdtSleq}
    \begin{split}
    \landupint_{0}^{t} |w(\tau)|& |\partial_\tau S| \text{d} \tau \leq  \sup_{x \in \mathbb{R}}|w_0(x)| V \|S_0\|_{BV} + KV^2 \sup_{x \in \mathbb{R}} |z_0(x)| \|S_0\|_{BV}^2 \\ 
    &+  K^2V^2\landupint_{0}^{t}|\partial_\tau S|  \Big(\sup_{x_0^3 \in \mathbb{R}} \landdownint_{0}^{\tau} |\partial_\sigma S| \sup_{x_0^2 \in \mathbb{R}} \Big(\landupint_{0}^{\sigma} |\partial_\phi S||w(\phi)|  \text{d} \phi \Big) \text{d} \sigma \Big) \text{d} \tau,
    \end{split}
\end{equation}
where one has used   \eqref{intdtS}.

Denote $W(t)$ by
\begin{equation*}
    W(t) \coloneqq \sup_{x_{0}^{2} \in \mathbb{R}} \landupint_{0}^{t} |w(\tau)| |\partial_\tau S| \text{d} \tau.
\end{equation*}
then one has
\begin{equation*}
\begin{split}
    W(t) \leq & V \sup_{x\in \mathbb{R}} |w_0(x)| \|S_0\|_{BV} + KV^2 \sup_{x\in \mathbb{R}} |z_0(x)| \|S_0\|_{BV}^2 \\
    &+  K^2V^2 \sup_{x_0^2 \in \mathbb{R}} \landupint_{0}^{t}|\partial_\tau S|  \Big(\sup_{x_0^3 \in \mathbb{R}} \landdownint_{0}^{\tau} |\partial_\sigma S|  W(\sigma) \text{d} \sigma \Big) \text{d} \tau,
    \end{split}
\end{equation*}
where 
\begin{equation*}
    \landdownint_{0}^{\tau} |\partial_\sigma S|  W(\sigma) \text{d} \sigma \equiv \int_{0}^{\tau} |\partial_\sigma S(\sigma,x^3(\sigma))|  W(\sigma) \text{d} \sigma.
\end{equation*}

Note that  $W(\widetilde{\sigma}) \leq W(\tau)$ for all $\widetilde{\sigma} \in [0,\tau]$, one has 
\begin{equation*}
    \begin{split}
    W(t) \leq & V \sup_{x\in \mathbb{R}} |w_0(x)| \|S_0\|_{BV} + KV^2 \sup_{x\in \mathbb{R}} |z_0(x)| \|S_0\|_{BV}^2 \\&+  K^2V^2 \sup_{x_0^2 \in \mathbb{R}} \landupint_{0}^{t} W(\tau) |\partial_\tau S| \left( \sup_{x_0^3 \in \mathbb{R}} \landdownint_{0}^{\tau} |\partial_\sigma S| \text{d} \sigma \right) \text{d} \tau\\
 \leq & V \sup_{x\in \mathbb{R}} |w_0(x)| \|S_0\|_{BV} + KV^2 \sup_{x\in \mathbb{R}} |z_0(x)| \|S_0\|_{BV}^2 \\ &+  \|S_0\|_{BV}K^2V^2
    \sup_{x_0^2 \in \mathbb{R}} \landupint_{0}^{t} W(\tau) |\partial_\tau S| \text{d} \tau.
    \end{split}
\end{equation*}
According to  the Gronwall inequality and  \eqref{intdtS}, one has
\begin{equation*}
    W(t) \leq \|S_0\|_{BV} U_1 \exp\left(K^2V^2 \|S_0\|_{BV}^2\right),
\end{equation*}
which, along with  \eqref{wbound4} and  \eqref{intdtS}, implies  the desired  bound for $w$.
\end{proof}

Based on the conclusions obtained in Proposition 4.1-4.2 and Lemma 4.1, now one has 
\begin{theorem} \label{REbound}
$T^* = T$.
\end{theorem}
\begin{proof}
Assume for contradiction that $T^* < T$. Via the continuity of our solution, in order to prove  this theorem, we just need to show that the facts  $|u| < c$ uniformly  and $\sqrt{\partial_\rho P} < c$ uniformly both   hold at $t=T^*$, and   $\rho(T^*,x) > 0$ for $(t,x) \in [0, T^*] \times \mathbb{R}$.

 First, according to Lemma 4.1, one has 
\begin{align*}
    \left|\ln\left(\frac{c+u}{c-u}\right) \right|  = 
    \frac{1}{c}|w+z| \leq \frac{max|z|+max|w|}{c}\quad \text{for} \quad  (t,x) \in [0, T^*] \times \mathbb{R},
\end{align*}
which means that  the fact  that  $|u| < c$  uniformly  still  holds at $t=T^*$.

Second, according to Proposition \ref{PropG} and Lemma \ref{RERiemannInvairantBound}, one gets
\begin{equation*}
    \psi(\rho) \leq \frac{z-w}{2} \leq \frac{\max|z|+\max|w|}{2}\quad \text{for} \quad  (t,x) \in [0, T^*] \times \mathbb{R}.
\end{equation*}
 By the properties of $\psi(\rho)$, one can take its inverse and get 
\begin{equation}\label{rhobound}
    \rho \leq \psi^{-1}\left(\frac{\max|z|+\max|w|}{2}\right),
\end{equation}
which, along with  Assumption \ref{AssumptionSoundSpeed<c}, yields  that the fact that   $\sqrt{\partial_\rho P} < c$ uniformly  still  holds at $t=T^*$.

Next we consider  the positivity of $\rho$. By $\eqref{riPDEs}$ and \eqref{partialShortHandNotations}, one has
\begin{equation}
    \label{z-w1}
    \partial_2 (z-w) = (\lambda_3 - \lambda_2)\partial_x w - a (\partial_2 S + \partial_3 S) .
\end{equation}
By \eqref{3tezheng} and Proposition \ref{EigvenvalueIneq}, one has
\begin{equation}
    \label{lambda2-lambda3}
    \lambda_2 - \lambda_3 = \frac{2c^2 \sqrt{\partial_\rho P}(c^2-u^2)}{c^4-u^2\partial_\rho P} > 0.
\end{equation}
It follows from \eqref{rhobound} and Proposition \ref{PropL}-\ref{ACoro_aleqK} that, there exists a fixed constant $\widetilde{K} > 0$ independent of $T^*$, such that
\begin{equation*}
    \sqrt{\partial_\rho P} \leq \widetilde{K} \left(z-w \right), \quad |a| \leq \widetilde{K} (z-w),
\end{equation*}
which, together with \eqref{z-w1}-\eqref{lambda2-lambda3}, yields that 
\begin{align*}
    \partial_2 (z-w) & 
    \geq -\frac{2c^2 \sqrt{\partial_\rho P}(c^2-u^2)}{c^4-u^2\partial_\rho P} \left|\partial_x w \right| - |a| |\partial_2S + \partial_3S| \\&
    \geq - \widetilde{K} \left(\frac{2c^2 (c^2-u^2)}{c^4-u^2\partial_\rho P} |\partial_x w| + |\partial_2 S + \partial_3 S| \right) (z-w).
\end{align*}
Then it follows from  the Gronwall inequality and Assumption \ref{REAssumpInitialData} that 
\begin{align*}
    z-w &\geq (z_0-w_0) \exp\left(\int_{0}^{t} - \widetilde{K} \left(\frac{2c^2 (c^2-u^2)}{c^4-u^2\partial_\rho P} |\partial_x w| + |\partial_2 S + \partial_3 S| \right) (\tau, x^2(\tau)) \text{d} \tau \right)>0.
\end{align*}

At last,  according to  Proposition \ref{PropG},  one has 
\begin{equation*}
    \rho = \psi^{-1}(z-w) > 0\quad \text{for} \quad  (t,x) \in [0, T^*] \times \mathbb{R}.
\end{equation*}
The proof of this theorem is  complete.
\end{proof}

\subsection{Non-homogeneous  ODEs of the Riemann variables} \label{RiccatiODEsSubsection}

The aim of this subsection is to  converts the PDEs \eqref{riPDEs} into two decoupled ODEs of  $(r,q)$ defined in \eqref{yandq}. For this purpose,  one needs to differentiate \eqref{riPDEs} with respect to $x$, and carefully analysis each of the terms via  the quantities and lemmas shown  below.
For $C^1$ solutions to the Cauchy problem \eqref{releul1D} with \eqref{non-initial}, the quantity $\widetilde{n} \coloneqq \frac{cn}{\sqrt{c^2-u^2}}$ introduced  in \eqref{tidleAlphaBeta} satisfies the    following equation  (see Chapter 5 in \cite{lidaqianBook}), 
\begin{equation}
    \label{tidle_n}
    \partial_t \widetilde{n} + \partial_x (u \widetilde{n}) = 0.
\end{equation}


The first lemma  is of crucial importance, which  gives some of the quantities that can be expressed in terms of the initial data, i.e. have an uniform bound. 
\begin{lemma}\label{REconserved}
Let $x^1(t)$ be the characteristic defined by $\eqref{characteristics}_1$, then
\begin{equation}
\begin{split}
    \label{dxS_0r}
    &\frac{\partial_x S(t,x^1(t))}{\widetilde{n}(t,x^1(t))} = \frac{\partial_x S_0(x_0^1)}{\widetilde{n}_0 (x_0^1)},\\
    &\partial_{xx}S(t,x^1(t)) - \frac{\partial_x S(t,x^1(t))}{\widetilde{n}(t,x^1(t))} \partial_x \widetilde{n} (t,x^1(t))\\
    =& \widetilde{n}(t,x^1(t))^2 \left(\partial_{xx}S_0 \frac{1}{\widetilde{n}_0^2} - \partial_x S_0 \frac{\partial_x \widetilde{n}_0}{\widetilde{n}_0^3} \right)(x_0^1).
    \end{split}
\end{equation}

\end{lemma}

\begin{proof}
We divided the proof into two steps: the proof for $\eqref{dxS_0r}_1$ and the  proof for $\eqref{dxS_0r}_2$.

Step 1: the proof for $\eqref{dxS_0r}_1$. Differentiating $\eqref{releul1D}_3$ with respect to $x$, then along the characteristic  $x^1(t)$, one can obtain  that
\begin{equation}\label{gugaxiaozi}
    \eta(t,x^1(t)) = \eta_0(x^1_0) \exp\left(-\int_{0}^{t}\partial_x u(s,x^1(s)) \text{d} s\right),
\end{equation}
where $\eta = \partial_x S$. 

It follows from \eqref{tidle_n} that 
\begin{equation*}
    \widetilde{n}(t,x^1(t)) = \widetilde{n}_0(x^1_0) \exp\left(-\int_{0}^{t}\partial_x u(s,x^1(s)) \text{d} s\right),
\end{equation*}
which, together with \eqref{gugaxiaozi}, implies $\eqref{dxS_0r}_1$.

Step 2: the proof for $\eqref{dxS_0r}_2$.
First,  differentiating \eqref{tidle_n} with respect to $x$, one can get 
\begin{align*}
 \partial_1(\partial_x \widetilde{n}) = -\partial_{xx}u \widetilde{n} - 2 \partial_x u \partial_x \widetilde{n}. 
\end{align*}
Define an integration factor $\mathcal{I}(t)$ by 
\begin{equation*}
    \mathcal{I}(t) \coloneqq \exp\left(2\int_{0}^{t}\partial_x u(s,x^1(s)) \text{d} s\right).
\end{equation*}
Multiplying   the equation above by $\mathcal{I}(t)$ on both sides, one has that
\begin{equation*}
    \partial_1(\partial_x \widetilde{n} \mathcal{I}) = - \partial_{xx} u \widetilde{n} \mathcal{I}.
\end{equation*}
Integrating the equation along the characteristic $x^1(t)$, one has that
\begin{equation}
    \label{dxrho}
    \partial_x \widetilde{n} (t,x^1(t))\mathcal{I}(t) = \partial_x \widetilde{n}_0(x^1_0) - \int_{0}^{t}\partial_{xx}u(s,x^1(s)) \widetilde{n}(s,x^1(s)) \mathcal{I}(s) \text{d} s. 
\end{equation}

Second,  differentiating  the equation $\eqref{releul1D}_3$  with respect to $x$ twice, one can obtain 
\begin{equation*}
    \partial_1(\partial_x \eta) = - \partial_{xx} u \eta -2 \partial_x u \partial_x \eta.
\end{equation*}
Multiplying the above equation by $\mathcal{I}(t)$, and integrating over $[0,t]$ along the characteristic $x^1(t)$, one has
\begin{equation*}
    \partial_x \eta (t,x^1(t))\mathcal{I}(t) = \partial_x \eta_0(x^1_0) - \int_{0}^{t}\partial_{xx}u(s,x^1(s)) \eta(s,x^1(s)) \mathcal{I}(s) \text{d} s,
\end{equation*}
which, along with  $\eqref{dxS_0r}_1$, yields that 
\begin{equation}\label{neimaer}
    \partial_x \eta (t,x^1(t))\mathcal{I}(t) = \partial_x \eta_0(x^1_0) - \frac{\eta_0(x^1_0)}{\widetilde{n}_0(x^1_0)} \int_{0}^{t}\partial_{xx}u(s,x^1(s)) \widetilde{n}(s,x^1(s)) \mathcal{I}(s) \text{d} s. 
\end{equation}

Then according to  \eqref{dxrho} and  \eqref{neimaer}, one has 
\begin{equation}\label{feige}
    \partial_x \eta (t,x^1(t)) -  \frac{\eta_0(x^1_0)}{\widetilde{n}_0(x^1_0)} \partial_x \widetilde{n} (t,x^1(t)) = \frac{1}{\mathcal{I}(t)} \left(\partial_x \eta_0(x^1_0) - \frac{\eta_0(x^1_0)}{\widetilde{n}_0(x^1_0)}\partial_x \widetilde{n}_0(x^1_0)\right).
\end{equation}

At last,  according to  $\eqref{dxS_0r}_1$ and the following fact 
\begin{equation*}
    \exp\left(-\int_{0}^{t}\partial_x u(s,x^1(s)) \text{d} s\right) = \frac{\widetilde{n}(t,x^1(t))}{\widetilde{n}_0(x^1_0)} \implies 
    \frac{1}{\mathcal{I}(t)} = \frac{\widetilde{n}(t,x^1(t))^2}{\widetilde{n}_0(x^1_0)^2},
\end{equation*}
one can show that  $\eqref{dxS_0r}_2$ holds.
\end{proof}

According to Lemma 4.2, one can get the following several relations.
\begin{corollary}\label{REconservedCoro}
\begin{equation}
    \label{d1theta}
    \partial_1 \eta = \frac{\eta}{\widetilde{n}} \partial_1 \widetilde{n}.
\end{equation}
\begin{equation}
    \label{d3theta}
    \partial_3 \eta - \frac{\eta}{\widetilde{n}}\partial_3 \widetilde{n} = (\lambda_3 - \lambda_1) \left(\partial_{xx}S - \frac{\partial_x S}{\widetilde{n}} \partial_x \widetilde{n}\right).
\end{equation}
\begin{equation}
    \label{d2theta}
    \partial_2 \eta - \frac{\eta}{\widetilde{n}}\partial_2 \widetilde{n} = (\lambda_2 - \lambda_1) \left(\partial_{xx}S - \frac{\partial_x S}{\widetilde{n}} \partial_x \widetilde{n}\right).
\end{equation}
\end{corollary}

\begin{proof}
First, differentiating  $\eqref{dxS_0r}_1$ in Lemma \ref{REconserved} with respect to $t$, one gets that
\begin{equation*}
    \partial_t \left(\frac{\eta(t,x^1(t))}{\widetilde{n}(t,x^1(t))}\right) = \partial_1 \left(\frac{\eta(t,x)}{\widetilde{n}(t,x)}\right) = 0, 
\end{equation*}
 which, implies \eqref{d1theta}.

Second, for \eqref{d3theta}, actually one has 
\begin{align*}
    \partial_3 \eta - \frac{\eta}{\widetilde{n}}\partial_3\widetilde{n} &=
    \partial_1 \eta + (\lambda_3-\lambda_1) \partial_{xx} S -  \frac{\eta}{\widetilde{n}}\partial_1 \widetilde{n} -  (\lambda_3-\lambda_1) \frac{\eta}{\widetilde{n}} \partial_x \widetilde{n} \\&
    = \left(\partial_1 \eta - \frac{\eta}{\widetilde{n}}\partial_1 \widetilde{n} \right) + (\lambda_3-\lambda_1)\left(\partial_{xx} S - \frac{\eta}{\widetilde{n}} \partial_x \widetilde{n} \right) \\&
    = (\lambda_3-\lambda_1) \left(\partial_{xx} S - \frac{\eta}{\widetilde{n}} \partial_x \widetilde{n} \right).
\end{align*}
Finally, \eqref{d2theta} can be obtained similarly.
\end{proof}

Next we introduce one  lemma  to transfer the coordinate system from $(\rho, u, S)$ to $(w,z,\hat{S})$.

\begin{lemma} \label{FromrhoUStoRiemannInvariant}
The vector function
\begin{equation}
    \label{Phi}
    \Phi(\rho,u,S) \coloneqq \Big( \frac{z-w}{2}(\rho,u,S), \ \frac{z+w}{c}(\rho,u,S), \ \hat{S}(\rho,u,S) \Big)^\top
\end{equation}
is a diffeomorphism, where $\hat{S} = S$. Moreover, denote $\nabla=(\partial_\rho, \partial_u,\partial_S)^\top$,  one has
\begin{equation}
    \label{JPhi}
    \nabla \Phi(\rho,u,S) = \begin{bmatrix} \partial_\rho F & 0 & \partial_S F \\ 
    0 & \frac{2c}{c^2 - u^2} & 0 \\ 0 & 0 & 1 \end{bmatrix},
\end{equation}
and 
\begin{equation}
    \label{JPhiInverse}
    (\nabla \Phi)^{-1}(\rho,u,S) = \begin{bmatrix} \frac{1}{\partial_\rho F} & 0 & - \frac{\partial_S F}{\partial_\rho F} \\ 
    0 & \frac{c^2 - u^2}{2c} & 0 \\ 0 & 0 & 1 \end{bmatrix}.
\end{equation}
\end{lemma}

The proof for the above lemma   follows from direct calculations, and here we omit it.

\begin{lemma}
Let $g$ and $h$ be the quantities  defined in \eqref{tidleAlphaBeta}. One has
\begin{equation}
    \label{dzhAnddwg}
    \displaystyle \partial_z h = \frac{\partial_z \lambda_3}{\lambda_3 - \lambda_2}, \quad
    \displaystyle \partial_w g = \frac{\partial_w \lambda_2}{\lambda_2-\lambda_3}.
\end{equation}
\end{lemma}
The verification of formulas shown in \eqref{dzhAnddwg}   follows from very complex  calculations, which can be found in Appendix \ref{apphg}.

At last, one can obtain the following decoupled ODEs of $(r,q)$.
\begin{theorem} \label{REDecoupledODEs}
Let $r$ and $q$ be defined in \eqref{yandq}. Then they satisfy the following ODEs
\begin{equation}\label{yuanshiode}
    \begin{cases}
    \partial_3 r = - \partial_w \lambda_3 e^{-h} r^2 + a_3 r + a_4,\\
    \partial_2 q = - \partial_z \lambda_2 e^{-g} q ^2 + b_3 q + b_4,
    \end{cases}
\end{equation}
where the coefficients  $a_i$ and $b_i$ are defined in \eqref{coeffientsA} and \eqref{coeffientsB} respectively.
\end{theorem}

\begin{proof}
We will focus only on $q$, and  $r$ can be dealt with similarly. 

First, differentiating  $\eqref{riPDEs}_2$ with respect to $x$, along with  $\eqref{releul1D}_3$, one can get 
\begin{align*}
    \partial_t \beta + \partial_x \lambda_2 \beta + \lambda_2 \partial_x \beta &= -\partial_x a (\lambda_2 - u)\eta - a (\partial_t \eta + \lambda_2 \partial_x \eta + \partial_x \lambda_2 \eta) \\
    &= -\partial_x a (\lambda_2 - u)\eta - a \partial_2 \eta - a \partial_x \lambda_2 \eta,
\end{align*}
which yields that 
\begin{equation*}
    \partial_2 \beta = - \partial_x \lambda_2 (\beta + a \eta) - \partial_x a (\lambda_2 - u) \eta - a \partial_2 \eta.
\end{equation*}
Then it follows from direct calculations and the definition of  $\widetilde{\beta}$ that 
\begin{align}
    \label{d2tidlebeta1}
    \partial_2 \widetilde{\beta} & = \partial_2 \beta + \partial_2(a\eta) \notag \\ &
    = - \partial_x \lambda_2 (\beta + a \eta) - \partial_x a (\lambda_2 - u) \eta - a \partial_2 \eta + \partial_2 a \eta + a \partial_2 \eta\\&
    = - \partial_x \lambda_2 \widetilde{\beta} + \eta \partial_1 a. \notag
\end{align}

Considering  the coefficients of the ODE $\eqref{yuanshiode}_2$ for $q$, namely $b_3$ and $b_4$, one observes  that they do not contain terms: derivatives of some quantities   with respect to $x$ or $t$, except for the conserved quantities listed in Corollary \ref{REconservedCoro} and $\partial_2 z$. Note by \eqref{riPDEs} and $\eqref{releul1D}_3$, $\partial_2 z$ can be expressed in terms of the conserved quantities as well. Then in order to derive $\eqref{yuanshiode}_2$ , one has to express $\partial_x \lambda_2$ and $\partial_1 a$ in terms of total derivatives with respect to $\partial_2$, as well as the terms in $b_3$ and $b_4$. For this purpose, the rest of this proof is divided into 3 steps: the reformulations of  $\partial_x \lambda_2$,    $\partial_1 a$ and  $\partial_2 \widetilde{n}$.

Step 1: the reformulation of  $\partial_x \lambda_2$. According to \eqref{tidleAlphaBeta}, one has 
\begin{align}
    \label{dxlambda2}
    \partial_x \lambda_2 &= \partial_w \lambda_2 \partial_x w + \partial_z \lambda_2 \partial_x z + \partial_{\hat{S}} \lambda_2 \partial_x \hat{S} \notag\\&
    =\partial_w \lambda_2 (\widetilde{\alpha} + a\eta) + \partial_z \lambda_2 (\widetilde{\beta} - a\eta) + \partial_{\hat{S}} \lambda_2 \partial_x \hat{S} \\&
    = \partial_w \lambda_2 \widetilde{\alpha} + \partial_z \lambda_2 \widetilde{\beta} + \eta(\partial_w \lambda_2 a -\partial_z \lambda_2 a +\partial_{\hat{S}} \lambda_2). \notag
\end{align}

Now one has to rewrite $\widetilde{\alpha}$. According to  $\eqref{riPDEs}_1$, one can get that
\begin{align*}
    & \partial_t w + \lambda_2 \partial_x w + (\lambda_3 -\lambda_2) \partial_x w  = a \partial_3 S, 
\end{align*}
which,  implies that
\begin{equation}\label{baditusida}
    \partial_x w = \frac{-\partial_2 w + a \partial_3 S}{\lambda_3-\lambda_2}.
\end{equation}
Then according to  $\eqref{tidleAlphaBeta}$ and \eqref{baditusida}, one has 
\begin{equation}
    \label{tidlealpha}
    \widetilde{\alpha} = - \frac{\partial_2 w}{\lambda_3 - \lambda_2} + \frac{a \partial_3 S}{\lambda_3 - \lambda_2} - a \eta.
\end{equation}
Substituting \eqref{tidlealpha} into \eqref{dxlambda2}, one gets that
\begin{equation}
    \label{dxlambda2again}
    \begin{split}
    \partial_x \lambda_2 = & -\partial_w \lambda_2 \frac{\partial_2 w}{\lambda_3 - \lambda_2} + \partial_w \lambda_2 \frac{a \partial_3 S}{\lambda_3 - \lambda_2} - \partial_w \lambda_2 a \eta \\ & + \partial_z \lambda_2 \widetilde{\beta} + \eta(\partial_w \lambda_2 a -\partial_z \lambda_2 a +\partial_{\hat{S}} \lambda_2).
    \end{split}
\end{equation}
According to  \eqref{dzhAnddwg} and $\eqref{riPDEs}_2$,  one has
\begin{align}
    \label{d2g}
    \partial_2 g &= \frac{\partial_w \lambda_2}{\lambda_2 - \lambda_3} \partial_2 w + \partial_z g \partial_2 z + \partial_{\hat{S}} g \partial_2 \hat{S} \notag \\&
    = \frac{\partial_w \lambda_2}{\lambda_2 - \lambda_3} \partial_2 w - a \partial_z g \partial_2 S + \partial_{\hat{S}} g \partial_2 S, 
\end{align}
which, together with \eqref{dxlambda2again}, implies that 
\begin{equation}
    \label{dxlambda2final}
    \begin{split}
    \partial_x \lambda_2 = & \partial_2 g + a \partial_z g \partial_2 S - \partial_{\hat{S}} g \partial_2 S + \partial_w \lambda_2 \frac{a \partial_3 S}{\lambda_3 - \lambda_2} \\ & - \partial_w \lambda_2 a \eta  + \partial_z \lambda_2 \widetilde{\beta}  + \eta(\partial_w \lambda_2 a -\partial_z \lambda_2 a +\partial_{\hat{S}} \lambda_2).
    \end{split}
\end{equation}
Substituting \eqref{dxlambda2final} into \eqref{d2tidlebeta1}, one gets that 
\begin{equation*}
    \begin{split}
    \partial_2 \widetilde{\beta} = & -\partial_2 g \widetilde{\beta} - (a\partial_z g - \partial_{\hat{S}} g) \partial_2 S \widetilde{\beta} \\ & - \partial_w \lambda_2 \frac{a \partial_3 S}{\lambda_3 - \lambda_2} \widetilde{\beta} - \partial_z \lambda_2 \widetilde{\beta}^2 + \widetilde{\beta}\eta(\partial_z \lambda_2 a -\partial_{\hat{S}} \lambda_2) + \eta \partial_1 a.
    \end{split}
\end{equation*}
Multiplying the above equation on  both sides by $e^g$, one has 
\begin{equation}
    \label{d2tidlebeta2}
    \partial_2 \left(\widetilde{\beta} e^g\right) = - \partial_z \lambda_2 e^g  \widetilde{\beta}^2 + \left[b_0 - e^g \eta \partial_z a (\lambda_1 - \lambda_2)\right] \widetilde{\beta} +e^g \eta \partial_1 a.
\end{equation}

Step 2: the reformulation of   $\partial_1 a$. According to \eqref{riPDEs}, one has
\begin{align}
    \label{d1a}
    \partial_1 a &= \partial_w a \partial_1 w + \partial_z a \partial_1 z + \partial_{\hat{S}} a \partial_1 \hat{S} \notag \\&
    = \partial_w a [\partial_3 w + (\lambda_1-\lambda_3) \partial_x w]+\partial_z a [\partial_2 z + (\lambda_1-\lambda_2) \partial_x z] \notag \\&
    = \partial_w a [a \partial_3 S + (\lambda_1-\lambda_3) \partial_x w]+\partial_z a [-a\partial_2 S + (\lambda_1-\lambda_2) \partial_x z] \\&
    = \partial_w a [ a (\lambda_3-\lambda_1) \eta + (\lambda_1-\lambda_3) \partial_x w]+\partial_z a [-a (\lambda_2-\lambda_1) \eta + (\lambda_1-\lambda_2) \partial_x z] \notag \\&
    = (\lambda_1-\lambda_3)\partial_w a \widetilde{\alpha} + (\lambda_1-\lambda_2) \partial_z a \widetilde{\beta}. \notag
\end{align}
It follows from direct calculations that 
\begin{equation*}
    \partial_2 \widetilde{n} = \partial_z \widetilde{n} \partial_2 z + \partial_w \widetilde{n} \partial_2 w + \partial_{\hat{S}} \widetilde{n} \partial_2 \hat{S},
\end{equation*}
which implies that 
\begin{equation}
    \label{d2wbyn}
    \partial_2 w  = \frac{\partial_2 \widetilde{n}}{\partial_w \widetilde{n}} -\frac{\partial_z \widetilde{n}}{\partial_w \widetilde{n}} \partial_2 z -\frac{\partial_{\hat{S}} \widetilde{n}}{\partial_w \widetilde{n}} \partial_2 S.
\end{equation}
Substituting \eqref{d2wbyn} into \eqref{tidlealpha}, one gets that 
\begin{equation}
    \label{tidlealpha2}
    \widetilde{\alpha} = - \frac{1}{\lambda_3 - \lambda_2} \left[\frac{\partial_2 \widetilde{n}}{\partial_w \widetilde{n}} -\frac{\partial_z \widetilde{n}}{\partial_w \widetilde{n}} \partial_2 z -\frac{\partial_{\hat{S}} \widetilde{n}}{\partial_w \widetilde{n}} \partial_2 S\right]+ \frac{a \partial_3 S}{\lambda_3 - \lambda_2} - a \eta,
\end{equation}
which, along with  \eqref{d1a}, yields that 
\begin{align*}
    e^g \eta
    \partial_1 a &= - \frac{(\lambda_1-\lambda_3) \partial_w a}{(\lambda_3 - \lambda_2) \partial_w \widetilde{n}} e^g \eta \partial_2 \widetilde{n} + b_1 +  e^g \eta (\lambda_1-\lambda_2) \partial_z a \widetilde{\beta}.
\end{align*}
Then according to the above relation and  \eqref{d2tidlebeta2}, one has 
\begin{equation}
    \label{d2tidlebeta3}
    \partial_2 \left(\widetilde{\beta} e^g\right) = -\partial_z \lambda_2 e^g \widetilde{\beta}^2 + b_0  \widetilde{\beta}  + b_1 - \frac{(\lambda_1-\lambda_3) \partial_w a}{(\lambda_3 - \lambda_2) \partial_w \widetilde{n}} e^g \eta \partial_2 \widetilde{n}.
\end{equation}

Step 3: the  reformulation of  $\partial_2 \widetilde{n}$.
Considering the  function $M(w,z,\hat{S})$  defined in \eqref{tidleAlphaBeta}, it follows from direct calculations and  \eqref{d2wbyn} that 
\begin{align}
    \label{d2Mntheta}
    \partial_2 (M\widetilde{n}\eta) &= \partial_w M \partial_2 w \widetilde{n}\eta +\partial_z M \partial_2 z \widetilde{n}\eta + \partial_{\hat{S}} M \partial_2 \hat{S} \widetilde{n}\eta + M \partial_2 \widetilde{n}\eta + M \widetilde{n} \partial_2 \eta \notag \\&
    =  \partial_w M \widetilde{n}\eta \frac{\partial_2 \widetilde{n}}{\partial_w \widetilde{n}} -
    \partial_w M \widetilde{n}\eta \frac{\partial_z \widetilde{n}}{\partial_w \widetilde{n}} \partial_2 z -
    \partial_w M \widetilde{n}\eta \frac{\partial_{\hat{S}} \widetilde{n}}{\partial_w \widetilde{n}} \partial_2 S + \partial_z M \partial_2 z \widetilde{n}\eta \\& \quad + \partial_{\hat{S}} M \partial_2 \hat{S} \widetilde{n}\eta + M \widetilde{n} \left(\partial_2 \eta - \frac{\eta}{\widetilde{n}}\partial_2 \widetilde{n}\right) + 2M\eta \partial_2 \widetilde{n} .\notag
\end{align}

One can check that, actually the function $M$ satisfies:
\begin{align*}
 \partial_w M \widetilde{n} + 2M \partial_w \widetilde{n} = - \frac{(\lambda_1-\lambda_3) \partial_w a}{(\lambda_3 - \lambda_2) } e^g.
\end{align*}
Multiplying the above equation by $\widetilde{n}$, one has that
\begin{equation*}
    \partial_w \left(M\widetilde{n}^2\right) = - \frac{(\lambda_1-\lambda_3) \partial_w a}{(\lambda_3 - \lambda_2) } e^g \widetilde{n},
\end{equation*}
which, along with  \eqref{d2Mntheta}, implies that 
\begin{equation*}
    \begin{split}
    - \frac{(\lambda_1-\lambda_3) \partial_w a}{(\lambda_3 - \lambda_2) \partial_w \widetilde{n}} e^g \eta \partial_2 \widetilde{n} = & \partial_2 (M\widetilde{n}\eta) + \partial_w M \widetilde{n}\eta \frac{\partial_z \widetilde{n}}{\partial_w \widetilde{n}} \partial_2 z +
    \partial_w M \widetilde{n}\eta \frac{\partial_{\hat{S}} \widetilde{n}}{\partial_w \widetilde{n}} \partial_2 S \\ & - \partial_z M \partial_2 z \widetilde{n}\eta  - \partial_{\hat{S}} M \partial_2 \hat{S} \widetilde{n}\eta - M \widetilde{n} \left(\partial_2 \eta - \frac{\eta}{\widetilde{n}}\partial_2 \widetilde{n}\right).
    \end{split}
\end{equation*}
Substituting the above relation into \eqref{d2tidlebeta3}, one has that
\begin{equation}
    \label{d2tidlebeta4}
    \partial_2 \left(\widetilde{\beta} e^g - M\widetilde{n}\eta\right) = -\partial_z \lambda_2 e^g \widetilde{\beta}^2 + b_0  \widetilde{\beta}  + b_1 + b_2.
\end{equation}
Notice that 
\begin{equation*}
    \widetilde{\beta} = e^{-g}q + e^{-g} M \widetilde{n}\eta \quad \text{and} \quad 
    \widetilde{\beta}^2 = e^{-2g}q^2 + 2e^{-2g}q M \widetilde{n}\eta + e^{-2g} M^2 \widetilde{n}^2\eta^2.
\end{equation*}
which, along with \eqref{d2tidlebeta4}, yields the desired ODE for $q$.

\end{proof}

\subsection{Decoupled ODE inequalities}

    The main purpose  of this section is to derive some decoupled   ODE inequalities on $(r,q)$ with good structures from \eqref{yuanshiode}, which play a key role on the estimate of the upper bound for  $(r,q)$. 

\begin{theorem} \label{ODEcompare}
Let $N_1$ and $N_2$ be the two constants introduced in  Theorem \ref{GradientBlowUp}. The following ODE inequalities hold:
\begin{equation}
    \label{yandqinequality}
    \begin{cases}
    \partial_3 r \leq - \frac{1}{2} \partial_w \lambda_3 e^{-h}(r^2 - N_{1}^2), \\
    \partial_2 q \leq - \frac{1}{2} \partial_z \lambda_2 e^{-g}(q^2 - N_{2}^2).
    \end{cases}
\end{equation}
\end{theorem}

  For simplicity, in the rest of this section, we denote  $D\geq 1$  a generic  constant depending only on $(z_0,w_0,S_0)$, $R$, $\gamma$ and $c$, and   independent of the time. Next, a definition and some additional lemmas will be given as follows.

\begin{definition}\label{Defintion4.1}

Let $f $ and $g $ be two functions. We say  $f \approx g$ if
    $D^{-1}f \leq  g  \leq D f$;
 and  we say $f \lesssim g$ if  $f \leq D g$.
\end{definition}

\begin{lemma} \label{napproxrho}
\begin{equation}
    \label{n=rho}
    \widetilde{n} \approx n \approx \rho.
\end{equation}
Furthermore, there exists a fixed positive constant $n_{\max}$ such that $0 < n \leq n_{\max}$.
\end{lemma}
\begin{proof}
It follows from  $n > 0$, \eqref{tidleAlphaBeta},  Theorem \ref{REbound} and its proof  that  $n \approx \widetilde{n}$. It remains   to show  $n \approx \rho$.  Actually, according to  Proposition \ref{AProp_nandrho}, one has
\begin{equation}\label{ngongshi}
    n = \frac{c^2 \rho}{n^{\gamma-1} \exp\left(\frac{S}{C_v}\right) + c^2},
\end{equation}
where one regards $n$ as a continuous function of $\rho$ and $S$.

According to  Theorem \ref{REbound},  there is a positive constant $n_{\max}$ such that  $0 < n \leq n_{\max}$, which, along with \eqref{ngongshi}, yields $n \approx \rho$.
\end{proof}
According to Lemma 4.5, one can obtain that 
\begin{corollary} \label{approxforPetc}
\begin{equation}
    \label{P=rho^gamma}
    P(\rho,S) \approx \rho^\gamma,\quad
    \partial_\rho P(\rho,S) \approx \rho^{\gamma-1}, \quad \partial_S P(\rho,S) \approx \rho^{\gamma},
\end{equation}
\begin{equation}
    \label{dSP=rho^gamma}
    \partial_{\rho S} P(\rho,S) \approx \rho^{\gamma-1},\quad 
    \partial_{\rho \rho} P(\rho,S) \approx \rho^{\gamma-2},\quad F(\rho,S) \approx \rho^{\frac{\gamma-1}{2}}, 
\end{equation}
\begin{equation}
    \label{dSFleq}
    \partial_{\rho} F(\rho,S) \approx \rho^{\frac{\gamma-3}{2}}, \quad
    |\partial_{S} F(\rho,S)| \lesssim \rho^{\frac{\gamma-1}{2}} + \rho^{\frac{3\gamma-3}{2}},\quad 
    |\partial_{\rho S} F(\rho,S)| \lesssim \rho^{\frac{\gamma-3}{2}} + \rho^{\frac{3\gamma-5}{2}},
\end{equation}
\begin{equation}
    \label{aleq}
    |a(\rho,S)| \lesssim \rho^{\frac{\gamma-1}{2}},\quad 
    |\partial_{\rho} a(\rho,S)| \lesssim \rho^{\frac{\gamma-3}{2}} + \rho^{\frac{3\gamma-5}{2}} + \rho^{\frac{\gamma-1}{2}}.
\end{equation}
\end{corollary}

\begin{proof}
First, the estimates on $P$, $\partial_\rho P$, $\partial_S P$, $\partial_{\rho S} P$ and  $\partial_{\rho \rho} P$ follow from \eqref{PnS}, Lemma \ref{napproxrho}, Theorem \ref{REbound},  Propositions \ref{EntropyBound}  and  \ref{AProp_drhoPetc}. 

Second, according to  the definition of $F(\rho, S)$ in  \eqref{F}, one has
\begin{equation}\label{xierwa}
    \partial_\rho F = \frac{c^2\sqrt{\partial_\rho P}}{P+\rho c^2},\quad \partial_S F = \int_{0}^{\rho} \frac{c^2\partial_{\sigma S} P}{2\sqrt{\partial_\sigma P}(P+\sigma c^2)} -\frac{c^2\partial_{S} P \sqrt{\partial_\sigma P}}{(P+\sigma c^2)^2} \text{d}\sigma,
\end{equation}
which, along with the estimates on $P$ obtained above, implies the desired estimates on  $F$, $ \partial_\rho F$ and $\partial_S F$.
The term $\partial_{\rho S}F$ can be dealt with  similarly.

Finally, it follows from  \eqref{a}  that
\begin{equation*}
    \partial_\rho a = - \partial_{\rho S}F - \frac{\partial_S P c^2 (\partial_\rho P +c^2)}{(P+\rho c^2)^2 \sqrt{\partial_\rho P}} + \frac{c^2}{P+\rho c^2 } \left(\frac{2\partial_\rho P \partial_{\rho S}P- \partial_{\rho S}P \partial_S P}{2\sqrt{\partial_\rho P}\partial_\rho P} \right).
\end{equation*}
Then one can obtain the desired estimates on $a$ and $ \partial_\rho a$  based on the estimates associated with   $P$ and $F$ obtained above.
\end{proof}

Next we consider the relation between $(g,h)$ and $\rho$.
\begin{lemma} \label{LemmaApproxhandg}
\begin{equation}
    \label{handgapprox}
    e^{-g} \approx e^{-h} \approx \frac{\sqrt{\rho}}{(\partial_\rho P)^\frac{1}{4}}.
\end{equation}
\end{lemma}

\begin{proof}
We will focus only on $e^{-g}$, and  $e^{-h}$ can be dealt with similarly. According to \eqref{Lambda}, \eqref{HandG} and \eqref{tidleAlphaBeta}, one has
\begin{equation*}
    e^{-g} = \frac{e^G +1}{\sqrt{e^{2G} - e^{G+H}}} e^{J} \quad \text{with} \quad J \coloneqq \int_{\epsilon}^{\frac{z-w}{2}} \frac{(c- \sqrt{\Lambda})^2}{2c^2 \sqrt{\Lambda}} (\Xi, \hat{S}) \text{d}\Xi.
\end{equation*}
The rest of  the proof is divided  into 2 steps: estimates on $e^{J}$ and $\frac{e^G +1}{\sqrt{e^{2G} - e^{G+H}}}$.

Step 1: the estimate on  $e^{J}$. First, it follows from \eqref{Lambda} and Lemma 4.3 that 
\begin{equation}\label{guadiaola}
\begin{split}
    \frac{1}{S-S_0}
    \int_{S_0}^S J\left(\frac{z-w}{2},\widetilde{S} \right) \text{d} \widetilde{S} &=
    \frac{1}{S-S_0}\int_{S_0}^{S}\int_{\epsilon}^{\frac{z-w}{2}} \frac{(c- \sqrt{\Lambda})^2}{2c^2 \sqrt{\Lambda}} (\Xi,\widetilde{S}) \text{d}\Xi \text{d}\widetilde{S} \\ 
    &= \frac{1}{S-S_0}\int_{S_0}^{S}\int_{\widetilde{\epsilon}(\widetilde{S})}^{\rho} \frac{(c- \sqrt{\partial_\sigma P})^2}{2c^2 \sqrt{\partial_\sigma P}} \partial_\sigma F(\sigma,\widetilde{S}) \text{d}\sigma \text{d}\widetilde{S}.
\end{split}
\end{equation}

 In the above integral, after the change of variable from $\Xi$ to $\sigma$,  the integration area  has changed from $(\widetilde{S},\Xi)\in [S_0,S]\times [\epsilon, \frac{z-w}{2}]$ to $(\widetilde{S},\sigma)\in [S_0,S]\times [\widetilde{\epsilon}(\widetilde{S}),\rho]$. By \eqref{Lambda} and \eqref{F}, one has
$$\Xi(\sigma, \widetilde{S}) \equiv \frac{z-w}{2}(\sigma,\widetilde{S}) \equiv F(\sigma,\widetilde{S}).$$
For  the function $\Upsilon \left(\rho, S \right) = \left(\frac{z-w}{2}, S \right)^\top$, one has
\begin{equation*}
    \nabla \Upsilon \left(\frac{z-w}{2}, S \right) = 
    \begin{bmatrix}
        \partial_\rho F & \partial_S F\\
        0 & 1
    \end{bmatrix} \implies \det \nabla \Upsilon = \partial_\rho F > 0.
\end{equation*}
By the Inverse Function Theorem, $\Upsilon$ is a bijection.
Taking the inverse of $\Upsilon$, one has
$$
\left(\rho, S \right)^\top = \Upsilon^{-1}\left(\frac{z-w}{2}, S \right)= (\Upsilon_1^{-1}, \Upsilon_2^{-1})^\top.
$$
In particular, when $\Xi(\sigma, \widetilde{S}) = \epsilon$, one has that $(\sigma, \widetilde{S})^\top = \Upsilon^{-1}(\epsilon, \widetilde{S})$, which implies that 
$$\sigma = \Upsilon_1^{-1}(\epsilon, \widetilde{S}) \coloneqq \widetilde{\epsilon}(\widetilde{S}) > 0.$$ 

Then according to  Proposition \ref{EntropyBound}, $\widetilde{\epsilon}(\widetilde{S})$ is  uniformly upper bounded. Now we claim  that $\widetilde{\epsilon}(\widetilde{S})$ has an uniform lower bound. Notice that for any $\ \widetilde{S} \in [-B,B]$, if 
\begin{equation}\label{aonier}
    \Xi(\widetilde{\epsilon}(\widetilde{S}), \widetilde{S}) = \epsilon, 
\end{equation}
then 
\begin{equation}\label{dengken}
  F(\widetilde{\epsilon}(\widetilde{S}), \widetilde{S}) = \epsilon.
\end{equation}
Assume for contradiction that $\widetilde{\epsilon}(\widetilde{S})$ does not have an uniform lower bound. Then $\forall \ \delta > 0$, one can find a $S^*$ such that \eqref{dengken} holds and  $\widetilde{\epsilon}(S^*) < \delta$. By Proposition \ref{PropG}, one has $F(\rho, S) \geq 0$ and $F(\rho, S) = 0$ if and only if $\rho = 0$. Then one can choose a $\delta$ small enough such that $F(\widetilde{\epsilon}(S^*), S^*) < \epsilon$, which contradicts with \eqref{dengken}.


Now taking  $S_0 \to S$, by the Lebesgue Differentiation Theorem (see \cite{evans}), one has
\begin{equation*}
\begin{split}
   & \lim_{S_0 \to S} \frac{1}{S-S_0}
    \int_{S_0}^S J\left(\frac{z-w}{2},\widetilde{S} \right) \text{d} \widetilde{S} = J\left(\frac{z-w}{2},S \right),\\[4pt]
   & \lim_{S_0 \to S} \frac{1}{S-S_0}\int_{S_0}^{S}\int_{\widetilde{\epsilon}(\widetilde{S})}^{\rho} \frac{(c- \sqrt{\partial_\sigma P})^2}{2c^2 \sqrt{\partial_\sigma P}} \partial_\sigma F(\sigma,\widetilde{S}) \text{d}\sigma \text{d}\widetilde{S} \\
   =& \int_{\widetilde{\epsilon}(S)}^{\rho} \frac{(c- \sqrt{\partial_\sigma P})^2}{2c^2 \sqrt{\partial_\sigma P}} \partial_\sigma F(\sigma,S) \text{d}\sigma,
    \end{split}
\end{equation*}
which, along with \eqref{guadiaola} and \eqref{xierwa},  implies that 
\begin{align*}
     J\left(\frac{z-w}{2},S \right) = \int_{\widetilde{\epsilon}(S)}^{\rho} \frac{(c- \sqrt{\partial_\sigma P})^2}{2c^2 \sqrt{\partial_\sigma P}} \partial_\sigma F(\sigma,S) \text{d}\sigma = \int_{\widetilde{\epsilon}(S)}^{\rho} \frac{(c- \sqrt{\partial_\sigma P})^2}{2(c^2\sigma +P)}  \text{d}\sigma .
\end{align*}

On the one hand, for the upper bound estimate, according to Theorem \ref{REbound},  one has 
\begin{equation}
    \label{I_inside_upper_bound}
    \frac{(c- \sqrt{\partial_\rho P})^2}{2(c^2\rho +P)} \leq \frac{c^2}{2(c^2\rho +P)} \leq \frac{c^2}{2c^2\rho} = \frac{1}{2\rho}.
\end{equation}

On the other hand, for the  lower bound estimate, one  claims that there exists a  constant $A>0$, such that for all $\rho > 0$, one has
\begin{equation}
    \label{I_inside_lower_bound}
    \frac{(c- \sqrt{\partial_\rho P})^2}{2(c^2\rho +P)} \geq -A\frac{\sqrt{\partial_\rho P}}{\rho} + \frac{1}{2\rho}.
\end{equation}
The above inequality holds if and only if 
\begin{equation}\label{jiyelini}
\sqrt{\partial_\rho P} + 2Ac^2 + 2A\frac{P}{\rho} \geq 2c + \frac{P}{2\rho \sqrt{\partial_\rho P}}.
\end{equation}
Actually, according to  Corollary \ref{approxforPetc} and Theorem \ref{REbound}, one has 
\begin{equation}\label{shengkelusi}
\frac{P}{\rho \sqrt{\partial_\rho P}} \approx \rho^{\frac{\gamma-1}{2}} \approx \sqrt{\partial_\rho P} < c,
\end{equation}
so  \eqref{I_inside_lower_bound}-\eqref{jiyelini} are indeed true if $A$ is large enough and depends only on the initial data and $c$. 

Then it follows from  \eqref{I_inside_upper_bound}-\eqref{I_inside_lower_bound} that
\begin{equation}
    \label{Iineq}
    -A \int_{\widetilde{\epsilon}(S)}^{\rho} \frac{\sqrt{\partial_\sigma P}}{\sigma} \text{d}\sigma + \frac{1}{2} \log\left(\frac{\rho}{\widetilde{\epsilon}(S)}\right) \leq J \leq \frac{1}{2} \log\left(\frac{\rho}{\widetilde{\epsilon}(S)}\right).
\end{equation}
Now by Corollary \ref{approxforPetc}, one has  $\frac{\sqrt{\partial_\rho P}}{\rho} \approx \rho^{\frac{\gamma-3}{2}}$, hence
\begin{equation*}
   \int_{\widetilde{\epsilon}(S)}^{\rho} \frac{\sqrt{\partial_\sigma P}}{\sigma} \text{d}\sigma \approx \rho^{\frac{\gamma-1}{2}} - \widetilde{\epsilon}(S)^{\frac{\gamma-1}{2}},
\end{equation*}
which, along with \eqref{shengkelusi}-\eqref{Iineq}, Theorem \ref{REbound} and its proof , yields that 
\begin{equation}\label{qubo}
    e^{J} \approx \sqrt{\rho}.
\end{equation}

Step 2: the estimate on $    \frac{e^G +1}{\sqrt{e^{2G} - e^{G+H}}} $.
First of all, by \eqref{Lambda} and \eqref{HandG}, one has
\begin{equation}\label{leidongduo}
    e^G = e^{\frac{w+z}{2}} \left(\frac{c+\sqrt{\partial_\rho P}}{c - \sqrt{\partial_\rho P}} \right).
\end{equation}
According to Lemma \ref{RERiemannInvairantBound} and Theorem \ref{REbound}, $e^G > 0$ has uniform upper and lower bounds.

Second, by \eqref{Lambda}, \eqref{HandG} and Theorem \ref{REbound}, one has
\begin{equation}
    \label{e^G-e^H}
    e^G - e^H  = \left(\frac{c+u}{c-u} \right) \left( \frac{4c\sqrt{\partial_\rho P}}{c^2 - \partial_\rho P} \right) \approx \sqrt{\partial_\rho P},
\end{equation}
which implies that 
\begin{equation}\label{libeili}
 \frac{e^G +1}{\sqrt{e^{2G} - e^{G+H}}}\approx \frac{1}{(\partial_\rho P)^{\frac{1}{4}}}.
\end{equation}
Then the desired estimate on $e^{-g}$  follows from \eqref{qubo} and \eqref{libeili}. The proof of this lemma is complete.
\end{proof}

Next, in order to prove the positivity of  $e^{-g} \partial_z  \lambda_2$ and $ e^{-h} \partial_w \lambda_3$, one needs the following relations:
\begin{lemma} \label{DerivativeOfGammaHGnaetc}
\begin{equation*}
\begin{split}
    \partial_w \Lambda =& -\frac{\partial_{\rho\rho}P}{2\partial_\rho F}, \quad \partial_z \Lambda = \frac{\partial_{\rho\rho}P}{2\partial_\rho F}, \quad \partial_{\hat{S}} \Lambda = -\frac{\partial_{\rho\rho}P \partial_S F}{\partial_\rho F} + \partial_{\rho S}P,\\
    \partial_w H =& \frac{1}{c} + \frac{c \partial_{\rho \rho} P}{2(c^2 - \partial_{\rho}P) \sqrt{\partial_{\rho}P}\partial_{\rho}F}, \quad \partial_z H = \frac{1}{c} - \frac{c \partial_{\rho \rho} P}{2(c^2 - \partial_{\rho}P) \sqrt{\partial_{\rho}P}\partial_{\rho}F},\\
    \partial_w G =& \frac{1}{c} - \frac{c \partial_{\rho \rho} P}{2(c^2 - \partial_{\rho}P) \sqrt{\partial_{\rho}P}\partial_{\rho}F}, \quad \partial_z G = \frac{1}{c} + \frac{c \partial_{\rho \rho} P}{2(c^2 - \partial_{\rho}P) \sqrt{\partial_{\rho}P}\partial_{\rho}F},\\
    \partial_{\hat{S}} H =& -\frac{c }{(c^2 - \partial_{\rho}P) \sqrt{\partial_{\rho}P}}\left[\partial_{\rho S}P - \partial_{\rho \rho}P\frac{\partial_{ S}F}{\partial_{\rho}F}\right] = - \partial_{\hat{S}} G,\\
    \partial_w a =& -\frac{1}{2\partial_\rho F}\partial_\rho a  = - \partial_z a,\quad     \partial_{\hat{S}} \widetilde{n} = - \frac{c\partial_\rho n \partial_S F}{\partial_\rho F \sqrt{c^2 - u^2}} + \frac{c \partial_S n}{ \sqrt{c^2 - u^2}},\\
    \partial_w \widetilde{n} =& - \frac{c\partial_\rho n}{2\partial_\rho F \sqrt{c^2 - u^2}} + \frac{un}{2c \sqrt{c^2 - u^2}}, \quad \partial_z \widetilde{n} = \frac{c\partial_\rho n}{2\partial_\rho F \sqrt{c^2 - u^2}} + \frac{un}{2c \sqrt{c^2 - u^2}}.
    \end{split}
\end{equation*}
\end{lemma}
\begin{proof}
The above formulae can be directly obtained by  the chain rule. We  only prove the one on  $\partial_w a$ as an example. First, via the chain rule,  one has
\begin{equation}\label{luni}
    \partial_w a = -\frac{1}{2}\partial_{\frac{z-w}{2}}a + \frac{1}{c}\partial_{\frac{z+w}{c}}a.
\end{equation}
By Lemma \ref{FromrhoUStoRiemannInvariant}, and note that $\partial_u a = 0$ as $a$ is independent of $u$, one can obtain that 
\begin{equation*}
\begin{split}
    \partial_{\frac{z-w}{2}} a &= \partial_\rho a \partial_{\frac{z-w}{2}} \rho + \partial_u a \partial_{\frac{z-w}{2}} u + \partial_S a \partial_{\frac{z-w}{2}} S = \frac{1}{\partial_\rho F} \partial_\rho a,\\
    \partial_{\frac{z+w}{c}} a &= \partial_\rho a \partial_{\frac{z+w}{c}} \rho + \partial_u a \partial_{\frac{z+w}{c}} u + \partial_S a \partial_{\frac{z+w}{c}} S
    = 0,
    \end{split}
\end{equation*}
which, along with \eqref{luni}, yields the desired formula on $\partial_w a$.
\end{proof}

Now based on Lemma \ref{DerivativeOfGammaHGnaetc}, one has 
\begin{corollary}\label{e^-gdzlambda_2>0}
\begin{equation*}
    e^{-g} \partial_z  \lambda_2 > 0 \quad \text{and} \quad e^{-h} \partial_w  \lambda_3 > 0.
\end{equation*}
\end{corollary}
\begin{proof}
It follows from \eqref{lambdas}, Lemma \ref{DerivativeOfGammaHGnaetc}, Corollary \ref{approxforPetc} and Theorem \ref{REbound} that 
\begin{equation}\label{kawani}
    \partial_z \lambda_2 = \frac{2 \partial_z G}{\left(e^G+1\right)^2} = \frac{2}{\left(e^G+1\right)^2} \left(\frac{1}{c} + \frac{c \partial_{\rho \rho} P}{2(c^2 - \partial_{\rho}P) \sqrt{\partial_{\rho}P}\partial_{\rho}F} \right) > 0.
\end{equation}
Hence $e^{-g} \partial_z  \lambda_2 > 0$. Similarly, one can prove that $e^{-h} \partial_w  \lambda_3 > 0$.
\end{proof}

Now, based on  the above lemmas, one can prove the boundedness of  the coefficients $a_i$ ($i=3,4$) and $b_i$ ($i=3,4$) in the decoupled ODEs \eqref{yuanshiode}. 

\begin{lemma}\label{b_3bound}
\begin{equation*}
    \frac{|b_3|}{\partial_z \lambda_2 e^{-g}} \leq D. 
\end{equation*}
\end{lemma}
\begin{proof}
According to \eqref{leidongduo}, Lemma \ref{RERiemannInvairantBound} and Theorem \ref{REbound},  one gets 
\begin{equation}
    \label{e^G = 1}
    e^G \approx 1,
\end{equation}
which, along  with Corollary \ref{approxforPetc}, \eqref{kawani} and Theorem \ref{REbound}, implies
\begin{equation}\label{yaoming}
    \partial_z \lambda_2 \approx 1.
\end{equation}
Hence by the definition of $b_3$ given by \eqref{coeffientsB}, one can get that
\begin{equation*}
    \frac{|b_3|}{\partial_z \lambda_2 e^{-g}} \lesssim  |M\widetilde{n} \eta|+|b_0|.
\end{equation*}
The rest of the proof can be divided into two steps: estimates on $|M\widetilde{n} \eta|$ and  $|b_0|$.

Step 1: the estimate on  $|M\widetilde{n} \eta|$. By the definition of $M$ given by \eqref{tidleAlphaBeta}, one has
\begin{align}\label{feidinande}
    |M\widetilde{n} \eta| \lesssim \left|\frac{\eta}{\widetilde{n}}\right| \left|\int_{z}^{w} e^g \widetilde{n} \partial_\varsigma a \left(\frac{\lambda_1-\lambda_3}{\lambda_3-\lambda_2}\right)(\varsigma,z,\hat{S}) \text{d}\varsigma\right|.
\end{align}
First, according to  Lemma \ref{DerivativeOfGammaHGnaetc}, Corollary \ref{approxforPetc}, $\gamma > 1$, Theorem \ref{REbound} and its proof, one has
\begin{equation}
    \label{dwa<1}
    \partial_w a \lesssim 1 + \rho^{\gamma-1} + \rho \lesssim 1.
\end{equation}

Second, it follows from \eqref{3tezheng} and direct calculations that 
\begin{equation*}
    \lambda_2 - \lambda_3 = \frac{2c^2 (c^2 - u^2)}{c^4- u^2 \partial_\rho P} \sqrt{\partial_\rho P},\quad 
    \lambda_1 - \lambda_3 = \frac{c^2 - u^2}{c^2 - u\sqrt{\partial_\rho P}}\sqrt{\partial_\rho P},
\end{equation*}
which, along with Theorem \ref{REbound}, implies that 
\begin{equation}
    \label{4.74}
    \left|\frac{\lambda_1-\lambda_3}{\lambda_3-\lambda_2}\right| \approx 1.
\end{equation}

At last,  according to  Lemmas \ref{napproxrho}-\ref{LemmaApproxhandg},  one has 
\begin{equation*}
    |e^g \widetilde{n}| \lesssim  \left| (\partial_\rho P)^{\frac{1}{4}} \sqrt{\rho} \right|,
\end{equation*}
which, along with the above estimates, Lemma \ref{REconserved}, Lemma \ref{napproxrho} and  Corollary \ref{approxforPetc}, implies that 
\begin{equation*}
\begin{split}
    |M\widetilde{n} \eta| \lesssim &  \left|\int_{0}^{\frac{z-w}{2}} (\partial_\rho P)^{\frac{1}{4}} \sqrt{\rho} (\Xi,\hat{S})  \text{d} \Xi \right|\\
    \lesssim &  \left|\int_{0}^{\rho} (\partial_\sigma P)^{\frac{1}{4}} \sqrt{\sigma} \partial_\sigma F  \text{d}\sigma\right| \lesssim  \rho^{\frac{3\gamma-1}{4}}.
    \end{split}
\end{equation*}
As $\gamma > 1$ and $\rho$ has an uniform upper bound by Theorem \ref{REbound} and its proof, $|M\widetilde{n}\eta|$ is bounded.

Step 2: the estimate on  $|b_0|$. First, one has  $|b_0| \leq \sum_{i=1}^{6} |b_0^i|$, where $b_0^i$ is the i-th term of $b_0$. First, by $\eqref{releul1D}_3$, Lemma \ref{LemmaApproxhandg} and Corollary \ref{approxforPetc}, one has
\begin{equation}\label{tulei}
    |b_0^1| = |e^g \partial_2 S a \partial_z g| \lesssim \frac{(\partial_\rho P)^{\frac{1}{4}}}{\sqrt{\rho}} (\lambda_2- \lambda_1) \left|\frac{\eta}{\widetilde{n}}\right| \widetilde{n} \rho^{\frac{\gamma-1}{2}} |\partial_z g|.
\end{equation}
It follows from \eqref{3tezheng}, \eqref{tidleAlphaBeta}, \eqref{e^G-e^H}, Lemma \ref{DerivativeOfGammaHGnaetc}, Corollary \ref{approxforPetc}, Theorem \ref{REbound} and direct calculations that 
\begin{equation}
\label{lambda_2-lambda_1}
\begin{split}
    \lambda_2 - \lambda_1 =& \frac{c^2-u^2}{c^2 + u \sqrt{\partial_\rho P}} \sqrt{\partial_\rho P}\lesssim \sqrt{\partial_\rho P},\\
    |\partial_z g| = &  \left| - \frac{e^G \partial_z G}{e^G+1} + \frac{2e^G \partial_z G - e^H (\partial_z G+ \partial_z H)}{2(e^G-e^H)} - \frac{(c-\sqrt{\partial_\rho P})^2}{4c^2 \sqrt{\partial_\rho P}} \right| \lesssim 1 + \frac{1}{\sqrt{\partial_\rho P}},
    \end{split}
\end{equation}
which, along with \eqref{tulei}, implies that 
\begin{equation*}
    |b_0^1| \lesssim (\partial_\rho P)^{\frac{5}{4}} \sqrt{\rho} \left|\frac{\eta}{\widetilde{n}}\right|\left(1 + \frac{1}{\sqrt{\partial_\rho P}}\right).
\end{equation*}

Similarly, by $\eqref{releul1D}_3$, \eqref{lambda_2-lambda_1} and Lemma \ref{LemmaApproxhandg}, one has 
\begin{equation}\label{luonaerduo}
    |b_0^2| = |e^g\partial_2 S \partial_{\hat{S}} g| \lesssim \frac{(\partial_\rho P)^{\frac{1}{4}}}{\sqrt{\rho}} \left|\frac{\eta}{\widetilde{n}}\right| \widetilde{n} \sqrt{\partial_\rho P} |\partial_{\hat{S}} g|.
\end{equation}
According to  \eqref{tidleAlphaBeta}, one has
\begin{equation}\label{maikong}
\begin{split}
    \partial_{\hat{S}} g = &- \frac{e^G \partial_{\hat{S}} G}{e^G+1} + \frac{2e^G \partial_{\hat{S}} G - e^H (\partial_{\hat{S}} G+ \partial_{\hat{S}} H)}{2(e^G-e^H)}-\partial_{\hat{S}}J,\\
    \partial_{\hat{S}}J=&- \int_{\epsilon}^{\frac{z-w}{2}} \frac{c^2 - \partial_\rho P}{4c^2 (\partial_\rho P)^{\frac{3}{2}}} \left[\partial_{\rho S}P -\frac{\partial_{\rho \rho}P \partial_S F}{2\partial_\rho F}\right](\Xi,\hat{S}) \text{d} \Xi .
\end{split}
\end{equation}
According to \eqref{HandG}, Lemma \ref{DerivativeOfGammaHGnaetc}, Corollary \ref{approxforPetc} and Theorem \ref{REbound}, one can obtain 
\begin{equation}\label{runiaoer}
    \quad |\partial_{\hat{S}} G| = |\partial_{\hat{S}} H| \lesssim \sqrt{\partial_\rho P}.
\end{equation}
By Corollary \ref{approxforPetc} and a change of variable as in the proof of Lemma \ref{LemmaApproxhandg}, one can obtain
\begin{align*}
    |\partial_{\hat{S}} J| &=  \left|\int_{\widetilde{\epsilon}(S)}^{\rho} \frac{c^2 - \partial_\sigma P}{4c^2 (\partial_\sigma P)^{\frac{3}{2}}} \left[\partial_{\sigma S}P -\frac{\partial_{\sigma \sigma}P \partial_S F}{2\partial_\sigma F}\right] \partial_\sigma F \text{d} \sigma\right|, \\& 
    \lesssim \left|\int_{\widetilde{\epsilon}(S)}^{\rho} \frac{1}{\sigma} \text{d} \sigma\right|
    \lesssim 1+ |\log(\rho)|,
\end{align*}
which, along with \eqref{e^G-e^H}, \eqref{luonaerduo}-\eqref{runiaoer} and direct calculations, yields that 
\begin{align*}
    |b_0^2|
    \lesssim \frac{(\partial_\rho P)^{\frac{3}{4}}}{\sqrt{\rho}} \left|\frac{\eta}{\widetilde{n}} \right| \widetilde{n} (1 + \sqrt{\partial_\rho P} +  |\log(\rho)|).
\end{align*}

Next, by $\eqref{releul1D}_3$, \eqref{e^G = 1}, \eqref{yaoming}, \eqref{4.74}, \eqref{runiaoer}, Lemma \ref{LemmaApproxhandg}-\ref{DerivativeOfGammaHGnaetc}, Corollary \ref{approxforPetc} and Theorem \ref{REbound}, one has $|\partial_w \lambda_2| \approx 1$ and
\begin{equation*}
\begin{split}
 |b_{0}^3| =&\left|e^g \partial_w \lambda_2 \frac{a \partial_3 S}{\lambda_3 - \lambda_2}\right| = \left|e^g \partial_w \lambda_2 \frac{a (\lambda_3-\lambda_1) \eta}{\lambda_3 - \lambda_2}\right|
 \lesssim \frac{(\partial_\rho P)^{\frac{1}{4}}}{\sqrt{\rho}} \rho^{\frac{\gamma-1}{2}} \widetilde{n} \left|\frac{\eta}{\widetilde{n}}\right|,\\
|b_{0}^4| =& |e^g \eta a \partial_z \lambda_2| \lesssim \frac{(\partial_\rho P)^{\frac{1}{4}}}{\sqrt{\rho}} \widetilde{n} \left|\frac{\eta}{\widetilde{n}} \right| \rho^{\frac{\gamma-1}{2}},\\
  |b_{0}^5| =& |e^g \eta \partial_{\hat{S}} \lambda_2| \lesssim \frac{(\partial_\rho P)^{\frac{1}{4}}}{\sqrt{\rho}} \widetilde{n} \left|\frac{\eta}{\widetilde{n}} \right| \left|\partial_{\hat{S}} \lambda_2 \right|\\
  =&\frac{(\partial_\rho P)^{\frac{1}{4}}}{\sqrt{\rho}} \widetilde{n} \left|\frac{\eta}{\widetilde{n}} \right| \left|\frac{2e^G \partial_{\hat{S}}G}{ (e^G+1)^2} \right|
  \lesssim  \frac{(\partial_\rho P)^{\frac{1}{4}}}{\sqrt{\rho}} \widetilde{n} \left|\frac{\eta}{\widetilde{n}} \right| \sqrt{\partial_\rho P}.
  \end{split}
\end{equation*}

Finally, by \eqref{lambda_2-lambda_1}, Lemma \ref{LemmaApproxhandg}, Corollary \ref{approxforPetc} and Theorem \ref{REbound},  one has
\begin{equation*}
    |b_0^6| = |e^g \eta \partial_z a (\lambda_1-\lambda_2)| \lesssim \frac{(\partial_\rho P)^{\frac{1}{4}}}{\sqrt{\rho}} \widetilde{n} \left|\frac{\eta}{\widetilde{n}} \right| |\partial_z a| \sqrt{\partial_\rho P}.
\end{equation*}
According to  Lemma \ref{DerivativeOfGammaHGnaetc} and  Corollary \ref{approxforPetc}, one gets
\begin{equation*}
    |\partial_z a| \lesssim 1 + \rho^{\gamma-1 } + \rho.
\end{equation*}
Hence
\begin{equation*}
    |b_0^6| \lesssim \frac{(\partial_\rho P)^{\frac{1}{4}}}{\sqrt{\rho}} \widetilde{n} \left|\frac{\eta}{\widetilde{n}} \right| \sqrt{\partial_\rho P} (1 + \rho^{\gamma-1 } + \rho).
\end{equation*}

Thus, the desired boundedness follows from the above estimates, Lemma \ref{REconserved}, Lemma \ref{napproxrho}, Corollary \ref{approxforPetc}, Theorem \ref{REbound} and its proof.
\end{proof}

\begin{lemma} \label{b_1bound}
\begin{equation*}
    \frac{|b_1|}{\partial_z \lambda_2 e^{-g}} \leq D. 
\end{equation*}
\end{lemma}
\begin{proof}The proof will be divided into two steps: the estimate on a key term  $\frac{1}{|\partial_w\widetilde{n}|}$ and then the boundedness of   $\frac{|b_1|}{\partial_z \lambda_2 e^{-g}}$.

Step 1: the estimate on  $\frac{1}{|\partial_w\widetilde{n}|}$. According to  Lemma \ref{DerivativeOfGammaHGnaetc}, if $u \leq 0$, then
\begin{equation}\label{yintao}
    |\partial_w \widetilde{n}| = \left|\frac{c\partial_\rho n}{2\partial_\rho F \sqrt{c^2 - u^2}} - \frac{un}{2c \sqrt{c^2 - u^2}}\right| \geq \frac{c\partial_\rho n}{2\partial_\rho F \sqrt{c^2 - u^2}}.
\end{equation}
If $u > 0$, by Theorem \ref{REbound}, one gets that
\begin{equation}\label{siqing}
\begin{split}
    |\partial_w \widetilde{n}| &\geq \frac{c\partial_\rho n}{2\partial_\rho F \sqrt{c^2 - u^2}} - \frac{n}{2 \sqrt{c^2 - u^2}}\\
    &= \frac{n}{2\sqrt{\partial_\rho P}\sqrt{c^2 - u^2}} \left(\frac{c\sqrt{\partial_\rho P}\partial_\rho n}{n\partial_\rho F}-\sqrt{\partial_\rho P}\right).
    \end{split}
\end{equation}
We claim that there exists some  constant  $D^*>0$   depending only on $(z_0,w_0,S_0)$, $R$, $\gamma$ and $c$, and   independent of $(t,x) \in [0,T) \times \mathbb{R}$, such that
\begin{equation}
    \label{dwn_upperbound}
    \frac{c\sqrt{\partial_\rho P}\partial_\rho n}{n\partial_\rho F}-\sqrt{\partial_\rho P} > D^*.
\end{equation}
According to  Proposition \ref{AProp_nandrho} and  \eqref{tianrui-1}, one has $\rho \partial_\rho n = n$. Then
\begin{equation}\label{sangqiesi}
\begin{split}
    &\frac{c\sqrt{\partial_\rho P}\partial_\rho n}{n\partial_\rho F}-\sqrt{\partial_\rho P} = \frac{\sqrt{\partial_\rho P}}{n \partial_\rho F} (c \partial_\rho n - n \partial_\rho F) \\
    =& \frac{\sqrt{\partial_\rho P}}{\rho n \partial_\rho F} (c n - \rho n \partial_\rho F)= \frac{\sqrt{\partial_\rho P}}{\rho \partial_\rho F} (c - \rho \partial_\rho F).
\end{split}
\end{equation}
By \eqref{xierwa}, Theorem \ref{REbound} and the fact that $\frac{P}{\rho} > 0$, one has
\begin{equation*}
    \rho \partial_\rho F = \frac{c^2 \sqrt{\partial_\rho P}}{c^2 + \frac{P}{\rho}} < \sqrt{\partial_\rho P} < c- D^{**},
\end{equation*}
for some  constant  $D^{**}>0$   depending only on $(z_0,w_0,S_0)$, $R$, $\gamma$ and $c$, and   independent of $(t,x) \in [0,T) \times \mathbb{R}$. Then one has 
\begin{equation}\label{liuxiaona}
    c- \rho \partial_\rho F > D^{**}.
\end{equation}
Also note that, according to \eqref{xierwa} and Corollary \ref{approxforPetc}, one has
\begin{equation*}
    \frac{\sqrt{\partial_\rho P}}{\rho \partial_\rho F} = \frac{c^2 + \frac{P}{\rho}}{c^2} \approx 1,
\end{equation*}
which, along with \eqref{sangqiesi}-\eqref{liuxiaona},  implies that  \eqref{dwn_upperbound} is indeed true.

Then according to \eqref{tianrui-1}, \eqref{yintao}-\eqref{dwn_upperbound}, Proposition \ref{EntropyBound}, Lemma \ref{napproxrho}, Corollary \ref{approxforPetc} and Theorem \ref{REbound}, one has 
\begin{equation}
    \label{dwn_lowerbound}
    \frac{1}{|\partial_w \widetilde{n}|} \lesssim \frac{\sqrt{\partial_\rho P}}{\rho}.
\end{equation}

Step 2: the boundedness of  $|b_1| e^g$. According to \eqref{yaoming}, one has 
\begin{equation*}
    \frac{|b_1|}{\partial_z \lambda_2 e^{-g}} \approx |b_1|e^{g},
\end{equation*}
where $|b_1|e^{g} \leq \sum_{i=1}^4 |b_1^i|e^g$, and $b_1^i$ is  the i-th term in $b_1$.

First, by $\eqref{riPDEs}_2$, \eqref{dwa<1}, \eqref{4.74}, \eqref{dwn_lowerbound} and Lemma \ref{LemmaApproxhandg}, one has
\begin{equation*}
\begin{split}
    |b_1^1|e^g =& \left|e^{2g} \eta \partial_w a \left(\frac{\lambda_1 - \lambda_3}{\lambda_3 - \lambda_2}\right)\frac{\partial_z \widetilde{n}}{\partial_w \widetilde{n}}\partial_2 z\right|\\
    \lesssim &\frac{\sqrt{\partial_\rho P}}{\rho} \left|\frac{\eta}{\widetilde{n}}\right| \widetilde{n} \frac{|\partial_z \widetilde{n}|\sqrt{\partial_\rho P}}{\rho} |a||\partial_2 S|.
    \end{split}
\end{equation*}
By $\eqref{releul1D}_3$, \eqref{lambda_2-lambda_1}, \eqref{tianrui-1}, Proposition \ref{EntropyBound}, Lemma \ref{napproxrho}, Lemma \ref{DerivativeOfGammaHGnaetc}, Corollary \ref{approxforPetc} and Theorem \ref{REbound}, one can get that
\begin{equation}
    \label{d2SEstimate}
    |\partial_z \widetilde{n}| \lesssim \frac{\rho}{\sqrt{\partial_\rho P}}, \quad |\partial_2 S| \lesssim \sqrt{\partial_\rho P} \left|\frac{\eta}{\widetilde{n}}\right| \widetilde{n}, \quad |a| \approx\sqrt{\partial_\rho P}.
\end{equation}
Hence one has 
\begin{equation*}
    |b_1^1|e^g \leq \frac{(\partial_\rho P)^{\frac{3}{2}}}{\rho^2} \left|\frac{\eta}{\widetilde{n}}\right|^2 \widetilde{n}^2.
\end{equation*}

Next, by \eqref{dwa<1}, \eqref{4.74}, \eqref{dwn_lowerbound}, \eqref{d2SEstimate} and Lemma \ref{LemmaApproxhandg}, one has
\begin{equation}\label{jiujingfazi}
    |b_1^2|e^g = \left|e^{2g} \eta \partial_w a \left(\frac{\lambda_1 - \lambda_3}{\lambda_3 - \lambda_2}\right)\frac{\partial_{\hat{S}} \widetilde{n}}{\partial_w \widetilde{n}}\partial_2 S\right| \lesssim  \left|\frac{\eta}{\widetilde{n}}\right|^2 \frac{(\partial_\rho P)^{\frac{3}{2}}}{\rho^2} \widetilde{n}^2 |\partial_{\hat{S}} \widetilde{n}|.
\end{equation}
From Theorem \ref{REbound}, Lemma \ref{DerivativeOfGammaHGnaetc} and Corollary \ref{approxforPetc}, one can get that
\begin{equation*}
    |\partial_{\hat{S}} \widetilde{n}| \lesssim \rho |\partial_\rho n| + |\partial_{S} n|.
\end{equation*}
According to  \eqref{tianrui-1}, Proposition \ref{EntropyBound} and Lemma \ref{napproxrho}, one can conclude 
\begin{equation*}
    |\partial_{S} n| \lesssim \rho^\gamma, \quad \rho |\partial_\rho n| = \rho,
\end{equation*}
so
\begin{equation}
    \label{dSnOrder}
    |\partial_{\hat{S}} \widetilde{n}| \lesssim \rho,
\end{equation}
which, along with \eqref{jiujingfazi},  implies
\begin{equation*}
    |b_1^2|e^g \leq \left|\frac{\eta}{\widetilde{n}}\right|^2 \frac{(\partial_\rho P)^{\frac{3}{2}}}{\rho^2} \widetilde{n}^2 \rho.
\end{equation*}

At last, by $\eqref{releul1D}_3$, \eqref{3tezheng}, \eqref{dwa<1}, \eqref{4.74}, Lemma \ref{napproxrho}-\ref{LemmaApproxhandg}, Corollary \ref{approxforPetc} and Theorem \ref{REbound}, $|b_1^3|e^g$ and $|b_1^4|e^g$ can be easily dealt with as
\begin{equation*}
\begin{split}
    |b_1^3|e^g =& \left|e^{2g} \eta \partial_w a \left(\frac{\lambda_1 - \lambda_3}{\lambda_3 - \lambda_2}\right)a \partial_3 S\right| \lesssim  \left|\frac{\eta}{\widetilde{n}}\right|^2 (\partial_\rho P)^{\frac{3}{2}} \rho,\\
    |b_1^4|e^g =& \left|e^{2g} \eta^2 \partial_w a \left(\lambda_1 - \lambda_3\right) a\right| \lesssim  \left|\frac{\eta}{\widetilde{n}}\right|^2 (\partial_\rho P)^{\frac{3}{2}} \rho.
    \end{split}
\end{equation*}

Thus, the desired boundedness follows from the above estimates, Lemma \ref{REconserved}, Lemma \ref{napproxrho}, Corollary \ref{approxforPetc},  Theorem \ref{REbound} and its proof.
\end{proof}

\begin{lemma}\label{zhengzhi}
\begin{equation*}
    \frac{|b_2|}{\partial_z \lambda_2 e^{-g}} \leq D. 
\end{equation*}
\end{lemma}
\begin{proof}
According to \eqref{yaoming}, one needs to show that that 
\begin{equation*}
    e^g |b_2| \leq D,
\end{equation*}
where $|b_2| \leq \sum_{i=1}^5 |b_2^i|$, and  $b_2^i$ is  the i-th term of $b_2$. 

First, by the definition of $M$ given in \eqref{tidleAlphaBeta}, one can obtain 
\begin{align*}
    \partial_w M &= -2\frac{\partial_w \widetilde{n}}{\widetilde{n}^3} \int_{z}^{w} e^g \widetilde{n} \partial_\varsigma a \left(\frac{\lambda_1-\lambda_3}{\lambda_3-\lambda_2}\right)(\varsigma,z,\hat{S}) \text{d}\varsigma - \frac{1}{\widetilde{n}} e^g \partial_w a \left(\frac{\lambda_1-\lambda_3}{\lambda_3-\lambda_2}\right) \\&
    = -2\frac{\partial_w \widetilde{n}}{\widetilde{n}} M - \frac{1}{\widetilde{n}} e^g \partial_w a \left(\frac{\lambda_1-\lambda_3}{\lambda_3-\lambda_2}\right).
\end{align*}
By the same method as how $|M\widetilde{n} \eta| \lesssim \rho^{\frac{3\gamma -1}{4}}$ was obtained in the proof of Lemma \ref{b_3bound}, one has that $|M| \lesssim \rho^{\frac{3\gamma -9}{4}}$, which, along with \eqref{tianrui-1}, \eqref{dwa<1}, \eqref{4.74}, Lemma \ref{napproxrho}-\ref{DerivativeOfGammaHGnaetc}, Corollary \ref{approxforPetc} and Theorem \ref{REbound}, yields
\begin{equation}
    \label{dwMorder}
    |\partial_w M| \lesssim \frac{\rho^{\frac{3\gamma -9}{4}}}{\sqrt{\partial_\rho P}} + \frac{(\partial_\rho P)^{\frac{1}{4}}}{\rho\sqrt{\rho}}.
\end{equation}
Then, by $\eqref{riPDEs}_2$, \eqref{dwn_lowerbound}, \eqref{d2SEstimate}, \eqref{dSnOrder}, \eqref{dwMorder}, Lemma \ref{napproxrho}-\ref{LemmaApproxhandg} and Corollary \ref{approxforPetc}, one has
\begin{align*}
    |b_2^1|e^{g} \leq & \left|e^{g} \eta \widetilde{n} \frac{\partial_z \widetilde{n}}{\partial_w \widetilde{n}}\partial_2 z \partial_w M \right| \\
    \lesssim  &
    (\partial_\rho P)^{\frac{1}{4}} \sqrt{\rho} \left|\frac{\eta}{\widetilde{n}}\right|^2 \left(\sqrt{\partial_\rho P} \rho^{\frac{3\gamma -1}{4}} + \sqrt{\rho} (\partial_\rho P)^{\frac{5}{4}}\right),\\
    |b_2^2|e^{g} =& \left|e^{g} \eta \widetilde{n} \frac{\partial_{\hat{S}} \widetilde{n}}{\partial_w \widetilde{n}}\partial_2 S \partial_w M \right|\\
    \lesssim &
    (\partial_\rho P)^{\frac{1}{4}} \sqrt{\rho} \rho^2 \left|\frac{\eta}{\widetilde{n}}\right|^2 \partial_\rho P |\partial_w M| \\
    \lesssim & (\partial_\rho P)^{\frac{1}{4}} \sqrt{\rho}\left|\frac{\eta}{\widetilde{n}}\right|^2 \left(\sqrt{\partial_\rho P} \rho^{\frac{3\gamma -1}{4}} + \sqrt{\rho} (\partial_\rho P)^{\frac{5}{4}}\right).
\end{align*}

Second, by $\eqref{riPDEs}_2$ and Lemma \ref{napproxrho}-\ref{LemmaApproxhandg}, one has
\begin{equation*}
    |b_2^3|e^{g} = \left|e^{g} \eta \widetilde{n} \partial_2 z \partial_z M \right| \lesssim \frac{(\partial_\rho P)^{\frac{1}{4}}}{\sqrt{\rho}} \left|\frac{\eta}{\widetilde{n}}\right|^2 \rho^3 \partial_\rho P |\partial_z M|.
\end{equation*}
Define the function 
\begin{equation*}
    \Omega(w,z,\hat{S}) \coloneqq e^g \widetilde{n} \partial_w a \frac{\lambda_1 - \lambda_3}{\lambda_3 - \lambda_2} = -e^g \widetilde{n} \partial_w a \frac{(c^2 + u \sqrt{\partial_\rho P})}{2c^2}.
\end{equation*}
According to  \eqref{tidleAlphaBeta}  and the fact that: $z = w$ if and only if $\rho = 0$, one has $\Omega(z,z,\hat{S})=0$ and
\begin{align*}
    \partial_z M &= -2\frac{\partial_z \widetilde{n}}{\widetilde{n}^3} \int_{z}^{w} \Omega(\varsigma,z,\hat{S}) \text{d}\varsigma - \frac{1}{\widetilde{n}^2} \Omega(z,z,\hat{S}) - \frac{1}{\widetilde{n}^2} \int_{z}^{w} \partial_z \Omega(\varsigma,z,\hat{S}) \text{d}\varsigma \\&
    = -2\frac{\partial_z \widetilde{n}}{\widetilde{n}} M  - \frac{1}{\widetilde{n}^2} \int_{z}^{w} \partial_z \Omega(\varsigma,z,\hat{S}) \text{d}\varsigma.
\end{align*}
Now by \eqref{d2SEstimate}, $|M| \lesssim \rho^{\frac{3\gamma-9}{4}}$ and Lemma \ref{napproxrho}, one has
\begin{equation*}
    |b_2^3|e^{g} \lesssim (\partial_\rho P)^{\frac{1}{4}}\sqrt{\rho} \left|\frac{\eta}{\widetilde{n}}\right|^2 \left(\rho^{\frac{3\gamma-1}{4}} + \partial_\rho P \left|\int_{z}^{w} \partial_z \Omega(\varsigma,z,\hat{S}) \ \text{d}\varsigma\right|\right).
\end{equation*}
All the terms associated with the integral involving $\partial_z \Omega$ can be easily bounded by the previous procedures, except for the term
\begin{equation*}
    \int_{z}^{w} e^g \widetilde{n} \partial_{z \varsigma}a \frac{(c^2 + u \sqrt{\partial_\rho P})}{2c^2}(\varsigma,z,\hat{S}) \text{d} \varsigma.
\end{equation*}
As $a(\rho,S)$ is a function independent of $u$, by Lemma \ref{FromrhoUStoRiemannInvariant} and explicit computations, one can obtain that
\begin{equation*}
    \partial_{w z} a = - \frac{1}{4 \partial_\rho F} \partial_\rho \left(\frac{\partial_\rho a}{\partial_\rho F}\right),
\end{equation*}
which, along with the Mean Value Theorem, yields that
\begin{align*}
    &\int_{z}^{w} e^g \widetilde{n} \partial_{z \varsigma}a \frac{(c^2 + u \sqrt{\partial_\rho P})}{2c^2}(\varsigma,z,\hat{S}) \text{d} \varsigma \\&= -\frac{1}{4} e^g(z, w^*, \hat{S}) \widetilde{n}(z, w^*, \hat{S}) \frac{(c^2 + u \sqrt{\partial_\rho P})}{2c^2}(z, w^*, \hat{S})
    \int_{z}^{w} \frac{1}{\partial_\rho F}\partial_\rho \left(\frac{\partial_\rho a}{\partial_\rho F}\right)(\varsigma,z,\hat{S})  \text{d} \varsigma,
\end{align*}
for some  $w^* \in [z,w]$. 

Here, via the change of variable formula, Corollary \ref{approxforPetc}, Theorem \ref{REbound} and its proof, one has 
\begin{equation*}
    \left|\int_{z}^{w} \frac{1}{\partial_\rho F}\partial_\rho \left(\frac{\partial_\rho a}{\partial_\rho F}\right)(\varsigma,z,\hat{S}) \text{d} \varsigma \right| = \left|\int_{0}^{\rho} \partial_\sigma \left(\frac{\partial_\sigma a}{\partial_\sigma F}\right) \text{d} \sigma\right| = \left|\frac{\partial_\rho a}{\partial_\rho F} \right| \lesssim 1.
\end{equation*}
Hence, by Lemma \ref{napproxrho}-\ref{LemmaApproxhandg}, one has
\begin{equation*}
    |b_2^3|e^{g} \lesssim (\partial_\rho P)^{\frac{1}{4}}\sqrt{\rho} \left|\frac{\eta}{\widetilde{n}}\right|^2 \left(\rho^{\frac{3\gamma-1}{4}} + \partial_\rho P \right).
\end{equation*}
Similarly,  $|\partial_{\hat{S}} M| \lesssim 1$ can be proved  by the same argument. Hence, by \eqref{d2SEstimate}, Lemma \ref{napproxrho}-\ref{LemmaApproxhandg}, one has
\begin{align*}
    |b_2^4|e^{g} =& \left|e^{g} \eta \widetilde{n} \partial_2 S \partial_{\hat{S}} M \right| \\
    \lesssim &\frac{(\partial_\rho P)^{\frac{1}{4}}}{\sqrt{\rho}} \left|\frac{\eta}{\widetilde{n}}\right|^2 \rho^3 \sqrt{\partial_\rho P} |\partial_{\hat{S}} M| \lesssim \frac{(\partial_\rho P)^{\frac{1}{4}}}{\sqrt{\rho}} \left|\frac{\eta}{\widetilde{n}}\right|^2 \rho^3 \sqrt{\partial_\rho P},\\
    |b_2^5|e^{g} =& \left|e^{g} \eta \widetilde{n}^2 M \left(\partial_2 \eta - \frac{\eta}{n} \partial_2 \widetilde{n}\right) \right|\\
    \lesssim &  \frac{(\partial_\rho P)^{\frac{1}{4}}}{\sqrt{\rho}} \left|\frac{\eta}{\widetilde{n}}\right| \rho^3 |M| \left|\partial_2 \eta - \frac{\eta}{n} \partial_2 \widetilde{n}\right|.
\end{align*}

Thus, the desired boundedness follows from the above estimates, Lemmas \ref{REconserved} and  \ref{napproxrho}, Corollaries \ref{REconservedCoro}-\ref{approxforPetc}, Theorem \ref{REbound} and its proof.
\end{proof}

\begin{lemma}\label{sumaozhen}
\begin{equation*}
    \frac{|b_4|}{\partial_z \lambda_2 e^{-g}} \leq D. 
\end{equation*}
\end{lemma}
\begin{proof}
The desired conclusion follows from the proof of  Lemmas \ref{b_1bound}-\ref{zhengzhi} and the fact that $|M\widetilde{n}\eta|$ is bounded from the proof of Lemma \ref{b_3bound}. 
\end{proof}

  \underline{Now we can give the proof of Theorem \ref{ODEcompare}}\\
We will focus only on $q$, and  $r$ can be dealt with similarly. First, according to Young's inequality, one has
\begin{equation*}
    b_3 q \leq |b_3 q| \leq \frac{\partial_z \lambda_2 e^{-g}}{2} q^2 + \frac{|b_3|^2}{2\partial_z \lambda_2 e^{-g}}, 
\end{equation*}
which, along with  Theorem \ref{REDecoupledODEs}, implies that 
\begin{equation}\label{liweifeng}
    \partial_2 q \leq -\frac{\partial_z \lambda_2 e^{-g}}{2} q^2 + \frac{|b_3|^2}{2\partial_z \lambda_2 e^{-g}} + |b_4|.
\end{equation}
Then the inequality $\eqref{yandqinequality}_2$ follows from \eqref{liweifeng} and Lemmas \ref{b_3bound}-\ref{sumaozhen}.

\subsection{Mass-energy density lower bound for $1 <\gamma < 3$}
 This subsection will be devoted to giving  a sharp enough lower bound estimate on $\rho$ for $1 <\gamma < 3$. For this purpose, we first give the upper  bounds of $r$ and $q$.
\begin{lemma} \label{yqleqYQ}
One has the following bounds on $r$ and $q$
\begin{equation*}
    \begin{cases}
    r \leq \max \{N_1, \sup_x r_0(x)\} \coloneqq \overline{R},\\[6pt]
    q \leq \max \{N_2, \sup_x q_0(x)\} \coloneqq \overline{Q},
    \end{cases}
\end{equation*}
where $r_0(x) = r(0,x)$ and $q_0(x) = q(0,x)$.
\end{lemma}

\begin{proof}
We will focus only on $q$, and  $r$ can be dealt with similarly. We consider 3 cases: $q_0(x_0^2) < N_2$, $q_0(x_0^2) > N_2$, and  $q_0(x_0^2) = N_2$.

First of all, consider the case $q_0(x_0^2) < N_2$. By the continuity of $q$, in a small open interval around $t = 0$, $q(t,x^2(t)) < N_2$. Define time $\widetilde{T}$ by 
\begin{equation}
    \widetilde{T} \coloneqq \inf_t \{t : q(t,x^2(t)) > N_2 \}.
\end{equation}
Then $\widetilde{T} > 0$. Note for $t \in [0,\widetilde{T}]$, one has $q(t,x^2(t)) \leq N_2$. 

It remains to show that $T = \widetilde{T}$. Suppose for contradiction that $\widetilde{T} < T$. Via  the continuity of $q(t,x^2(t))$, there exists some small $\upsilon > 0$, such that if $t \in (\widetilde{T}, \widetilde{T}+\upsilon)$, then $q(t,x^2(t)) > N_2$. By $\eqref{yandqinequality}_2 $, one has 
$$\partial_2 q \leq 0 \quad \text{for} \quad t \in (\widetilde{T}, \widetilde{T}+\upsilon).$$
Integrating this inequality  along  $x^2(t)$  over  $(\widetilde{T},\widetilde{T} + \delta)$ for any $\delta$ satisfying $0 < \delta < \upsilon$, one gets
\begin{equation*}
    q(\widetilde{T} + \delta, x^2(\widetilde{T} + \delta)) \leq q(\widetilde{T}, x^2(\widetilde{T})) \leq  N_2.
\end{equation*}
This is a contradiction to $q(t,x^2(t)) > N_2$ if $t \in (\widetilde{T}, \widetilde{T}+\upsilon)$, so one must have $\widetilde{T} = T$.

Second, consider the case $q_0(x_0^2) > N_2$. Once again by the continuity of $q$, in a small open interval around $t = 0$, one has 
$q(t,x^2(t)) > N_2$. Define time $\widetilde{T}$ by 
\begin{equation*}
    \widetilde{T} \coloneqq \inf_t \{t : q(t,x^2(t)) < N_2 \}.
\end{equation*}
Then $\widetilde{T} > 0$. Note for $t \in [0,\widetilde{T}]$, $q(t,x^2(t)) \geq  N_2$. By $\eqref{yandqinequality}_2$, one has
$$\partial_2 q \leq 0 \quad \text{for} \quad t \in [0,\widetilde{T}].$$
 Integrating this inequality  along  $x^2(t)$ over $[0,t]$ for any $t\in (0,\widetilde{T}]$, one has 
$$q(t,x^2(t)) \leq q_0 (x_0^2) \leq \sup_x q_0(x) \quad \text{ for} \quad t \in [0,\widetilde{T}],  $$
so $q$ has the upper bound $\overline{Q}$ in this interval. Now by the definition of $\widetilde{T}$ and the continuity of $q$, there exists some $\upsilon > 0$ such that $q(t,x^2(t)) < N_2$ on $t \in (\widetilde{T}, \widetilde{T} + \upsilon)$. In particular, fix a $\delta$ such that $0 < \delta < \upsilon$, one has $q(\widetilde{T} + \delta,x^2(\widetilde{T} + \delta)) < N_2$. Then the rest of the proof  is reduced to the case $q_0(x_0^2) < N_2$. Hence $q(t,x^2(t)) \leq \overline{Q}$. 

Finally,  the case $q_0(x_0^2) = N_2$ can be dealt with  similarly. 
\end{proof}

The lower bound on the $\rho $ now can  be given along the characteristics $x^2(t)$ and $x^3(t)$. 

\begin{lemma} \label{DensityLowerBound}
Assume  $1 < \gamma < 3$. One has 
\begin{equation*}
    \rho^{\frac{3-\gamma}{4}}(t,x^2(t)) \geq \frac{1}{D(t+1)}, \quad \rho^{\frac{3-\gamma}{4}}(t,x^3(t)) \geq \frac{1}{D(t+1)}.
\end{equation*}
 
\end{lemma}

\begin{proof}
We will focus only on $\rho^{\frac{3-\gamma}{4}}(t,x^2(t))$, and  $\rho^{\frac{3-\gamma}{4}}(t,x^3(t))$ can be dealt with similarly.
First, according to $\eqref{releul1D}_3$, \eqref{tidleAlphaBeta}, \eqref{yandq} and \eqref{riPDEs}, one has
\begin{equation}\label{haojunmin}
\begin{split}
    \partial_2(z-w) &= -a\partial_2 S -(\partial_t w + \lambda_2 \partial_x w) \\
    &= -a\partial_2 S -(\partial_3 w + (\lambda_2-\lambda_3) \partial_x w) \\
    &= -a\partial_2 S -a \partial_3 S - (\lambda_2-\lambda_3) \partial_x w\\
    &= -2a\eta (\lambda_2-\lambda_1) - (\lambda_2-\lambda_3) \widetilde{\alpha} \\
 &= -2a\eta (\lambda_2-\lambda_1) - (\lambda_2-\lambda_3) r e^{-h} - (\lambda_2-\lambda_3) L\eta \widetilde{n} e^{-h},
 \end{split}
\end{equation}
which, along with  Proposition \ref{EigvenvalueIneq}, Lemma \ref{yqleqYQ} and the fact 
\begin{equation*}
\begin{split}
    -2a\eta(\lambda_2-\lambda_1) \geq & -2|a||\eta| (\lambda_2-\lambda_1),\\
    -L\eta \widetilde{n} (\lambda_2-\lambda_3)e^{-h} \geq & -|L||\eta| \widetilde{n} (\lambda_2-\lambda_3)e^{-h},
    \end{split}
\end{equation*}
yields that
\begin{equation}
    \label{d2(z-w)}
    \partial_2 (z-w) \geq -2|a||\eta| (\lambda_2-\lambda_1) -|L||\eta| \widetilde{n} (\lambda_2-\lambda_3)e^{-h} - (\lambda_2-\lambda_3) \overline{R} e^{-h}.
\end{equation}
Notice that, according to \eqref{F}, one has 
\begin{equation*}
    \partial_2 (z-w) = 2\partial_\rho F \partial_2 \rho +2\partial_S F \partial_2 S,
\end{equation*}
which, along with \eqref{d2(z-w)}, yields
\begin{equation*}
    2\partial_\rho F \partial_2 \rho \geq -2|a||\eta| (\lambda_2-\lambda_1) -|L||\eta| \widetilde{n} (\lambda_2-\lambda_3)e^{-h} - (\lambda_2-\lambda_3) \overline{R} e^{-h} - 2|\partial_S F| |\partial_2 S| .
\end{equation*}
Multiplying the above inequality on  both sides by $\frac{e^{h}}{(\partial_\rho P)^{\frac{1}{4}}} >0$, one has that
\begin{equation}
    \label{d2rho1}
    \begin{split}
    2\frac{e^{h}}{(\partial_\rho P)^{\frac{1}{4}}} \partial_\rho F \partial_2 \rho \geq & -2\frac{e^{h}}{(\partial_\rho P)^{\frac{1}{4}}} |a||\eta| (\lambda_2-\lambda_1) -\frac{1}{(\partial_\rho P)^{\frac{1}{4}}} |L||\eta| \widetilde{n} (\lambda_2-\lambda_3) \\ &- \frac{1}{(\partial_\rho P)^{\frac{1}{4}}} (\lambda_2-\lambda_3) \overline{R}  - 2\frac{e^{h}}{(\partial_\rho P)^{\frac{1}{4}}}|\partial_S F| |\partial_2 S| .
    \end{split}
\end{equation}
Via the similar argument used  in the last subsection, one can check that 
\begin{equation*}
\begin{split}
    \frac{e^{h}}{(\partial_\rho P)^{\frac{1}{4}}}|a||\eta| (\lambda_2-\lambda_1) \leq & D,\quad 
    \frac{1}{(\partial_\rho P)^{\frac{1}{4}}} |L||\eta| \widetilde{n} (\lambda_2-\lambda_3) \leq  D,\\
    \frac{1}{(\partial_\rho P)^{\frac{1}{4}}} (\lambda_2-\lambda_3) \overline{R} \leq  & D,\quad 
    \frac{e^{h}}{(\partial_\rho P)^{\frac{1}{4}}} |\partial_S F| |\partial_2 S| \leq  D,
    \end{split}
\end{equation*}
which, along with \eqref{d2rho1}, implies  that
\begin{equation}
    \label{d2rho2}
    2\frac{e^{h}}{(\partial_\rho P)^{\frac{1}{4}}} \partial_\rho F \partial_2 \rho \geq -D.
\end{equation}
According to Lemma \ref{LemmaApproxhandg} and Corollary \ref{approxforPetc}, one has $\frac{e^{h}}{(\partial_\rho P)^{\frac{1}{4}}} \partial_\rho F \approx \rho^{\frac{\gamma-5}{2}}$, which, along with  \eqref{d2rho2}, yields
\begin{equation*}
    \partial_2 \left( \rho^{\frac{\gamma-3}{2}} \right) \leq D.
\end{equation*}
 Integrating the above inequality along the characteristic $x^2(t)$, one finally gets 
\begin{equation*}
    \rho^{\frac{\gamma-3}{2}} (t,x^2(t)) \leq D( t + 1).
\end{equation*}
\end{proof}

\subsection{Singularity formation with strong compression: Proof of Theorem \ref{GradientBlowUp}}

We will focus only on $q$, and  $r$ can be dealt with similarly. Due to  $q_0(x^*) < -N_2$, one has:  $\exists \  \upsilon>0$ such that 
$$q_0(x^*) \leq -N_2 \sqrt{1+\upsilon},$$
which, along with   Theorem \ref{ODEcompare}, implies that 
\begin{equation}\label{fuxing}
   q(t,x^2(t)) \leq q_0(x^*) \leq -N_2 \sqrt{1+\upsilon} \quad \forall \ t \in [0,T].
\end{equation}

According to  Theorem \ref{ODEcompare} and \eqref{fuxing}, one can obtain 
\begin{align*}
    \partial_2 q  \leq & - \frac{1}{2} \partial_z \lambda_2 e^{-g} q^2 +  \frac{1}{2} \partial_z \lambda_2 e^{-g} N_2^2\\
     \leq & - \frac{1}{2} \partial_z \lambda_2 e^{-g} q^2 +  \frac{1}{2} \partial_z \lambda_2 e^{-g} \frac{q^2}{1+\upsilon}\\
    = & - \frac{1}{2} \partial_z \lambda_2 e^{-g} q^2 \frac{\upsilon}{1+\upsilon},
\end{align*}
which yields that 
\begin{equation*}
    -\partial_2 \left(\frac{1}{q}\right) \leq - \frac{1}{2} \partial_z \lambda_2 e^{-g}  \frac{\upsilon}{1+\upsilon}.
\end{equation*}
Integrating  the above inequality along the characteristic $x^2(t)$, one gets 
\begin{equation*}
    \frac{1}{q(t,x^2(t))} \geq \frac{1}{q_0(x^*)} + \frac{\upsilon}{2(1+\upsilon)} \int_{0}^{t} \partial_z \lambda_2 e^{-g}(s,x^2(s)) \text{d}s.
\end{equation*}
As $q_0(x^*) < 0$, for  proving  this  theorem, it is sufficient to show
\begin{equation}
    \label{inttoinfty}
    \int_{0}^{t} \partial_z \lambda_2 e^{-g}(s,x^2(s)) \text{d}s \to \infty \quad \text{as} \quad t\to \infty.
\end{equation}
According to \eqref{yaoming}, Corollary \ref{approxforPetc}-\ref{e^-gdzlambda_2>0} and Lemma \ref{LemmaApproxhandg}, one has  $ \partial_z \lambda_2 e^{-g} > 0$ and $ \partial_z \lambda_2 e^{-g} \approx \rho^{\frac{3-\gamma}{4}}$. From this, for the case $\gamma \geq 3$, then
$$\rho^{\frac{3-\gamma}{4}} \geq (\sup_{(t,x) \in [0,T) \times \mathbb{R}} \rho)^{\frac{3-\gamma}{4}},$$ 
so \eqref{inttoinfty} will be satisfied. Otherwise, for the case  $1<\gamma<3$, by Lemma \ref{DensityLowerBound}, \eqref{inttoinfty} is also satisfied. Hence Theorem \ref{GradientBlowUp} is proved.

\begin{appendices}

\section{On the equation of state for the polytropic gas}

This subsection will be devoted to showing  some properties on the equation of state for the polytropic gas.

\begin{proposition} \label{AProp_nandrho}
$\rho$, $n$ and $S$ satisfy the relation
\begin{equation}
    \label{nrhoS}
    n^\gamma \exp\left(\frac{S}{C_v}\right) + c^2(n-\rho) = 0.
\end{equation}
Moreover, if $|S|\leq C$ for some positive constant $C$,   one has 
\begin{equation}
    \label{rhoto0}
    n(\rho,S) \sim \mathcal{O}(\rho) \quad \text{as} \quad \rho \to 0.
\end{equation}
\end{proposition}

\begin{proof}
First,  \eqref{nrhoS} can be obtained directly from  \eqref{PnS} and \eqref{formula-5}. 

Second, denote 
\begin{equation*}
    G(n,\rho,S) \coloneqq n^\gamma \exp\left(\frac{S}{C_v}\right) + c^2(n-\rho).
\end{equation*}
Due to  $n \geq 0$, a simple calculation will yield that  $\frac{\partial G}{\partial n} >0$.
Hence by the implicit function theorem, one can  express $n$ as a function of $\rho$ and $S$.

 Since  $n \geq 0$ and $\rho \geq 0$, it follows from \eqref{nrho} that  $\rho = 0\Longleftrightarrow n = 0$.  We show  $\rho$ and $n$ tends to $0$ at the same decay  rate. When $n >0$, we can divide \eqref{nrhoS} by $n$ and obtain 
\begin{equation*}
  \frac{\rho}{n}=1+  \frac{1}{c^2}n^{\gamma-1} \exp\left(\frac{S}{C_v}\right).
\end{equation*}
Since 
$
n(\rho,S) \rightarrow 0 \ \text{as} \ \rho \rightarrow 0
$,
one has 
\begin{equation*}
  \lim_{\rho\rightarrow 0}\frac{\rho}{n(\rho,S)}=1,
\end{equation*}
which implies  \eqref{rhoto0}.
\end{proof}

\begin{proposition} \label{AProp_drhoPetc}
If the pressure $P$ is given in terms of $(\rho,S)$, i.e., $P=P(\rho,S)$, then 
\begin{equation}
\begin{split}
    \label{drhoP}
    \partial_\rho P =& \frac{c^2 \gamma(\gamma-1)n^{\gamma-1} \exp\left(\frac{S}{C_v}\right)}{\gamma n^{\gamma-1} \exp\left(\frac{S}{C_v}\right)+c^2},\\[1mm]
    \partial_S P =&\frac{1}{C_v}\frac{c^2 (\gamma-1)n^{\gamma} \exp\left(\frac{S}{C_v}\right)}{\gamma n^{\gamma-1} \exp\left(\frac{S}{C_v}\right)+c^2},\\[1mm]
    \partial_{\rho S} P =& \frac{c^4 \gamma  (\gamma-1)n^{\gamma-1} \exp\left(\frac{S}{C_v}\right)\left(c^2 + n^{\gamma-1}\exp\left(\frac{S}{C_v}\right) \right)}{C_v\left(\gamma n^{\gamma-1} \exp\left(\frac{S}{C_v}\right)+c^2 \right)^3},\\[1mm]
    \partial_{\rho \rho} P =& \frac{c^6 \gamma  (\gamma-1)^2n^{\gamma-2} \exp\left(\frac{S}{C_v}\right)}{\left(\gamma n^{\gamma-1} \exp\left(\frac{S}{C_v}\right)+c^2 \right)^3}.
    \end{split}
\end{equation}
\end{proposition}

\begin{proof}
We consider the relation $n=n(\rho,S)$.  It follows from $G(n,\rho,S)=0$ that 
\begin{equation}\label{tianrui-1}
\begin{split}
    \partial_\rho n =&   \frac{c^2}{\gamma n^{\gamma-1} \exp\left(\frac{S}{C_v}\right)+c^2},\quad 
    \partial_S n  = \frac{-n^\gamma \exp\left(\frac{S}{C_v}\right)}{C_v\left(\gamma n^{\gamma-1} \exp\left(\frac{S}{C_v}\right)+c^2\right)}.
    \end{split}
\end{equation}

     Denote  $P(\rho,S) = \widetilde{P}\left(n(\rho,S),S\right)$. By the chain rule, one has 
\begin{equation}
    \partial_\rho P = \partial_n \widetilde{P} \partial_\rho n, \quad 
    \partial_S P = \partial_n \widetilde{P} \partial_S n + \partial_S \widetilde{P},
\end{equation}
which, along with \eqref{PnS}  and \eqref{tianrui-1}, yields $\eqref{drhoP}_1-\eqref{drhoP}_2$.

For the second derivatives, by the chain rule, one has 
\begin{equation}\label{chairui}
    \frac{\partial^2 P}{\partial \rho \partial S} = \frac{\partial^2 \widetilde{P}}{\partial n \partial S} \frac{\partial n}{\partial \rho} + \frac{\partial^2 n}{\partial \rho \partial S} \frac{\partial \widetilde{P}}{\partial n} + \frac{\partial^2 \widetilde{P}}{\partial n^2} \frac{\partial n}{\partial \rho} \frac{\partial n}{\partial S}.
\end{equation}
It follows from \eqref{PnS} that 
\begin{equation*}
\begin{split}
    \frac{\partial^2 \widetilde{P}}{\partial n^2} =& \gamma(\gamma-1)^2 n^{\gamma-2} \exp\left(\frac{S}{C_v}\right),\\
    \frac{\partial^2 \widetilde{P}}{\partial n \partial S} =& \gamma(\gamma-1) \frac{1}{C_v} n^{\gamma-1} \exp\left(\frac{S}{C_v}\right),\\
    \frac{\partial^2 n}{\partial \rho \partial S} =& 
    - \frac{c^2 \gamma n^{\gamma-1} \exp\left(\frac{S}{C_v}\right)\left(n^{\gamma-1} \exp\left(\frac{S}{C_v}\right) +c^2\right)}{C_v\left(\gamma n^{\gamma-1} \exp\left(\frac{S}{C_v}\right)+c^2 \right)^3},
    \end{split}
\end{equation*}
which, along with \eqref{chairui}, yields $\eqref{drhoP}_3$. Similarly, one can obtain $\eqref{drhoP}_4$.
\end{proof}

\begin{proposition} \label{AProp_propertiesofP}
\begin{equation}
    \label{positiveP}
    \begin{split}
    P(\rho,S) \geq & 0, \quad P(\rho,S) = 0 \iff \rho = 0,\\
    \partial_\rho P(\rho,S) \geq & 0, \quad \partial_\rho P(\rho,S) = 0 \iff \rho = 0.
    \end{split}
\end{equation}
\end{proposition}
\begin{proof}

  $\eqref{positiveP}_1$ follows from \eqref{PnS}  and the fact that $n = 0$ if and only if $\rho = 0$. 
$\eqref{positiveP}_2$ follows from  \eqref{drhoP} and the fact that $n = 0$ if and only if $\rho = 0$. 
\end{proof}

It follows from Propositions    \ref{AProp_nandrho} and  \ref{AProp_drhoPetc}  that 
\begin{proposition}\label{AProp_OrderOfPetc}
\begin{equation}
    \label{Pto0}
    \begin{split}
    P(\rho,S) & \sim \mathcal{O}(\rho^\gamma) \quad \text{as} \quad \rho \to 0;\ \  
    \partial_\rho P(\rho,S)  \sim \mathcal{O}(\rho^{\gamma-1}) \quad \text{as} \quad \rho \to 0;\\
    \partial_S P(\rho,S) & \sim \mathcal{O}(\rho^\gamma) \quad \text{as} \quad \rho \to 0; \ \ 
    \partial_{S \rho} P(\rho,S)  \sim \mathcal{O}(\rho^{\gamma-1}) \quad \text{as} \quad \rho \to 0.
    \end{split}
\end{equation}
\end{proposition}

\section{A property of continuous functions}
We will give one property of continuous functions in this subsection  that will be used in our proof. 
\begin{proposition}\label{B_SplitContinuousFunctions}
Let $f(x)$ be a continuous function on $x \in [a,b]$, where $a$ and $b$ are two fixed numbers. Then we can split the interval $[a,b]$ into countably many sub-intervals
\begin{equation}
    \label{interval}
    [a,b] = \bigcup_{i = 0}^{\infty} [x_i, x_{i+1}],
\end{equation}
where $x_0 = a$ and $x_\infty = b$, $(x_i,x_{i+1})$ and $(x_j,x_{j+1})$ are disjoint if $i \neq j$, such that on each of the $(x_i,x_{i+1})$ we have either $f(x) > 0$, $f(x) < 0$ or $f(x) = 0$. 
\end{proposition}
\begin{proof} It is obvious that we only need to consider the case  that the sign of $f(x)$ will change in the whole interval $[a,b]$.

 First, for  any point $x_0\in (a,b)$ such that $f(x_0) \neq 0$,
\begin{enumerate}

\item   if $x_0=a$,  there is an interval $[a,x_b]$ ($a<x_b< b$) such that $f(x) \neq 0$ for $x\in [a,x_b)$, and $f(x_b)=0$;\\

\item  if $x_0=b$,  there is an interval $[x_a,b]$ ($a< x_a<b$) such that $f(x) \neq 0$ for $x\in (x_a,b]$, and $f(x_a)=0$;\\

\item if $x_0\in [a,b]$, one has three cases:

 \begin{enumerate}

\item   there is an interval $x_0\in [a,x_b]$ ($a<x_b< b$) such that $f(x) \neq 0$ for $x\in (a,x_b)$;\\
\item    there is an interval $x_0\in [x_a,b]$ ($a< x_a<b$) such that $f(x) \neq 0$ for $x\in (x_a,b)$;\\
\item    there is an interval $x_0\in [x_a,x_b]$ ($a< x_a<x_b< b$) such that $f(x) \neq 0$ for $x\in (x_a,x_b)$, and $f(x_a) =f(x_b)=0$,
\end{enumerate}

\end{enumerate}
where one has used the continuity of  the function $f$.
 
Second,  consider the set of all rational numbers $\textbf{Q}\cap [a,b]=\{q_i\}_{i=0}^{\infty}$. We denote $(y_i,y_{i+1})$ the interval satisfying:
 \begin{enumerate}

\item   there at least exists one real number   $q_i \in (y_i,y_{i+1})$ for some $i \in \mathbb{N}$;
\item    $f(x) \neq 0$ for $x\in (y_i,y_{i+1})$.
\end{enumerate}
Because  $\textbf{Q}\cap [a,b]$ is  dense in $[a,b]$, such kind of intervals indeed exists. Denote
\begin{equation*}
    O \coloneqq (a,b) \setminus \Big(\bigcup_{i = 0}^{\infty} [y_i, y_{i+1}] \Big).
\end{equation*}
Note $O$ is an open set as it is a open set minus a closed set, and $f \equiv 0$ on $O$ by construction. As $O$ is open, one can express it as countably many disjoint open intervals (see \cite{realanalysis}, section 1.4, Proposition 9), hence one has 
\begin{equation*}
    O = \bigcup_{i = 0}^{\infty} (z_i, z_{i+1})
\end{equation*}
Now note 
\begin{equation*}
    [a,b] = \Big(\bigcup_{i = 0}^{\infty} [y_i, y_{i+1}] \Big) \bigcup \overline{O} = 
    \Big(\bigcup_{i = 0}^{\infty} [y_i, y_{i+1}] \Big) \bigcup \Big(\bigcup_{j = 0}^{\infty} [z_j, z_{j+1}] \Big)
\end{equation*}
After some relabelling we will get \eqref{interval}, and all the properties stated in the proposition are satisfied by construction.
\end{proof}

\section{The explicit derivation of weight functions:  $h_1$, $h_2$, $h$ and $g$}\label{apphg}


This subsection is divided into four steps, where we will show that the functions  $(h_1,h_2)$ given in  \eqref{fgmm} satisfy \eqref{h1h2def}, and $(h,g)$ given in  \eqref{tidleAlphaBeta} satisfy \eqref{dzhAnddwg}.

Step 1:  $h_1$ satisfies \eqref{h1h2def}. First, according to  the formulas of $(\lambda_1,\lambda_2)$ in \eqref{h1h2def}, one has 
\begin{equation*}
    \begin{split}
        \partial_z \lambda_1 = \frac{2c \partial_z (e^{H_{1}})}{(e^{H_{1}} + 1)^2}, \quad 
        \lambda_1 - \lambda_2 =&  -2c \left( \frac{e^{H_{2}} - e^{H_{1}}}{(e^{H_{2}}+1)(e^{H_{1}}+1)} \right),
    \end{split}
\end{equation*}
which yields that 
\begin{equation}
    \label{C1}
    \begin{split}
        \frac{\partial_z \lambda_1}{\lambda_1 - \lambda_2 } = &
        - \frac{\partial_z (e^{H_{1}}) (e^{H_{2}}+1)}{(e^{H_{1}}+1)(e^{H_{2}} - e^{H_{1}})} 
        =
        - \left( \frac{1}{e^{H_{1}}+1} + \frac{1}{e^{H_{2}} - e^{H_{1}}} \right) \partial_z (e^{H_{1}})\\
        =& - \partial_z \ln (e^{H_{1}}+1) - \frac{\partial_z (e^{H_{1}})}{e^{H_{2}} - e^{H_{1}}}.
    \end{split}
\end{equation}

Second, it follows from   \eqref{fg}, that  $H_{2} + H_{1} = \frac{2(w+z)}{c}$. Then 
\begin{equation*}
    \begin{split}
    \partial_z \ln (e^{H_{2}+H_{1}} - e^{2H_{1}}) &=
    \frac{1}{e^{H_{2}+H_{1}} - e^{2H_{1}}} \left( \frac{2}{c} e^{\frac{2(w+z)}{c}} - 2 e^{H_{1}} \partial_z (e^{H_{1}}) \right)\\ &= 
    \frac{1}{e^{H_{2}} - e^{H_{1}}} \left( \frac{2}{c} e^{H_{2}} - 2 \partial_z (e^{H_{1}}) \right),
    \end{split}
\end{equation*}
which implies that 
\begin{equation}
    \label{C2}
    \begin{split}
        \frac{\partial_z (e^{H_{1}})}{e^{H_{2}} - e^{H_{1}}} &= 
        \frac{e^{H_{2}}}{c(e^{H_{2}} - e^{H_{1}})} - \frac{1}{2} \partial_z \ln (e^{H_{2}+H_{1}} - e^{2H_{1}}) \\
        &= \frac{1}{c(1 - e^{H_{1}-H_{2}})} - \frac{1}{2} \partial_z \ln (e^{H_{2}+H_{1}} - e^{2H_{1}}).
    \end{split}
\end{equation}
According to  \eqref{fg}, one can obtain  that 
\begin{equation*}
    1 - e^{H_{1}-H_{2}} = 1 - \left( \frac{1-\sqrt{\gamma}\hspace{1mm} \text{Tan}(Y)}{1+\sqrt{\gamma}\hspace{1mm} \text{Tan}(Y)} \right)^2 = \frac{4\sqrt{\gamma}\hspace{1mm} \text{Tan}(Y)}{(1+\sqrt{\gamma}\hspace{1mm} \text{Tan}(Y))^2},
\end{equation*}
which, along with  \eqref{C2}, provides that 
\begin{equation}
    \label{C3}
    \frac{\partial_z (e^{H_{1}})}{e^{H_{2}} - e^{H_{1}}} = \frac{(1+\sqrt{\gamma}\hspace{1mm} \text{Tan}(Y))^2}{4c\sqrt{\gamma}\hspace{1mm} \text{Tan}(Y)}- \frac{1}{2} \partial_z \ln (e^{H_{2}+H_{1}} - e^{2H_{1}}).
\end{equation}
Then it follows from   \eqref{C1} and \eqref{C3} that 
\begin{equation}
    \label{C4}
    \frac{\partial_z \lambda_1}{\lambda_1 - \lambda_2 } = - \partial_z \ln (e^{H_{1}}+1) + \frac{1}{2} \partial_z \ln (e^{H_{2}+H_{1}} - e^{2H_{1}}) - \frac{(1+\sqrt{\gamma}\hspace{1mm} \text{Tan}(Y))^2}{4c\sqrt{\gamma}\hspace{1mm} \text{Tan}(Y)}.
\end{equation}

Moreover, according to  \eqref{fg} and the fact $\frac{dz}{dY} = \frac{4c\sqrt{\gamma}}{\gamma -1}$, one has 
\begin{equation*}
    \begin{split}
        \partial_z \ln(e^{H_1} + 1) &= \partial_z \ln \left(e^{\frac{w+z}{c}}\left(\frac{1-\sqrt{\gamma}\hspace{1mm} \text{Tan}(Y)}{1+ \sqrt{\gamma}\hspace{1mm} \text{Tan}(Y)} \right) + 1 \right) \\
        &= \partial_z \ln \left( \frac{(1+e^{\frac{w+z}{c}}) - \sqrt{\gamma} (-1 + e^{\frac{w+z}{c}})\text{Tan}(Y)}{1 + \sqrt{\gamma}\hspace{1mm} \text{Tan}(Y)} \right) \\
        &= \partial_z \ln \left( (1+e^{\frac{w+z}{c}}) \text{Cos}(Y) - \sqrt{\gamma} (-1 + e^{\frac{w+z}{c}}) \text{Sin}(Y) \right) \\
        & \quad - \partial_z\ln(\text{Cos}(Y)) - \partial_z \ln(1 + \sqrt{\gamma}\hspace{1mm} \text{Tan}(Y)),\\
        \frac{1}{2}\partial_z \ln(e^{H_1+H_2} - e^{2H_1}) &= 
        \frac{1}{2} \partial_z \ln \left(e^{2(\frac{w+z}{c})} - e^{2(\frac{w+z}{c})} \left(\frac{1-\sqrt{\gamma}\hspace{1mm} \text{Tan}(Y)}{1+ \sqrt{\gamma}\hspace{1mm} \text{Tan}(Y)} \right)^2 \right) \\
        &= \frac{1}{2} \partial_z \left( \ln (e^{2(\frac{w+z}{c})}) + \ln\left(\frac{4\sqrt{\gamma}\hspace{1mm} \text{Tan}(Y)}{ (1+ \sqrt{\gamma}\hspace{1mm} \text{Tan}(Y))^2}\right) \right) \\
        &= \frac{1}{c} + \frac{1}{2} \partial_z \ln (4\sqrt{\gamma}\hspace{1mm} \text{Tan}(Y)) - \partial_z \ln (1+ \sqrt{\gamma}\hspace{1mm} \text{Tan}(Y))\\
        &= \frac{1}{c} + \frac{1}{2} \partial_z \ln (\text{Tan}(Y)) - \partial_z \ln (1+ \sqrt{\gamma}\hspace{1mm} \text{Tan}(Y))\\
        &= \frac{1}{c} + \frac{1}{2} \partial_z \ln (\text{Sin}(Y)) - \frac{1}{2} \partial_z \ln (\text{Cos}(Y)) - \partial_z \ln (1+ \sqrt{\gamma}\hspace{1mm} \text{Tan}(Y)),\\
        \frac{(1+\sqrt{\gamma}\hspace{1mm} \text{Tan}(Y))^2}{4c\sqrt{\gamma}\hspace{1mm} \text{Tan}(Y)} &= 
        \frac{1}{4c\sqrt{\gamma}\hspace{1mm} \text{Tan}(Y)} + \frac{1}{2c} + \frac{\sqrt{\gamma}\hspace{1mm} \text{Tan}(Y)}{4c} \\
        &= \frac{1}{4c\sqrt{\gamma}}\partial_Y \ln(\text{Sin}(Y)) + \frac{1}{2c} - \frac{\sqrt{\gamma}}{4c} \partial_Y \ln(\text{Cos}(Y))\\
        &= \frac{1}{\gamma-1} \partial_z \ln(\text{Sin}(Y)) + \frac{1}{2c} - \frac{\gamma}{\gamma -1} \partial_z \ln(\text{Cos}(Y)),
        \end{split}
\end{equation*}
which, along with  \eqref{C4}, yields  that 
\begin{equation*}
    \begin{split}
           \frac{\partial_z \lambda_1}{\lambda_1 - \lambda_2 }  = &- \partial_z \ln \left( (1+e^{\frac{w+z}{c}}) \text{Cos}(Y) - \sqrt{\gamma} (-1 + e^{\frac{w+z}{c}}) \text{Sin}(Y) \right) \\ & + \frac{\gamma -3}{2(\gamma-1)} \partial_z \ln(\text{Sin}(Y)) +  \frac{3\gamma-1}{2(\gamma -1)} \partial_z \ln(\text{Cos}(Y)) + \frac{1}{2c}.
    \end{split}
\end{equation*}

From this one can easily see  that the $h_1$ given by \eqref{fgmm} indeed satisfies \eqref{h1h2def}.

Step 2:  $h_2$ satisfies \eqref{h1h2def}. First, via the similar argument used in Step 1, one can verify that
\begin{equation}
    \label{C5}
    \frac{\partial_w \lambda_2}{\lambda_1 - \lambda_2 } = - \partial_w \ln (e^{H_{2}}+1) + \frac{1}{2} \partial_w \ln (e^{2H_{2}} - e^{H_{1} + H_{2}}) + \frac{(1-\sqrt{\gamma}\hspace{1mm} \text{Tan}(Y))^2}{4c\sqrt{\gamma}\hspace{1mm} \text{Tan}(Y)}.
\end{equation}
Second, it follows from   \eqref{fg}, the fact $\frac{dw}{dY} = \frac{4c\sqrt{\gamma}}{1-\gamma }$ and direct calculations  that the function  $h_2$ given in  \eqref{fgmm} satisfies \eqref{h1h2def}.
\end{appendices}

Step 3:  $h$ satisfies \eqref{dzhAnddwg}. According to \eqref{HandG} and \eqref{lambdas}, one can verify that
\begin{equation}
    \frac{\partial_z \lambda_3}{\lambda_3 - \lambda_2 } = - \partial_z \ln (e^{H}+1) + \frac{1}{2} \partial_z \ln (e^{H+G} - e^{2H}) - \frac{(c+\sqrt{\Lambda})}{4c^2\sqrt{\Lambda}}.
\end{equation}
It is easy to see  that the $h$ defined in \eqref{tidleAlphaBeta} indeed satisfies \eqref{dzhAnddwg}.

Step 4:  $g$ satisfies \eqref{dzhAnddwg}. According to \eqref{HandG} and \eqref{lambdas}, one can verify that
\begin{equation}
    \frac{\partial_w \lambda_2}{\lambda_3 - \lambda_2 } = - \partial_w \ln (e^{G}+1) + \frac{1}{2} \partial_w \ln (e^{2G} - e^{G+H}) + \frac{(c-\sqrt{\Lambda})}{4c^2\sqrt{\Lambda}}.
\end{equation}
It is easy to see  that the $g$ defined in \eqref{tidleAlphaBeta} indeed satisfies \eqref{dzhAnddwg}.

\bigskip

\noindent{\bf Acknowledgement:}
The research  was supported in part by the Royal Society-- Newton International Fellowships NF170015, Newton International Fellowships Alumni AL/201021, and 
the UK Engineering and Physical Sciences Research Council Award
EP/L015811/1.

\bigskip

\noindent{\bf Conflict of Interest:} The authors declare  that they have no conflict of
interest.  The authors also  declare that this manuscript has not been previously  published, and will not be submitted elsewhere before your decision.

\end{document}